\documentclass[12pt,leqno]{amsart}
\usepackage[all]{xy}
\usepackage{amssymb,mathrsfs,euscript,enumitem,marginnote}
\usepackage[usenames,dvipsnames]{pstricks}
\usepackage{amsmath,bbold}
\usepackage{fourier}
\usepackage[hypertex]{hyperref}
\textheight9in  \def\DATE{December 16, 2021}
\hoffset - 2cm  \hbadness=100000
\newtheorem{theorem}{Theorem}[section]
\newtheorem{corollary}[theorem]{Corollary}

\newtheorem{lemma}[theorem]{Lemma}
\newtheorem{proposition}[theorem]{Proposition}

\theoremstyle{definition}
\newtheorem{example}[theorem]{Example}
\newtheorem{remark}[theorem]{Remark}
\newtheorem{definition}[theorem]{Definition}

\def\MA{{\it MA}}
 
\def\oJac{\hbox{$\EuScript J\! ac$}}
\def\oA{{\EuScript A}}
\def\oB{{\EuScript B}}
\def\ol{{\ell}}
\def\colim{\mathop{\rm colim}\displaylimits}
\def\jednas{(1,\ldots,1)}
\def\oF{{\overline F}}
\def\arity{\relax}
\def\oG{{\overline G}}

\def\pres#1{#1'}
\def\oLLL{\mathcal{L}_\infty}

\def\lev{{[\![}}
\def\prav{{]\!]}}
\def\L#1#2#3{{\big(#1,(#2,#3)\big)}}
\def\R#1#2#3{{\big((#1,#2),#3\big)}}
\def\X{{\big(1,(2,3)\big)}}\def\Y{{\big(2,(3,1)\big)}}
\def\Z{{\big(3,(1,2)\big)}}
\def\A{{\big[1,[2,3]\big]}}\def\B{{\big[2,[3,1]\big]}}
\def\C{{\big[3,[1,2]\big]}}

\def\Cat{{\tt Cat}}
\def\MZ{{M\bbZ}}
\def\coll#1{\{#1\}_{n \geq 1}}

\def\Filt{{\rm Filt}}
\def\MFilt{{\rm MFilt}}
\def\SMFilt{{\rm SMFilt}}
\def\aA{{\EuScript A}}
\def\bB{{\EuScript B}}
\def\filt{{\mathscr F}}
\def\bbZ{{\mathbb Z}}
\def\Span{{\rm Span}}
\def\Sh#1{{\it Sh}_{#1}}
\def\Sub{{\tt Sub}}
\def\Ass{\EuScript{A}{\it ss}}
\def\Poiss{\EuScript{P}{\it oiss}}
\def\DG{\hbox{$\EuScript{D}\hskip -.1em  \it g$}}
\def\Lie{\hbox{$\mathcal{L}{\hskip -.16em\it ie}$}}
\def\Com{\hbox{$\mathcal{C}{\hskip -.16em\it om}$}}
\def\LieAdm{\hbox{$\mathcal{L}{\hskip -.16em\it ie\,}\mathcal{A}{\hskip
      -.16em\it dm}$}} 
\def\rmAss{{\rm Ass}}
\def\rmLieAdm{{\rm LieAdm}}

\def\Free{{\mathbb F}(E)}
\def\Diff{\hbox{${\mathcal D}\hskip -.06em {\it iff}$\hskip -.2em }_A}

\def\End{\hbox{${\mathcal E}\hskip -.1em {\it nd}$}}
\def\oDer{\hbox{${\mathcal D}\hskip -.1em {\it er \hskip -.1em }_A$}}
\def\tnabla{{\widetilde \nabla}}
\def\tu{{\tilde 1}}
\def\tA{{\widetilde A}}

\def\redukcem#1{\vbox to .3em{\vss\hbox{$#1$}}}
\def\Hom{{\rm Lin}_\bbk}
\def\Der#1{{\rm Der}^{\, #1}(A)}

\def\tDer#1{{\rm Der}^{\, #1}(\tA)}
\def\Dif#1{{\rm Diff}^{\, #1}(A)}
\def\DiF#1{{\rm Diff}^{\, #1}}
\def\dif#1{\overline{\rm Diff}^{\, #1}(A)}
\def\Assoc{{\tt AssocCom}}
\def\Linfty{{\tt Lie}_\infty}
\def\oLinfty{{\tt operLie}_\infty}
\def\Asss{{\rm Ass}}

\def\rada#1#2{#1,\ldots,#2}
\def\oP{{\EuScript P}}
\def\oS{{\EuScript S}}
\def\oQ{{\EuScript Q}}
\def\gbiLie{{\mathfrak g}_{{\it biLie}}(V)}
\def\tgbiLie{{\widehat {\mathfrak g}}_{{\it biLie}}(V)}
\def\gIBL{{\mathfrak g}_{{\it IBL}}(V)}
\def\tgIBL{{\widehat {\mathfrak g}}_{{\it IBL}}(V)}
\def\tS{{\widehat \S}}

\def\susp{\uparrow \hskip -.3em}
\def\desusp{\downarrow \hskip -.3em}
\def\ext{\mbox{\large$\land$}}

\def\sign#1{{(-1)^{#1}}}
\def\Comut#1{{\Psi^{#1}_\nabla}}
\def\PD#1{\Phi^{#1}_\nabla}
\def\tPD#1{\Phi^{#1}_\tnabla}
\def\FD#1{\Psi^{#1}_\nabla}
\def\sgn{\mbox{\rm sgn}}
\def\Rada#1#2#3{{#1}_{#2},\ldots,{#1}_{#3}}
\def\LLL{{L_\infty}}
\def\S{{\mathbb S}}
\def\Jac{{\rm Jac}}
\def\bbk{{\mathbb k}}

\def\BaVi{{\EuScript B\EuScript V}}
\def\BV{{\it BV}}
\def\IBL{{\it IBL}}

\def\bfk{{\mathbb k}}
\def\id{{\mathbb 1}}
\def\Lin{{\rm Lin}_\bbk}
\def\shade{\hbox{${\EuScript L} {\it ie}^\diamond {\EuScript B}$}}
\def\shadew{\hbox{${\EuScript L} {\it ie}{\EuScript B}$}}
\def\ot{\otimes}
\def\squeezedldots{{\mbox{$. \hskip -.8pt . \hskip -.8pt .$}}}
\def\sleft{\hbox {$\{ \hskip -.25em \{  \hskip -.25em   \{$} }
\def\sright{\hbox {$\} \hskip -.25em \}  \hskip -.25em   \}$}}

\catcode`\@=11

\def\@evenfoot{\rule{0pt}{20pt}[\DATE] \hfill [{\tt \jobname.tex}]}
\def\@oddfoot{\rule{0pt}{20pt}{[\tt \jobname.tex}]\hfill [\DATE]}
\@namedef{subjclassname@2010}{%
  \textup{2020} Mathematics Subject Classification}
\catcode`\@=13

\newcommand{\rev}[2]{{#2}}
\newcommand{\revnm}[2]{{#2}}


\makeatletter
\providecommand\@dotsep{5}
\def\listtodoname{List of Todos}
\def\listoftodos{\@starttoc{tdo}\listtodoname}
\makeatother

\title[Calculus of multidifferential operators and operator algebras]
{Calculus of multilinear differential operators, operator $\LLL$-algebras and $\IBL_\infty$-algebras}

\thanks{The second author supported by grant GA \v CR
  18-07776S. Both authors supported by RVO: 67985840 and Praemium Academi\ae {} of M.~Markl.}

\begin{document}
\bibliographystyle{plain}

\parskip3pt plus 1pt minus .5pt
\baselineskip 18pt  plus 1.5pt minus .5pt

\author{Denis Bashkirov}

\email{bashkirov@math.cas.cz,\ markl@math.cas.cz}

\author{Martin Markl}
\address{The Czech Academy of Sciences, Institute of Mathematics, {\v Z}itn{\'a} 25,
         115 67 Prague 1, The Czech Republic}

\hyphenation{multifilt-ra-tion}

\begin{abstract}
\rev{1}
{%
 We propose an operadic framework suitable for describing algebraic structures with operations being multilinear differential operators of varying orders or, more generally, formal series of such operators.
 The framework is built upon the notion of a multifiltration of a linear operad generalizing the concept of a filtration of an associative algebra.  
 We describe a particular way of constructing and analyzing multifiltrations based on a presentation of a linear operad in terms of generators and relations. In particular, that allows us to observe a special role played in this context by Lie, Lie-admissible and $\Linfty$-structures.
 As a main application, and the original motivation for the present work, we show how a certain generalization of the well-known big bracket construction of Lecomte\textendash Roger and Kosmann-Schwarzbach encompassing the case of homotopy involutive Lie bialgebras can be obtained.
}

\end{abstract}


\maketitle

\setcounter{tocdepth}{1}
\tableofcontents

\section*{Introduction}
\paragraph{\textbf{Motivation.}}
The deformation complex of a Lie bialgebra $V$ admits a particularly
simple description in terms of a certain graded Lie structure, known
in the literature as the \emph{big bracket}, supported on the shifted
Grassmann algebra 
$\ext^{\bullet}(V \oplus
V^*)[2]$, cf.~\cite{kosmann-schwarzbach:CRAS91,Kos,Krav,lecomte1990modules}.  
Geometrically, the
latter can be identified with an odd Poisson bracket on the shifted
cotangent bundle $T^*V[-1]$ and thus can be constructed from as little
input data as a pairing between $V$ and $V^*$.  One of the points of
our interest for the present work is an analogous construction for a
particular class of Lie bialgebras, namely \emph{involutive Lie
  bialgebras}, characterized by an additional property of the form
$[-,-]\circ \delta = 0$, where $[-,-]:V\otimes V\to V$ and
$\delta: V\to V\otimes V$ are the bracket and the cobracket of a Lie
bialgebra $V$ respectively.  Some of the most notable examples of such
algebras include the Goldman\textendash Turaev Lie bialgebra on the
vector space of non-trivial free homotopy classes of loops on an
oriented surface \cite{Chas} and, more generally, the Chas\textendash
Sullivan Lie bialgebra on the string homology of a~compact oriented
manifold \cite{ChasSul}. Furthermore, as shown by K.~Cieliebak and
J.~Latschev \cite{CielLat}, the linearized homology of an augmented
strongly homotopy Batalin\textendash Vilkovisky algebra 
(or a~\emph{BV$_{\infty}$-algebra}), which is free as a strictly
commutative associative
algebra, comes equipped with an involutive Lie bialgebra structure. As
an application, this gives rise to an involutive Lie bialgebra
structure on the linearized contact homology of a closed contact
manifold with respect to an exact symplectic filling.

The problem of devising an appropriate homotopy counterpart of
involutive Lie bialgebras arises in the context of string field
theory. While the classical (genus zero) open-closed string field
theory is encoded by a certain splice of $A_\infty$ and
$\LLL$-algebras, comprising what is commonly known as the open-closed
homotopy algebra \cite{KajStash}, an enhanced structure - that of a
\emph{homotopy involutive Lie bialgebra} (or $\IBL_ \infty$-algebra) -
is needed to set up the full BV master equation in the quantum case
(arbitrary genus).  As per general theory, constructing such a
homotopy algebraic structure involves building a minimal resolution
for the PROP of involutive Lie bialgebras. While that was accomplished
by R.~Campos, S.~Merkulov and T.~Willwacher \cite{CMW}, our motivation
for the present work was to elucidate on a Lie-algebraic structure,
akin to the big bracket, present on the corresponding deformation
complex.

\rev{2}
{%
As it turns out, constructing such an analog of the big bracket would require bypassing a certain no-go result concerning
differential-operator properties of Lie brackets. Namely, the
well-known results of A.~Kirillov~\cite{Kir} and 
J.~Grabowski~\cite{Grab} state that if $A$ is an algebra of smooth functions on a smooth manifold or, more generally, a reduced commutative ring $A$, and $[-,-]:A\otimes A\to A$ is a Lie bracket that happens to be a differential operator of order $n<\infty$ with respect to each of the arguments, then the order $n$ cannot exceed $1$, yielding in the case of $n=1$ the classical cases of Poisson and Jacobi structures. 
}

One may attempt to overcome this by introducing nilpotents (thus considering infinitesimal deformations of Lie algebras), by relaxing the antisymmetry condition on the bracket (thus arriving to Leibniz-type algebras, cf.~Example~\ref{KanEx}) or, as we undertake in the present work, by considering Lie brackets comprised as formal infinite series of bilinear differential operators of ongoingly increasing order,
\[
 [-, -]= [-, -]_1 + [-, -]_2 h + [-, -]_3 h^2 + \cdots,
\]
where $h$ is a formal parameter.

Formalizing the latter concept 
has lead us to a more general notion of
a formal multilinear differential operator algebra, which in turn required
extrapolating some basic notions of the differential calculus and
{\it D}-modules from the realm of associative algebras to the case of
operads. 
\revnm{added}{In a sense, this may be regarded as an approach to the general notion of a deformation quantization of an algebra over a linear operad.}

\paragraph{\textbf{Multifiltrations and the differential calculus for operads.}}
As we recall, differential operators
on a commutative associative $\bbk$-algebra $A$, where $\bbk$ is a
field, 
are defined in terms of
a filtration
\[
 \Dif{-1}:=0 \subset \Dif{0} \subset \Dif{1} \subset \dots \subset
 {\it End}(A)
\]
of the linear endomorphism algebra ${\it End_\bbk}(A)$ compatible with the
standard commutator brack\-et in the sense that
\[
[
\Dif k, \Dif l]\subset \Dif {k+l-1},\ k,l \geq 0.
\]  
A proposed notion of a
\emph{multifiltration} of a linear operad (cf.~Definition~\ref{MFdef}) 
is meant to provide an $n$-ary analog of this
concept. Specifically, it is defined in terms of a poset of
$\bbk$-linear subspaces of a given $\bbk$-linear operad $\oP$
controlled by the combinatorics of integer-valued multiindices
reflecting, in our case, the differential-operator orders of the
individual inputs of a~$\bbk$-linear mapping
$O:A\otimes\dots\otimes A\to A$. The corresponding combinatorial data
is encoded by a~certain poset-valued operad $\MZ$, similarly to how
$\mathbb{Z}$-graded filtrations of associative algebras are defined in
terms of the ordered monoid $(\mathbb{Z}, +)$. It is worth noting that in this
generalized $n$-ary setting the commutator bracket gets replaced by a
double-indexed family of operadic commutators $[-,-]_{ij}$. A
multilinear differential \emph{operator algebra} is then defined, just
as in case of an ordinary algebra over an operad, in terms of a
structure morphism into the endomorphism operad $\End_A$ or
$\End_{A\lev h \prav}$, where $h$ is a formal parameter, but this time
both come equipped with some extra data in the form of differential
operator multifiltrations.

Due to the specifics of the original problem, we pay particular
attention to the case of \emph{operator $\LLL$-algebras}, the latter
being $\LLL$-algebras supported on (graded) commutative associative algebras and
with the structure operations representable as formal infinite series
of differential operators of certain orders. The big bracket and the
$\IBL_\infty$-algebras arise as particular examples of such
algebras. We 
generalize the former by introducing the \emph{superbig} bracket. The
term is meant to indicate that it contains the big bracket while, as noticed by
Y.~Kosmann-Schwarzbach, the big bracket itself gives rise to several
simpler brackets relevant for deformation theory. The interest is further reinforced by Theorem \ref{Vymenim
  to kolo?} that singles out Lie-related operads as the ones
satisfying a certain minimality condition with respect to their
multifiltrations. 
\revnm{reworded}{The operads satisfying this minimality condition, which we call {\em tight\/}, 
are characterized by the property that each of them admits a presentation $\oP=\Free/(R)$, where the relations $R$ with respect to the generators $E$
stay within the multifiltration components of $\oP$ of the lowest
possible order. Here, the relevant multifiltration (the \emph{standard {\it D}-multifiltration\/}) is canonically associated with a choice of generators $E$ of $\oP$.
For instance, in the particular case of the Lie operad $\Lie$ the property of being tight shows up in the fact that the Jacobiator
\[
\Jac(a,b,c) := [a,[b,c]] + [b,[c,a]] +[c,[a,b]]
\] 
is a differential operator of order $1$ in each variable, provided that an antisymmetric operation $[-,-]:A\otimes A\to A$ on a graded commutative associative algebra $A$ is a differential operator of order $1$ with respect to each of the arguments. 
An analogous property does not hold e.g.~for associative algebras. 
Given an operation $\star : A \ot A \to A$ which is a first-order differential operator in each variable
with respect to the graded commutative associative structure on $A$, its associator $\rmAss(a,b,c)=(a\star b)\star c - a\star (b\star c)$ is, in general, a differential operator of order $2$ with respect to its arguments.
The property of being tight comes in handy in our approach to deformation of the big bracket. 
}

\begin{center}
{\bf Layout of the paper}
\end{center}

The paper is divided into two parts. In Section~\ref{DiffOpsSec} of
{\bf Part~\ref{general}\/} we collect some basic facts concerning differential
operators on \rev{4}{graded} commutative associative algebras.  That section claims no
originality whatsoever, but we pay particular attention to 
the non-unital setting, keeping in mind the case
of algebras of smooth functions with finite support on non-compact
manifolds.

Section~\ref{Pristi nedeli je vyrocni schuze.}
features a proposed operadic framework for working with multilinear differential
operators, where the notion of a multifiltration  of a $\bbk$-linear
operad is introduced.
Particular attention is paid to the case of {\it D}-multifiltrations that
formalize the compositional properties  of multilinear differential
operators.
As the main technical tool for constructing {\it D}-multifiltrations we
generalize the notion of the standard filtration of an associative algebra
to $\bbk$-linear operads. 
Explicit examples of standard {\it D}-multifiltrations are to be found in
Section~\ref{4 dny s Jarkou na chalupe.}. Finally, in
Section~\ref{OpAlgsSec},  
operator and formal operator algebras are introduced and the first
examples are given.

{\bf Part~\ref{Exam}} is devoted to some concrete examples of 
multilinear differential operator algebras. 
Specifically, Sections~\ref{Po sesti dnech.} and~\ref{V Jicine jsem nebyl kvuli
  spatne predpovedi pocasi.} cover the case of operator $\LLL$-algebras and their
particular instances -- $\IBL_\infty$-alge\-bras, commutative $\BV_\infty$-algebras,  operator Lie algebras
and Poisson algebras. In particular, we state some results concerning operator Lie algebras
whose underlying algebra is free as a graded commutative associative algebra, 
having in mind a certain deformation of the big bracket that we discuss in Section~\ref{Pujdu si zabehat ale moc se mi nechce.}.
\rev{69}{In the latter section, in addition to the operadic machinery, the proof of Theorem~\ref{Tento vikend bude pocasi velmi hrozne.} and the content of Remarks~\ref{Na chalupu pojedeme asi az zitra.} and ~\ref{MinModRmk} require some technical results concerning PROPs and properads. Since these are not used elsewhere in the paper, we refer the reader to the works ~\cite{val} and ~\cite{merkulov-valletteI} for the necessary background material.
The last section features a brief overview of operator algebras over some operads other than $\Lie$ and $\oLLL$.
Namely, we discuss an example of a formal associative operator algebra provided by Terilla's deformation formula and an example 
of a Leibniz operator algebra due to Kanatchikov.
}

\noindent
{\bf Conventions.} 
Throughout the text, $\bbk$ will denote a field of characteristic
$0$. 
The symmetric group on $n$ elements will be denoted by $\Sigma_n$, with
$\id_n \in \Sigma_n$ denoting its unit.  
All algebraic objects will be assumed to live in the symmetric monoidal
category of graded $\bbk$-vector spaces
\rev{48}{%
with the symmetry 
\[
\tau_{V,U}:x\otimes y \longmapsto (-1)^{|x|\cdot |y|}y\otimes x
\]
for any homogeneous $x\in V$, $y\in U$. 
In particular for $\sigma\in\Sigma_n$ and homogeneous $x_1,\dots, x_n\in V$,
the {\em Koszul sign\/} $\epsilon(\sigma)=\epsilon(\sigma;\Rada x1n) \in \{-1,+1\}$ is defined~via
\[
x_1\otimes\dots\otimes x_n = \epsilon(\sigma;x_1,\dots,x_n)
\cdot x_{\sigma(1)}\otimes \dots \otimes x_{\sigma(n)}.
\]
We will also denote
\[
\label{Co to svedeni znamena?}
\chi(\sigma) = \chi(\sigma;\Rada x1n)
:= \sgn(\sigma)\cdot \epsilon(\sigma;\Rada x1n).
\]
}
We will use the notation
\hbox{$\uparrow\! V$}, resp.~\hbox{$\downarrow\! V$} for the suspension,
resp.~the desuspension, of a graded vector space $V$.
\revnm{relocated}
{%
The free graded commutative unital associative algebra on a graded $\bbk$-vector space $X$
will be denoted by
\[
\S(X) := \bigoplus_{n \geq 0} \S^n(X),
\]
where $\S^n(X)$ is the $n$-th symmetric power of $X$. We also denote
\begin{align*}
\S^{\leq n}(X)& := \bigoplus_{0 \leq k \leq n} \S^k(X) = \bbk \oplus  X
\oplus  \S^2(X) \oplus \cdots \oplus \S^n(X),\ \hbox { and}
\\
\S_+^{\leq n}(X)& := \bigoplus_{1 \leq k \leq n} \S^k(X) = X
\oplus  \S^2(X) \oplus \cdots \oplus \S^n(X). 
\end{align*}
}
\revnm{}
{All operads are assumed to be unital.}
The Jacobiator
is usually defined as
\[
\Jac(a,b,c) = [a,[b,c]] + [b,[c,a]] +[c,[a,b]]
\]
while in the context of $L_\infty$-algebras the form 
\[
\Jac(a,b,c) = [[a,b],c] + [[b,c],a]+ [[c,a],b]
\]
is preferred. Since the difference is  only an overall sign which
plays no r\^ole in our theory, we will freely use, depending on the
context, both conventions. 
\noindent

{\bf Acknowledgments:} We are indebted to Vladimir Dotsenko for drawing
our attention to citations~\cite{.} and~\cite{guid} and to the anonymous referee for many helpful suggestions.
\part{Calculus of multilinear differential operators}
\label{general}

\section{Higher order derivations and differential operators} 
\label{DiffOpsSec}
In this section we present  some necessary 
terminology and results concerning higher order differential
operators and derivations. While standard
citations~\cite{Koszul,Nakai, EGA4, markl:la,markl:ab,MV}
assume the existence
of a unit in underlying
algebra, we need to work in a nonunital setup. This
requires particular care since some concepts of
the unital case do not translate directly. The main results here are
Propositions~\ref{V pondeli to nebude zadna slava.},~\ref{Zitra to snad
  alespon prestehujeme.} and~\ref{Predevcerem jsem vlekal.}. Their
proofs are given at the end of this section.

Throughout this section, we    
suppose that $A$ is a~graded commutative associative, not necessary
unital,  
algebra and \hbox{$\nabla: A \to A$}, possibly decorated with indices, a
homogeneous linear map. As in \cite{markl:la}, 
we define inductively, for each $n\geq 1$, 
the {\em deviations\/} \
$\PD n: A^{\ot n}\to A$ by
\begin{eqnarray}
\nonumber 
\PD1(a) &:=& \nabla(a),
\\
\nonumber 
\PD2(a,b)&:=& \nabla(ab) - \nabla(a)b - \sign {|\nabla|\cdot|a|} a\nabla(b),
\\
\label{Nejak mne boli v krku}
&\vdots&
\\
\nonumber 
\PD{n+1}(\Rada a1{n+1})&:=&
\PD n(a_1,\ldots,a_na_{n+1})- \PD n(\Rada a1n)a_{n+1}
\\
\nonumber 
&&-\sign{|a_n|\cdot |a_{n+1}|}\PD n(\Rada
a1{n-1},a_{n+1})a_n.  
\end{eqnarray}
A non-inductive formula
for $\PD n$ can be found in~\cite[page 373]{markl:la}. In noncommutative probability theory,
$\PD{n+1}$ is known as the $n$th {\em infinitesimal cumulant\/} of\,$\nabla$ with respect to the multiplication of $A$.  

\begin{definition}
A linear map $\nabla : A \to A$ is a {\em derivation of order $r$\/} if $\PD
{r+1}$ is identically zero. We denote by $\Der r$, $r \geq 0$, 
the linear space of
derivations of order $r$.
\end{definition}

Notice that $\Der 0 = 0$ while $\Der 1$ is the space of usual
derivations of the algebra $A$. It follows from
(\ref{Nejak mne boli v krku}) that if $\PD {r+1}$ identically vanishes,
then so does $\PD {r+2}$, thus $\Der r \subset \Der {r+1}$. 
For homogeneous linear maps 
$\nabla_1,\nabla_2: A \to A$ we denote, as
usual, by 
\[
[\nabla_1,\nabla_2] := \nabla_1\circ \nabla_2 -
{(-1)}^{|\nabla_1||\nabla_2|} \nabla_2\circ \nabla_1
\] 
their graded commutator.

\begin{definition}
\label{Kdy budu mit ty voziky?}
The space $\Dif r$, $r \geq 0$, of  
differential operators of order $r$ is defined inductively as follows:
\begin{itemize}
 \item [(i)]
 $\Dif 0:=\big\{\,L_a:A\to A\ |\ \,a\in A\,\big\}$, where $L_a:x\mapsto a x$ is
 the operator of left multiplication by $a \in A$, and
 \item [(ii)]
 $\Dif r
:=\big\{\,\nabla \ | \ [\nabla,L_a]\in \Dif {r-1}\text{ for all }a\in A\,\big\}$,
$r \geq 1$.
\end{itemize}
\end{definition}

Furthermore, as a convenient convention we set $\Dif {-1}:=0$. 
This is consistent with the above definition, since
$[\nabla, L_a]=0\in \Dif {-1}$
for any $\nabla\in \Dif 0$. To complete the picture, we recall 
still another definition of differential operators that can be found
in the literature. It uses the derived $\bbk$-linear mappings
\[
 \Psi_\nabla^{n}(a_1,a_2\dots a_n):=[\dots[[\nabla, L_{a_1}],
 L_{a_2}],\dots L_{a_{n}}]:A\longrightarrow A, \ \Rada a1n \in A.
\]
As a definition, we set $\Psi_\nabla^0:=\nabla$. In particular, the first few iterations read
\begin{align}
\nonumber
 \Psi_\nabla^0(x)=&\ \nabla(x),
\\
\nonumber
 \Psi_\nabla^1(a_1)(x)=&\ \nabla(a_1x)- \sign{|a_1|\cdot |\nabla|}a_1\nabla(x),
\\
\label{Psi2}
 \Psi_\nabla^2(a_1,a_2)(x)=&\ \nabla(a_1a_2x)-
\sign{|a_1|\cdot |\nabla|}
a_1\nabla(a_2x)
\\
\nonumber
&\ -
\sign{|a_2|\cdot |\nabla|}a_2\nabla(a_1x)+
\sign{(|a_1| + |a_2|)\cdot |\nabla|}a_1a_2\nabla(x), \hskip 1em \&c.
\end{align}
For $r \geq -1$ we define
\begin{equation}
\label{Dovolam se Milanovi?}
\dif r := \big\{\ \nabla \ | \ \Comut {r+1}(\Rada a1{r+1}) = 0
\hbox { \ for all \  $\Rada a1{r+1} \in A$}\ \big\}.
\end{equation}
Below we formulate the main results of this section,
Propositions~\ref{V pondeli to nebude zadna slava.},~\ref{Zitra to
  snad alespon prestehujeme.} and~\ref{Predevcerem jsem vlekal.}. Their
proofs are to be found at the end of this section.
The first one specifies the relation between the above three~definitions.

\begin{proposition}
\label{V pondeli to nebude zadna slava.}
For arbitrary $r \geq 0$, $\Der r \subset \Dif r \subset \dif r$.
If $A$ has a unit $1 \in A$, then  $\Dif r = \dif r$ and, moreover,
\begin{equation}
\label{Kdy se zase vznesu?}
\Der r = \big\{\, \nabla \in \Dif r  \ | \ \nabla(1) = 0 \, \big\}. 
\end{equation}
\end{proposition}

\begin{corollary}
\label{Je pondeli a uz se nutim do prace.}
If $A$ is unital, then there exists a canonical isomorphism
\[
\Dif r \cong \Der r \oplus A, \ r \geq 0.
\]   
\end{corollary}

\begin{proof}
Notice that the operator  of left multiplication $L_a: A \to A$ belongs
to $\Dif r$ for any $r \geq 0$. Thus $\nabla - L_{\nabla(1)} \in \Dif
r$ and, since it clearly annihilates the unit,  $\nabla - L_{\nabla(1)} \in \Der
r$ by~\eqref{Kdy se zase vznesu?}. Thus the correspondence 
\[
\nabla
\longmapsto (\nabla - L_{\nabla(1)}) \oplus \nabla(1)
\]
defines a map $\Dif r \to \Der r \oplus A$ whose inverse is given by
$\theta \oplus a \mapsto \theta + L_a$.
\end{proof}

\rev{59}
{%
\begin{lemma}
\label{Koupil jsem si tenisky.}
Derivations (resp.\ differential operators) of order $n$ on
$\S(X)$ are uniquely determined by their 
  restriction to $\S_+^{\leq n}(X)$  (resp.~to $\S^{\leq n}(X)$). 
\end{lemma}

\begin{proof}
The derivations part is \cite[Proposition 3]{markl:la}. Now, given a
differential operator $\nabla$ of order~$n$, we consider, using Corollary~\ref{Je pondeli a uz se nutim do prace.}, its unique
decomposition $\nabla = \theta + L_a$, where $\theta$ is a derivation
of order $n$ and $L_a$ is the operator of left multiplication by $a\in \S(X)$ As mentioned, the former is determined by its restriction to $\S_+^{\leq n}(X)$, while $a=\nabla(1)$.
\end{proof}
}

For subspaces $S_1,S_2$ of the space $\Hom(A,A)$ of $\bbk$-linear endomorphisms $A \to A$ denote
\[
S_1 \circ S_2 := \big\{\ \nabla_1 \circ \nabla_2 \ | \ \nabla_1 \in
S_1, \nabla_2 \in 
S_2\ \big \}\
\hbox { and }\
[S_1 , S_2] := \big\{\ [\nabla_1 , \nabla_2] \ | \ \nabla_1 \in
S_1, \nabla_2 \in 
S_2\ \big \}.
\]
One then has

\begin{proposition}
\label{Zitra to snad alespon prestehujeme.}
Under the above notation, the following inclusions hold 
for arbitrary $m,n \geq 0$: 
\begin{itemize}
\item[(i)]
$\Der m \circ \Der n \subset \Der {m+n}$,
\item[(ii)]
$\Dif m \circ \Dif n \subset \Dif {m+n}$, and
\item[(iii)]
$\dif m \circ \dif n \subset \dif {m+n}$.
\end{itemize}
Likewise, for the graded commutators one has
\begin{itemize}
\item[(iv)]
$\big[\Der m , \Der n\big] \subset \Der {m+n-1}$, and
\item[(v)]
$\big[\Dif m , \Dif n\big] \subset \Dif {m+n-1}$.
\end{itemize}
\end{proposition}

\begin{remark}
Notice that, for a general non-unital algebra $A$, the inclusion  
\[
\big[\dif m , \dif n\big] \subset \dif
{m+n-1}
\] 
analogous to (iv) and (v) above, need not hold.
As an example, take $A$ to be a $d$-dimensional \hbox{$\bbk$-vector}
space
\rev{6}{placed in degree zero}
with trivial multiplication. Then $\Psi^1_\nabla = 0$ for arbitrary
$\nabla$, so  \redukcem{\dif 0} is, by definition, the
space of all linear endomorphisms $A \to A$, i.e.\ the algebra of $d
\times d$ matrices $M_d(\bfk)$. If $d \geq 2$,  $M_d(\bfk)$ is
non-commutative, therefore
\[
0  \not= \big[M_d(\bfk),M_d(\bfk)\big] = \big[\dif 0 , \dif 0\big]
\not \subset  \dif{-1}  =0.
\]
\end{remark}

The last of the main statements of this section is

\begin{proposition}
\label{Predevcerem jsem vlekal.}
 For a arbitrary $r\geq 0$, both $\Dif r$ and \, $\dif r$ are sub-bimodules of the space
 $\Hom(A,A)$ of linear maps $A \to A$
 with its natural $A$-$A$-bimodule structure
 \begin{align*}
  (b,\nabla)\mapsto L_b \circ \nabla,\quad (\nabla,b)\mapsto
   \nabla \circ L_b.
 \end{align*}
Moreover, $\Der r$ is a left submodule of \ $\Hom(A,A)$ with
respect to the action $ (b,\nabla)\mapsto L_b \circ \nabla$.
\end{proposition}

\begin{remark}
The subspaces $\Der r \subset \Hom(A,A)$ are not, in general, {\em right\/}
submodules with respect to the action $(\nabla,a)\mapsto
   \nabla \circ L_a$, not even when $A$ is unital.
Assume, for instance, that $\nabla \in \Der
   1$, i.e.\ that $\PD2(a_1,a_2) = 0$ for each $a_1,a_2 \in A$. It is
   easy to check that then
\[
\Phi^2_{\nabla \circ L_a}(a_1,a_2) = - \nabla(a)a_1a_2.
\]
Now take $A$ to be the polynomial ring $\bbk[x]$, $\nabla : = 
\frac {d\hphantom{x}}{dx}$ the standard
derivation and $a \in  \bbk[x]$ any non-constant polynomial.
While $\nabla \in \Der 1$, $\Phi^2_{\nabla \circ L_a} \not= 0$, so  
$\nabla \circ L_a \not\in \Der 1$.
\end{remark}

The relations between the various subspaces of $\Hom(A,A)$ introduced
above are summarized in the diagram
\[
\xymatrix@C=1.2em@R=1.2em{
  &\ar@{^{(}->}[d]
\ar@{^{(}->}[r]    0= \Der 0 \ &\ar@{^{(}->}[r]\ar@{^{(}->}[d] \Der 1\
&\ar@{^{(}->}[d]\ \ar@{^{(}->}[r] \Der 2 \
&\ar@{^{(}->}[r] \cdots & \Hom(A,A)  \ar@{=}[d] 
\\
0 = \Dif {-1} \ \ar@{^{(}->}[r] \ar@{=}[d]  &\ar@{^{(}->}[d]
\ar@{^{(}->}[r] \Dif 0\ &\ar@{^{(}->}[d]\ar@{^{(}->}[r] \Dif 1\ &\ar@{^{(}->}[d] \ar@{^{(}->}[r] \Dif 2\
&\ar@{^{(}->}[r] \cdots & \Hom(A,A)  \ar@{=}[d] 
\\
0 = \dif {-1} \ \ar@{^{(}->}[r]  &
\ar@{^{(}->}[r] \dif 0\ &\ar@{^{(}->}[r] \dif 1 \ &\ \ar@{^{(}->}[r] \dif 2\
&\ar@{^{(}->}[r] \cdots & \Hom(A,A)  
}
\]
in which the top row consists of inclusions of left $A$-modules and
the remaining two rows of inclusions of $A$-$A$-bimodules. The
vertical inclusions between the upper two rows are inclusions of left
$A$-modules, the  inclusions between the bottom ones are that of 
$A$-$A$-bimodules. If $A$ possesses a unit, the bottom two rows are isomorphic.
The rest of this section is devoted to the proofs of the above
propositions and necessary auxiliary results, some of them being of
independent interest. 

\vskip .5em
\noindent 
In what follows, to simplify the exposition, we assume that all objects are of degree $0$. In the general graded case, the
formulas can 
\rev{7}
{%
be properly adjusted by following the Koszul sign rule saying that whenever we commute two entities
of degrees $p$ and $q$, respectively, we multiply by~$(-1)^{pq}$.
}

Let $R$ be a (noncommutative) ring and let $[x,y]$ denote the commutator 
$xy-yx$ for all $x,y\in R$. Then 
\begin{align}
 \label{NCLeibniz}
 [xy,z]=xyz-zxy=xyz-xzy+xzy-zxy=x[y,z]+[x,z]y.
\end{align}
Applying this to $R=\Hom(A,A)$ with the standard composition as the multiplication, $x=L_a$,
$y=\nabla_1$ and $z=\nabla_2$, where $\nabla_{1}$, $\nabla_2$ are arbitrary
$\bbk$-linear endomorphisms of $A$, we get
\begin{subequations}
\begin{align}
\label{NCLeibniz2}
   [\nabla_1\circ \nabla_2, L_a]=\nabla_1\circ [\nabla_2, L_a]+[\nabla_1, L_a]\circ \nabla_2.
\end{align}
Similarly, for any $\nabla\in \Hom(A,A)$, $a_1,a_2\in A$,
\begin{align}
\label{NCLeibniz4}
 [\nabla, L_{a_1}\circ L_{a_2}]=[\nabla, L_{a_1}]\circ L_{a_2}+L_{a_1}\circ[\nabla, L_{a_2}].
\end{align}
Furthermore, the Jacobi identity for the commutator reads
\begin{align*}
  \label{NCLeibniz3}
  \big[[\nabla_1, \nabla_2], L_a\big]=\big[\nabla_1,[\nabla_2,
  L_a]\big]- \big[\nabla_2,[\nabla_1, L_a]\big]. 
\end{align*}
A simple formula 
\begin{equation}
\label{Uz mam novou vrtuli - jak dlouho vydrzi?}
L_{a_1} \circ L_{a_2} = L_{a_1a_2}
\end{equation}
that holds for any $a_1,a_2 \in A$, will
also be useful. We will also need the following 
\rev{8}{simple}
\end{subequations}

\begin{lemma}
\label{Uz jsme to alespon vyndali.}
Assume that $A$ is unital. Then
a map $\Delta :A \to A$
commutes with the operator $L_a$ of left multiplication for any $a \in A$ if and only if
it is itself an operator of left multiplication. 
\end{lemma}

\begin{proof}
By definition, $[\Delta,L_a] = 0$ means that  
$\Delta (au) = a\Delta(u)$ for all $u \in A$. 
Taking $u = 1$ gives
$\Delta (a) = a\Delta(1) =  \Delta(1)a$,
so~$\Delta$ is the operator of left multiplication by $\Delta(1)$.
\end{proof}

The following statement is also an easy observation.

\begin{lemma}
\label{Marcelka}
Assume that $\nabla \in \Hom(A,A)$ is  such that \ $\nabla(1) = 0$. 
For each $n\geq 1$, \hbox{$\PD n(\Rada a1{n}) = 0$} if at
least one of its variables equals $1$.
\end{lemma}

\begin{proof}
The claim is obvious for $n=1$. For $n > 2$ it follows from 
defining formulas~(\ref{Nejak mne boli v krku}) by simple induction.
\end{proof}

For a commutative associative (unital or nonunital) algebra $A$
\rev{9}{denote by $\tA$ its `unitalization,'} i.e.\  the original algebra with
an artificially added unit. Explicitly,
$\tA = A \oplus \bfk$ as $\bbk$-vector spaces, and the multiplication
given~by
\[
(a' \oplus \gamma')(a'' \oplus \gamma'') = (a'a'' + \gamma' a'' + \gamma'' a') \oplus (\gamma'\gamma'')
\]
for $a',a''\in A$, $\gamma',\gamma'' \in \bbk$. Notice that if $A$ was unital,
its unit does not coincide with the newly
added unit $\tu$   of $\tA$. 
Lemma~\ref{Marcelka} will be used in the proof of

\begin{lemma}
\label{Privezl jsem tyce.}
Let \ $\nabla \in \Hom(A,A)$ and  \ $\tnabla \in \Hom(\tA,\tA)$ be its
extension by \ $\tnabla (\tu) :=0$. Then $\nabla \in \Der r $ if and
only if \ $\tnabla \in \tDer r$.
\end{lemma}

\begin{proof}
If  $\nabla \in \Der r$, $\PD {r+1}$ is identically zero by
definition, and the same is true also for
$\tPD {r+1}$. Indeed, if all $\Rada a1{r+1}$ belong to $A$, then
\[
\tPD  {r+1}(\Rada a1{r+1}) = \PD  {r+1}(\Rada a1{r+1})  = 0
\]
since  $\nabla \in \Der r$. If at least one of  $\Rada a1{r+1}$ equals
the added unit $\tu$, then $\tPD  {r+1}(\Rada a1{r+1})$ vanishes 
by Lemma~\ref{Marcelka}. The opposite implication is clear, since the restriction
of a differential operator to a subalgebra 
is a differential operator again.
\end{proof}

\rev{9}{%
The unitalization $\tA$ and the related Lemma~\ref{Privezl jsem tyce.} will be
invoked again in the proof of Proposition~\ref{Zitra to snad alespon
  prestehujeme.}.} The next lemma provides an inductive formula for iterated left
multiplications of the same spirit as~\eqref{Nejak mne boli v krku}.

\begin{lemma}
\label{IBRec}
For any $n\geq 1$ and any $a_1,a_2\dots a_{n+1}\in A$,
\begin{align*}
 \Psi_\nabla^{n+1}(a_1,\dots, a_{n+1})=
 \Psi_\nabla^{n}(a_1,\dots, a_{n}a_{n+1})-a_n\Psi_\nabla^{n}(a_1,\dots, a_{n-1},a_{n+1})-a_{n+1}\Psi_\nabla^{n}(a_1,\dots, a_{n}).
\end{align*}
\end{lemma}
\begin{proof}
First note that, by the noncommutative Leibniz identity \eqref{NCLeibniz4},
 we have
 \begin{align*}
  [\Psi_\nabla^{n-1}(a_1,\dots,a_{n-1}), L_{a_n}\circ
   L_{a_{n+1}}]=\ &
  [\Psi_\nabla^{n-1}(a_1,\dots, a_{n-1}), L_{a_n}]\circ L_{a_{n+1}}
\\
  &\ +L_{a_n}\circ [\Psi_\nabla^{n-1}(a_1,\dots, a_{n-1}),
    L_{a_{n+1}}]
\\
=\ &\Psi_\nabla^{n}(a_1,\dots, a_{n})\circ L_{a_{n+1}}+
a_n\Psi_\nabla^{n}(a_1,\dots, a_{n-1},a_{n+1}).
 \end{align*}
Invoking~\eqref{Uz mam novou vrtuli - jak dlouho vydrzi?}, we also have
\begin{align*}
 [\Psi_\nabla^{n-1}(a_1,\dots, a_{n-1}), L_{a_n}\circ L_{a_{n+1}}]
=[\Psi_\nabla^{n-1}(a_1,\dots, a_{n-1}), L_{a_na_n+1}]
=\Psi_\nabla^{n}(a_1,\dots, a_{n}a_{n+1})
\end{align*}
which, combined with the previous display, gives
 \[
  \Psi_\nabla^{n}(a_1,\dots, a_{n})\circ L_{a_{n+1}}=
\Psi_\nabla^{n}(a_1,\dots, a_{n}a_{n+1})-a_n\Psi_\nabla^{n}(a_1,\dots,a_{n+1}).
 \]
 Substituting this into
 \begin{align*}
  \Psi_\nabla^{n+1}(a_1,\dots, a_{n+1})
&=[\Psi_\nabla^{n}(a_1,\dots, a_{n}), L_{a+1}]
  =\Psi_\nabla^{n}(a_1,\dots, a_{n})\circ L_{a+1}-
  L_{a_{n+1}}\circ \Psi_\nabla^{n}(a_1,\dots, a_{n})
 \end{align*}
 yields the required result.
\end{proof}

\begin{corollary}
 For any $n\geq 1$, $\Psi_\nabla^n(a_1,\dots, a_n)$ is 
symmetric as a function of $a_1,\dots, a_n\in A$.
\end{corollary}

\begin{proof}
Induction on $n$, using Lemma~\ref{IBRec}.
\end{proof}

An intriguing and important relation between $\Psi_\nabla^n$, $\PD n$ and $\PD
{n+1}$ is given in

\begin{proposition}
Let $\nabla:A\to A$ be a\ $\bbk$-linear mapping. Then, for $n\geq 1$
and any $x,a_1,\dots, a_n\in A$,
 \begin{equation}
\label{Bude to problem.} 
 \Psi_\nabla^{n}(a_1,\dots, a_n)(x)=
  \Phi_\nabla^{n+1}(a_1,\dots, a_{n},x)+x \Phi_\nabla^{n}(a_1,\dots, a_n).
 \end{equation}
\end{proposition}
\begin{proof}
 For the base case $n=1$, we have
 \[
  \Psi_\nabla^1(a)(x)=\nabla(ax)-a\nabla(x)
=\nabla(ax)-a\nabla(x)-x\nabla(a)+x\nabla(a)
  =\Phi_\nabla^2(a,x)+x\Phi_\nabla^1(a).
 \]
 For $n \geq 1$, we begin by noting that, by definition,
 \begin{align*}
  \Psi_\nabla^{n+1}(a_1,\dots, a_{n+1})(x)
&=\ [\Psi_\nabla^{n}(a_1,\dots, a_n), L_{a_{n+1}}](x)
\\
  &=\ \Psi_\nabla^n(a_1,\dots ,a_n)(a_{n+1}x)-
a_{n+1}\Psi_\nabla^n(a_1,\dots ,a_n)(x).
 \end{align*}
 By induction, the two terms in the right-hand side are equal to
 \begin{align*}
\Phi_\nabla^{n+1}(a_1,\dots a_n, a_{n+1}x)+a_{n+1}x\Phi_\nabla^n(a_1,\dots, a_n)
 \end{align*}
 and
 \begin{align*}
-  a_{n+1}\Phi_\nabla^{n+1}(a_1,\dots, a_n,x)
- a_{n+1}x\Phi_\nabla^n(a_1,\dots, a_n),
 \end{align*}
 respectively.
 Hence,
 \begin{align*}
\Psi_\nabla^{n+1}(a_1,\dots, a_{n+1})(x)
=&\ \Phi_\nabla^{n+1}(a_1,\dots, a_n, a_{n+1}x)
-a_{n+1}\Phi_\nabla^{n+1}(a_1,\dots, a_n,x)
\\
=&\ \Phi_\nabla^{n+1}(a_1,\dots, a_n, a_{n+1}x)
-a_{n+1}\Phi_\nabla^{n+1}(a_1,\dots, a_n,x)
\\&-x\Phi_\nabla^{n+1}(a_1,\dots,a_{n+1})
+x\Phi_\nabla^{n+1}(a_1,\dots,a_{n+1})
\\
&\hskip -1.2em
\overset{\text{by }\eqref{Nejak mne boli v krku}}{=}
\Phi_\nabla^{n+2}(a_1,\dots,a_{n+1},x)+x\Phi_\nabla^{n+1}(a_1,\dots, a_{n+1})
\end{align*}
 as desired. 
\end{proof}

Notice that~\eqref{Bude to problem.} with $x=1$  combined with 
Lemma~\ref{Marcelka}
gives the well-known equation
\begin{equation}
\label{Chobotnicka}
\Comut n(\Rada a1n)(1) = 
\PD{n}(a_1\ldots,a_n).
\end{equation}

\begin{lemma}
 \label{DiffOpCrit2}
Assume that  $\nabla:A\to A$ is a $\bbk$-linear mapping.
Then $\nabla \in \Dif r$ for some  $r\geq 1$ if and only if, for arbitrary
$a_1,\dots, a_r\in A$,
\begin{subequations}
\begin{equation}
\label{Udela se neco na voziku?}
\Psi_\nabla^r(a_1,\dots, a_r)\in \Dif 0.
\end{equation}
If $A$ is unital, the above condition is equivalent to
\begin{equation}
\label{Kolik najedu?}
\Psi_\nabla^{r+1}(a_1,\dots, a_{r+1})= 0
\end{equation}
for any $\Rada a1{r+1} \in A$.
\end{subequations}

\end{lemma}
\begin{proof}
Let $\nabla\in \Dif r$.
Then, as it follows directly from the definition, for any $0<k\leq r+1$ and 
$a_1,a_2,\dots, a_{k}\in A$,
\begin{align*}
 \label{OrderReduction}
 \Psi_\nabla^{k}(a_1,a_2,\dots, a_{k})=
[\squeezedldots[[\nabla, L_{a_1}], L_{a_2}],\dots L_{a_{k}}] \in 
\Dif {r-k}.
\end{align*}
This with $r=k$ gives~\eqref{Udela se neco na voziku?}. 

To get the converse, consider first the base case $r=1$. 
Namely, let $\nabla:A\to A$
be such that for any $a\in A$, $\Psi_\nabla^1(a)=[\nabla,L_a]\in \Dif0$.
 Then $\nabla$ is a differential operator of order one by its very definition. 
 
Now, let $r>1$ and $\nabla$ be such that
$\Psi_{\nabla}^{r}(a,a_1,\dots, a_{r-1})\in \Dif0$ for
any $a,a_1,\dots, a_{r-1}\in A$.  We have
 \[
  \Psi_{[\nabla,L_a]}^{r-1}(a_1,\dots a_{r-1})
=\Psi_{\nabla}^{r}(a_1,\dots, a_{r-1},a)\in \Dif0.
 \]
 Then, by induction, $[\nabla, L_a]\in \Dif {r-1}$. Hence, 
$\nabla\in \Dif r$ which finishes the proof of the first part of the lemma.
 
Let us proceed to the second part assuming that $A$ is unital.
If $\nabla\in \Dif r$, we already know that
$\Psi_\nabla^r(a_1,\dots, a_r)\in \Dif 0$. Since  $[\Delta,L_a]
= 0$ for any $\Delta \in \Dif 0$ and $a \in A$,  
\[
\Psi_\nabla^{r+1}(a_1,\dots, a_{r+1})=   [\Psi_\nabla^{r}(a_1,\dots,
a_{r}),L_{a_{r+1}}] = 0,  
\]
which is~\eqref{Kolik najedu?}.
 
On the other hand, 
the vanishing~\eqref{Kolik najedu?} means that $\Psi_\nabla^{r}(a_1,\dots,
a_{r-1})$ commutes with the operator $L_a$ for any
$a \in A$ so it is, by Lemma~\ref{Uz jsme to alespon vyndali.}, an 
operator of left multiplication, i.e.~\eqref{Udela se neco na voziku?} holds.
Thus $\nabla \in \Dif r$ by the first part of the lemma.
\end{proof}

\begin{remark}
  Note that, without the unitality assumption on $A$, the statement of the
  second part of Lemma~\ref{DiffOpCrit2} is false.  Indeed, let $A$ be
  a vector space with trivial multiplication.  Since the operators of
  left multiplication are trivial as well,
  $\Psi_\nabla^1(a)=[\nabla, L_a]=0$ for any $a\in A$ and
  $\nabla : A \to A$, yet $\nabla$ is a differential operator in
  $\Dif 0$ as per Definition \ref{Kdy budu mit ty voziky?} only if
  $\nabla = 0$.
\end{remark}

\begin{proof}[Proof of Proposition~\ref{V pondeli to nebude zadna
    slava.}]
As in the proof of the first part of Lemma~\ref{DiffOpCrit2}, we inductively
establish that $\nabla \in \dif r$ if and only if
$\Psi_\nabla^{r+1}(a_1,\dots, a_{r+1})= 0$ for each $\Rada a1{r+1} \in
A$. If $\nabla \in \Dif
r$,  $\Psi_\nabla^{r}(a_1,\dots, a_{r})  \in \Dif 0$ by~\eqref{Udela se
  neco na voziku?}, so 
\[
\Psi_\nabla^{r+1}(a_1,\dots, a_{r+1})= 
[\Psi_\nabla^{r}(a_1,\dots, a_{r}),L_{a_{r+1}}]= 0,
\] 
thus $\nabla \in \dif r$. This proves the 
inclusion $\Dif r \subset \dif r$.

Assume that $\nabla \in \Der r$. By definition, $\PD{r+1}(\Rada a1r,x)
= 0$ for arbitrary $\Rada a1r,x$   so, by~\eqref{Bude to problem.} with $n=r$, 
\[
\Comut r(\Rada a1r)(x) =   \PD r(\Rada a1r)x.
\]
Thus $\Comut r(\Rada a1r)$ is the operator of left multiplication by
$\PD r(\Rada a1r) \in A$, meaning that 
\[
\Comut r(\Rada a1r)  \in \Dif 0
\]
thus $\nabla \in \Dif r$ by Lemma~\ref{DiffOpCrit2}. Therefore $\Der r \subset
\Dif r$, finishing the proof of the first part of the proposition.

Assume that $A$ is unital. By Lemma~\ref{Uz jsme to alespon vyndali.},
$\Comut r(\Rada a1n)$ commutes with the operator $L_a$ for any
$a \in A$ if and only if it is an operator of left multiplication. This
proves that $\Dif r = \dif r$.

Let $\nabla \in \Dif r$ be  such that $\nabla(1) = 0$. 
By~(\ref{Chobotnicka}), if $\Comut {r+1}(\Rada a1{r+1})$   vanishes,
and so does
$\PD{r+1}(a_1\ldots,a_{r+1})$. Thus $\nabla \in \Der r$. 

On the other hand, if $\nabla \in \Der r$,    $\PD {r+1}$
vanishes by definition, and so does  $\Comut {r+1}(\Rada a1{r+1})$    for
  all $\Rada a1{r+1}$ by~\eqref{Bude to problem.} with
  $n=r+1$. 
Therefore $\nabla \in \Dif r$. 
It remains to prove that $\nabla(1) = 0$.   Taking $a_1 =
  a_2 = \cdots = 1$ in~(\ref{Nejak mne boli v krku}) gives
\begin{equation}
\label{Jarka opet bez prace.}
\PD n(1,\ldots,1) = - \PD {n+1}(1,\ldots,1)
\end{equation}  
for any $n\geq 1$. Since $\nabla$ is a derivation of order $r$, 
$ \PD {r+1}(1,\ldots,1) =0$
thus  $\PD {1}(1) = \nabla(1) = 0$ by iterating (\ref{Jarka opet bez
  prace.}). This finishes the proof.
\end{proof}

\begin{proof}[Proof of Proposition~\ref{Zitra to snad alespon prestehujeme.}]
  We start by proving item~(ii), by induction on $m+n$.  Let $\nabla_1$
  and $\nabla_2$  belong to $\Dif 0$ which  by definition means
  that they are  both operators of left multiplication. By~(\ref{Uz
    mam novou vrtuli - jak dlouho vydrzi?}), their composite
  $\nabla_1\circ \nabla_2$ is an operator of left multiplication as
  well, so   $\nabla_1\circ \nabla_2 \in \Dif 0$.

Suppose that (ii) been established for all
$m+n\leq k$. Consider the case  $m+n=k+1$.
 For any $a\in A$,
 in equation~(\ref{NCLeibniz2}), $[\nabla_1,L_a]$ is of order 
$<m-1$ and $[\nabla_2,L_a]$ is of order $<n-1$. 
Then by the inductive assumption, each of the summands, and hence the
left hand side of~(\ref{NCLeibniz2}), is a~differential operator of
order $<m+n-1$, thus $\nabla_1\circ \nabla_2$ is of order $<m+n$, as desired.

The proof of (iii) differs from that of (ii) only in the
first inductive step. If $\nabla_1,\nabla_2 \in \redukcem{\dif 0}$,
by definition, for any
  $a\in A$, we have $[\nabla_1,L_a]=[\nabla_2,L_a]=0$. Hence, 
by~(\ref{NCLeibniz2}), $[\nabla_1\circ \nabla_2,L_a]=0$, and thus
  $\nabla_1\circ \nabla_2 \in \redukcem{\dif 0}$. We then proceed inductively as before.

It follows from~(\ref{Kdy se zase vznesu?}) combined with already
proven cases that~(i) holds for derivations annihilating the unit
of an unital algebra. Let $A$ be an arbitrary, not necessarily unital,
algebra, and $\nabla_1$, $\nabla_2$ derivations of orders $<m$ and $<n$,
respectively. Their extensions  $\tnabla_1$, $\tnabla_2$ to the unital
algebra~$\tA$ are derivations of the same respective orders by
Lemma~\ref{Privezl jsem tyce.}, so  $\tnabla_1\circ \tnabla_2$ is a derivation of order
$< m+n$ by the above reasoning related to the unital case.   
Notice that the composite  $\tnabla_1\circ \tnabla_2$ annihilates
the unit $\tu$ of $\tA$ and extends  $\nabla_1\circ \nabla_2$, so the
later is a derivation of order $< m+n$ by Lemma~\ref{Privezl jsem tyce.} again.

The proof of the remaining items is similar
except that instead of~(\ref{NCLeibniz2}) we use the 
Jacobi identity~(\ref{NCLeibniz4}).
Let us prove~(v)  by induction on $m+n$.
The base case is $m=n=0$ when $\nabla_1 = L_a$,  $\nabla_2 = L_b$ for
some $a,b \in A$. Then $[\nabla_1, \nabla_2]=[L_a, L_b]=0$ as
expected. 

Suppose that the statement has been established for $m+n\leq k$.
Consider the case $m+n=k+1$.  
In~(\ref{NCLeibniz4}), $[L_a, \nabla_1]$ is of order $m-1$ and $[L_a,
\nabla_2]$ is of order $n-1$. Then by the inductive assumption, each
of the summands, and hence the left hand side of~(\ref{NCLeibniz4}), 
is a differential operator of order $m+n-2$. Thus
$[\nabla_1, \nabla_2]$ is of order $m+n-1$, as desired.
Item~(iv) can be easily derived from~(v) by the extension trick employed in the
proof of~(i).
\end{proof}

\begin{proof}[Proof of Proposition~\ref{Predevcerem jsem vlekal.}]
The invariance of all spaces with respect to the left action follows
from the obvious equations
\[
\Phi_{L_b \circ \nabla}^n = b \cdot \PD n,\ 
\Psi_{L_b \circ \nabla}^n = b \cdot \Psi_{\nabla}^n, \ b \in A,
\]
and~(\ref{Uz mam novou vrtuli - jak dlouho vydrzi?}) which proves the
invariance of $\Dif 0$.
The invariance under the right action can be proved inductively, using
the equation
\[
[\nabla \circ L_b, L_a]=
\nabla \circ [L_b, L_a] + [\nabla, L_a] \circ L_b=  [\nabla, L_a] \circ L_b.
\]
The details can be safely left to the reader.
\end{proof}

\section{Multifiltrations and multilinear differential operators}
\label{Pristi nedeli je vyrocni schuze.}

The section is devoted to introducing an operadic framework for describing algebraic structures,
whose operations  are represented by multilinear differential operators. 
The standard references for operads
include~\cite{markl-shnider-stasheff:book,markl:handbook} and the more
recent~\cite{loday-vallette}. In what follows, all operads are assumed to be
unital, $\bbk$-linear and connected, meaning that $\oP(0) = 0$.

\subsection{Filtrations of algebras}
\label{OpMultSec}
We proceed by recalling the basic terminology concerning filtrations of associative algebras. 
Recall that an (increasing) \emph{filtration}
$\{F_i \aA\}_{i\in \mathbb{Z}}$ on an associative 
\hbox{$\bbk$-linear} algebra $\aA$ is a collection of $\bbk$-subspaces
\begin{align}
 \label{FiltFlag}
 \cdots \subseteq F_{-1} \aA \subseteq F_{0}\aA \subseteq F_{1}\aA \subseteq \dots \subseteq \aA
\end{align}
such that $F_i \aA\cdot F_j \aA \subseteq F_{i+j} \aA$ for all
$i,j\in\mathbb{Z}$.
If $\aA$ is unital with a unit $1_\aA$, we assume in addition that $1_\aA \in F_0\aA$. 
All filtrations on $\aA$ are naturally ordered by the componentwise inclusion and form a poset $\Filt(\aA)$ with the largest element being the trivial filtration $F_i \aA := \aA$ for all $i\in \mathbb{Z}$

\rev{12, \\ 13, 26(b)}
{%
One says that a filtration $FA = \{F_i \aA\}_{i\in \mathbb{Z}}$ is 
\emph{exhaustive} if $\bigcup\limits_{i\in \mathbb{Z}} F_i\aA = \aA$ and 
is \emph{bounded} if there exists $b\in \mathbb{Z}$ such that $F_i\aA = 0$ for all $i<b$.
Furthermore, a filtration $FA$ such that 
\begin{align}
\label{Dfilt}
[F_i \aA, F_j \aA] \subseteq F_{i+j-1}\aA 
\end{align}
for all $i,j\in\mathbb{Z}$.
is called a \emph{{\it D}-filtration} \cite{BeiBer}.
Note, in particular, that in such a case $F_1 A$ is a Lie algebra with respect to the commutator bracket.
An algebra $\aA$ is said to be \emph{almost commutative} if it admits a bounded exhaustive {\it D}-filtration.
The following standard example is of particular relevance for us.
}
\begin{example}
\label{DiffFilt}
Let $C$ be a commutative associative $\bbk$-algebra and $\aA :=
End(C)$ be the associative algebra of $\bbk$-linear endomorphisms of
$C$ with the multiplication given by the usual composition of linear maps. 
Proposition~\ref{Zitra to snad alespon prestehujeme.} implies that the collections of subspaces
\rev{15}
{%
\[
\DiF k(C) = \big\{O : C \to C \ | \ 
\hbox{$O$ is a differential operator of order $k$}\big\},\ k \in \bbZ
\]
and
\[
\text{Der}^k(C) =\big\{O : C \to C \ | \ 
\hbox{$O$ is a derivation of order $k$}\big\},\ k \in \bbZ
\]
comprise well-defined bounded {\it D}-filtrations of $\aA$.
}
\end{example}
\rev{16}
{%
Given an algebra filtration $F\aA=\{F_i \aA\}_{i\in \mathbb{Z}}$, one readily notes that a $\bbk$-algebra homomorphism $f: \aA \to \bB$ induces a filtration on $\bB$ by setting $G_i \bB := f(F_i \aA)$ for all $i\in \mathbb{Z}$. 
Furthermore, if $F\aA$ is exhaustive, then so is the induced filtration 
$G\bB=\{G_i \bB\}_{i\in \mathbb{Z}}$, provided that $f$ is surjective.
}
\begin{example}
  \label{UEAex}
  Let $\mathfrak{g}$ be a Lie algebra. An application of the previous observation to the evident filtration on the tensor algebra $T(\mathfrak{g})$ by the monomial degrees and the canonical surjection of $T(\mathfrak{g})$ onto $U(\mathfrak{g}) := T(\mathfrak{g})/I$,
  where $I$ is the ideal generated by the elements of the form ${x\otimes y - y\otimes x - [x,y]}$ for all $x, y \in \mathfrak{g}$,
  yields a filtration on the universal enveloping algebra of $\mathfrak{g}$.
  
  One may find it instructive to consider the special case of
$\mathfrak{g}$ being the free Lie algebra $\mathbb{L}(E)$ on a
generator space $E$, in which case $U(\mathbb{L}(E))\simeq T(E)$. The
filtration $FT(E)=\{F_i T(E)\}_{i\in \mathbb{Z}}$ of $T(E)$ obtained in this way is
different from the canonical filtration of $T(E)$ by the
monomial degrees.
\end{example}

\rev{16}
{%
\begin{example}
 \label{StdFilt}
 More generally, given a unital algebra $\aA$ with of a generating set $S$, the canonical surjection $T(E)\twoheadrightarrow \aA$, where $E:=\Span(S)$,
 gives rise to an exhaustive filtration on $\aA$. Explicitly, such a filtration can be constructed inductively by setting 
 $F_i\aA = 0$ for all $i < 0$, $F_0 \aA:= \bbk$, $F_1 \aA:=\bbk + E$ and $F_k \aA:= F_1 \aA \cdot F_{k-1} \aA$ for all $k>1$.  
 
 The filtration obtained in this way is known as the \emph{standard filtration} of an algebra $\aA$ with respect to a generating set $S$ \cite[Section~1.6]{McConnellRobson}. It is a {\it D}-filtration whenever the space of generators $E$ is closed under the commutator.
 Its operadic analog, which is to be introduced in Section \ref{4 dny s Jarkou na chalupe.}, will be particularly useful for us later on.
\end{example}
}

\begin{example}
\label{Uz aby bylo jaro.}
Consider the tensor algebra $T(E)$ on the graded generator space $E=\Span(\Delta)$ with $\Delta$ being of an odd degree.
If $\{F_p T(E)\}_{p\geq 0}$ is the filtration of $T(E)$ as per Example
~$\ref{UEAex}$ and $\{T^q(E)\}_{q\geq 0}$ is the standard grading of the
same algebra by the 
monomial degrees, then
\[
F_p T(E) = \bigoplus_{k \leq p}T^{2k}(E).
\]  
Interpreting the associative unital 
algebra $\DG := T(E)/(\Delta^2)$ as an operad concentrated in arity
one, $\DG$-algebras are differential graded (dg) vector
spaces. \rev{17}{The operad $\DG$ will serve in Section~\ref{4 dny s
    Jarkou na chalupe.} as an example of a tight operad generated by
  an operation of arity one.}
\end{example}

An index-free description of the filtration data can be obtained as
follows. \rev{19}{Given a graded vector space $\aA = \bigoplus_{i \in {\mathbb
    Z}} \aA_i$, let $\Sub(\aA)$ denote the
modular lattice of all graded linear subspaces $V \subset \aA$, i.e.\
families $\{V_i \subset \aA_i\}_{i \in {\mathbb Z}}$ of linear
subspaces, with the meet
and the join operations being the componentwise  subspace intersection and the
componentwise subspace sum, respectively.} As a~poset, $\Sub(\aA)$ carries a natural
category structure.  

\rev{18}{The tensor product induces, for arbitrary graded vector spaces
$\aA'$ and $\aA''$, a~functor \[
\Sub(\aA') \times \Sub(\aA'') \to
\Sub(\aA' \ot \aA'')
\] 
defined by  
\[
\Sub(\aA') \times \Sub(\aA'') \ni
\{V'_i\}_i
\times \{V''_j\}_j \longmapsto  \big\{\bigoplus_{i+j=s}V'_i \ot
V''_j\big\}_s \in \Sub(\aA' \ot \aA'').
\] 
Likewise, any homogeneous linear map $\aA  \to \oB$ induces a
functor $\Sub(\aA) \to \Sub(\oB)$. 
As a consequence, any homogeneous ~$\bbk$-bilinear operation $\vartheta : \aA \ot \aA \to \aA$
 defines a functor
(denoted by the same symbol) 
\begin{equation}
\label{Bude vrtule?}
\vartheta : \Sub(\aA) \times \Sub(\aA) \to
\Sub(\aA)
\end{equation}
given as the composition $\Sub(\aA) \times \Sub(\aA) \to
\Sub(\aA \ot \aA)  \to \Sub(\aA)$ of the above functors.} An obvious 
analog of the induced functor~(\ref{Bude vrtule?})
holds also for arbitrary multilinear maps \hbox{$\vartheta : \aA^{\otimes n}
\to \aA$.}

\begin{example}
\label{Za chvili volam Taxovi.}
Let $\aA$ be a unital associative algebra. Its
multiplication \, $\cdot : \aA \ot \aA \to \aA$ induces on $\Sub(\aA)$
a monoid structure in the cartesian category ${\Cat}$ of small
categories, whose unit is $\Span(e)$, the linear span of the algebra 
unit $e \in \aA$.
Likewise, the graded commutator bracket $[-,-] : \aA\ot \aA \to \aA$ given by 
$[a',a''] := a' \cdot a'' - \sign{|a'||a''|}  a'' \cdot a'$ induces a symmetric
bifunctor (denoted by the same symbol) 
$[-,-] :  \Sub(\aA) \times \Sub(\aA) \to \Sub(\aA)$.
\end{example}

Let $\bbZ$ be the integers, $ + : \bbZ \times \bbZ \to \bbZ$ the
standard addition and $[-,-] : \bbZ \times \bbZ \to \bbZ$ the
operation given by  $[p,q] := p+q-1$ for $p,q \in \bbZ$. If we
consider $\bbZ$ as a small category with the category structure given
by its standard order, then $(\bbZ,+)$ and $(\bbZ, [-,-])$ are commutative
monoids in ${\Cat}$, with the units being $0$ and $1$ respectively. 
The following proposition 
refers to the structures introduced in Example~\ref{Za chvili volam Taxovi.}. 
\rev{20}{By a magma we mean a binary operation with no specified axioms.}

\begin{proposition}
\label{Snad zavola.}
Let $\aA$ be a unital associative algebra. 
The data of a {\it D}-filtration on $\aA$ is equivalent to that of a
functor $\filt: \bbZ \to \Sub(\aA)$ which is a~lax monoidal morphism $(\bbZ,\ +\ ) \to (\Sub(\aA),\ \cdot\ )$ and, simultaneously, a lax magma morphism $(\bbZ,[-,-]) \to (\Sub(\aA),[-,-])$ 
\end{proposition}

\begin{proof}
Given such a functor $\filt: \bbZ \to \Sub(\aA)$, let use denote 
$F_p\aA := \filt(p)$, $p \in
\bbZ$. The functoriality of $\filt$ is equivalent to $F_p \aA \subseteq
F_q\aA$ whenever $p \leq q$. Furthermore, for $\filt$ to be a lax morphism $(\bbZ,\ +\ ) \to (\Sub(\aA),\ \cdot\
)$ means, by definition, the existence of natural morphisms
\[
\filt(p) \cdot \filt(q) \longrightarrow \filt(p+q), \ p,q \in \bbZ
\ \hbox { and } \
\filt(\Span(e)) \longrightarrow \filt(0)
\] 
in $\Sub(\aA)$. In terms of the filtration determined by $\filt$ it
means that   $F_p\aA \cdot F_q\aA \subseteq F_{p+q}\aA$ and $e \in F_0\aA$.
By the same argument we verify that $\filt$ is a lax morphism 
$(\bbZ,[-,-]) \to (\Sub(\aA),[-,-])$ if and only if $[F_p \aA, F_q\aA]
\subseteq F_{p+q-1}\aA$.
\end{proof}

\subsection{Multifiltrations of operads}
We are going to generalize the notions recalled in the previous
subsection to the realm of $\bbk$-linear operads.
\rev{25}
{%
In essence, this amounts to introducing an $n$-ary analog of the $\mathbb{Z}$-indexed flag of subspaces~\eqref{FiltFlag} and modifying appropriately the corresponding compatibility condition for the algebra multiplication. 
Furthermore, extrapolating the notion of a {\it D}-filtration will require introducing an operadic analog of the standard commutator bracket.
}

While filtrations of associative algebras are defined in terms of the ordered monoid of integers~$\bbZ$, filtrations of operads are to be controlled by an operad $\MZ = \coll{\MZ(n)}$ of posets of integer-valued multiindices.
\rev{23}
{%
Specifically, for all $n\geq 1$, we set $\MZ(n):=\bbZ^n$ with the poset product order inherited from $\bbZ$.
That is, the set $\MZ(n)$ is partially ordered via
\[
(\rada {p'_1}{p'_n}) \preceq (\rada {p''_1}{p''_n})\ \hbox { if and only if }\
p'_i \leq p''_i\ \hbox { for each }\ 1 \leq i \leq n.
\]
}
With the obvious right permutation action of the symmetric groups $\Sigma_n$, the unit $(0) \in \MZ(1)$ and the structure operations 
\[
\circ_i: \MZ(m)
\times \MZ(n) \to \MZ(m+n-1),
\ n \geq 1,\ 1 \leq i \leq m, 
\] 
given for 
$(\Rada a1m) \in \MZ(m)$ and $(\Rada b1n) \in \MZ(n)$ by
\begin{equation}
\label{Vcera hodina s Terejem v termice!}
(\Rada a1m) \circ_i (\Rada b1m) := (\Rada a1{i-1},
\rada {b_1 + a_i}{b_n + a_i},
\Rada a{i+1}m),
\end{equation}
$\MZ$ forms an operad in the cartesian category 
${\Cat}$ of small categories.

Furthermore, for each $n\geq 1$, $\MZ(n)$ is a lattice with the join and the meet operations being
\[
  (\rada {p'_1}{p'_n})\vee (\rada {p''_1}{p''_n}) :=  (\rada {\max(p'_1, p''_1)}{\max(p'_n, p''_n)})
\]
and
\[
  (\rada {p'_1}{p'_n})\wedge (\rada {p''_1}{p''_n}) :=  (\rada {\min(p'_1, p''_1)}{\min(p'_n, p''_n)})
\]
respectively. 
Finally, for $\vec p = (\rada {p_1}{p_n})\in \MZ(n)$, we
denote the maximum of $\rada {p_1}{p_n}$ by $\max \vec p \in \bbZ$.

In what follows, given $\vec a = (\Rada a1m)$ and an index $1 \leq i
\leq m$,  we will use the shorthand notation
$\vec a_L := (\Rada a1{i-1})$, $\vec a_R := (\Rada a{i+1}m)$ and write $\vec a = (\vec a_L, a_i, \vec a_R)$. 
Now, for $\vec b = (\Rada b1n)$, equation~(\ref{Vcera hodina s Terejem v
  termice!}) reads 
\[
\vec a \circ_i \vec b := (\vec a_L, \vec b + a_i, \vec a_R),
\]
where $\vec b + a_i := (\rada {b_1 + a_i}{b_n + a_i})$.

\begin{remark}
The $\Cat$-operad $\MZ$ introduced above is an ordered
version of S.~Giraudo's combinatorial operad 
${\sf T}\! M$~\cite[Eq.~(3.1)]{guid} for $M$ being the
additive monoid of integers. More applications of that construction are presented in \cite{.}.
\end{remark}

\begin{definition}
\label{MFdef}
An (increasing) \emph{multifiltration} $F\oP =\{F_{\vec p}\oP(n)\}_{\vec p \in \MZ(n), n \geq 1}$ of a $\bbk$-linear operad $\oP$
is a collection of $\bbk$-subspaces  $F_{\vec p}\oP(n)\subset
\oP(n)$ indexed by the elements  $\vec p \in\MZ(n)$ for all $n
\geq 1$ subject to the following conditions:
\begin{itemize}
 \item[(i)] Monotonicity:
$F_{\vec p'}\oP(n) \subseteq F_{\vec p''}\oP(n)\
 \hbox { if } \vec p' \preceq \vec p ''$.
 \item[(ii)] Equivariance: 
$F_{\vec p}\oP(n)\cdot \sigma = F_{\vec p\cdot \sigma}\oP(n)$, 
where $\sigma\in \Sigma_n$ acts on $\vec p = (\rada {p_1} {p_n})$ via 
\[
\vec p\cdot \sigma = (\rada {p_{\sigma(1)}} {p_{\sigma(n)}}).
\]
\item[(iii)] Compositional compatibility: 
$F_{\vec a}\oP(m) \circ_i F_{\vec b}\oP(n) \subseteq F_{\vec a \circ_i
  \vec b}\oP(m+n-1)$
for all $\vec a \in \MZ(m)$,  $\vec b \in \MZ(n)$ and $1 \leq i \leq
m$. 
\item[(iv)] Unitality: 
If $e \in \oP(1)$ is the operadic unit, then $e\in F_{(0)}\oP(1)$.
\end{itemize}
We will routinely shorten $\{F_{\vec p}\oP(n)\}_{\vec p \in \MZ(n), n \geq 1}$ to $\{F_{\vec p}\oP(n)\}_{\vec p,n}$, whenever there is no danger of confusion.

A multifiltration $F\oP$ is said to be \emph{exhaustive} if 
\[
\oP(n)=\bigcup\limits_{\vec p \in \MZ(n)} F_{\vec p}\oP(n),
\ \hbox{ for all \ $n\geq 1$.}
\]
and is said to be \emph{bounded} if for any $n\geq 1$, there exists $\vec b_n \in \MZ(n)$
such that $F_{\vec p}\oP(n) = 0$ for all $\vec p < \vec b_n$.
\end{definition}
\rev{14}
{%
\begin{definition}
\label{MFiltMorph}
Let $\oP$ and $\oQ$ be $\bbk$-linear operads equipped with multifiltrations $F\oP$ and $G\oQ$ respectively.
Then a \emph{multifiltered operad morphism} $\phi: (\oP, F\oP)\to (\oQ, G\oQ)$ is an operad morphism $\phi: \oP \to \oQ$ such that
$\phi(F_{\vec p}\oP(n))\subseteq G_{\vec p}\oQ(n)$ for all $n\geq 1$ and $\vec p\in \MZ(n)$.
\end{definition}
}
All multifiltrations of $\oP$ form a poset $\MFilt(\oP)$,
where $F'\oP \preceq F''\oP$ if and only if
$F'_{\vec p}\oP(n) \subseteq F''_{\vec p}\oP(n)$ for all
$\vec p \in \MZ(n)$ and $n \geq 1$. Furthermore, for any two multifiltrations $F'\oP$, $F''\oP$, their componentwise intersections
\begin{align}
\label{FiltInt}
 G_{\vec p}\oP(n) := F'_{\vec p}\oP(n)\cap F''_{\vec p}\oP(n),\quad n\geq 1, \vec p\in \MZ(n)
\end{align}
and sums 
\begin{align}
\label{FiltSum}
 K_{\vec p}\oP(n) := F'_{\vec p}\oP(n) + F''_{\vec p}\oP(n),\quad n\geq 1, \vec p\in \MZ(n)
\end{align}
constitute well-defined multifiltrations. Both constructions admit obvious generalizations to the case of intersections and sums of arbitrarily large collections of multifiltrations and turn $\MFilt(\oP)$ into a lattice. Two trivial examples of multifiltrations are immediate.
\begin{example}
\label{TrivMF}
Any $\bbk$-linear operad $\oP$ possesses the 
multifiltration with $F_{\vec p}\oP(n) := \oP(n)$ for all \hbox{$\vec p
\in \MZ(n)$} and $n\geq 1$. It is the largest element of $\MFilt(\oP)$.
\end{example}

\begin{example}
A filtration $F\aA=\{F_i\aA\}_{i\in\mathbb{Z}}$ of an associative algebra $\aA$ is a multifiltration of $\aA$ treated as an operad concentrated in arity $1$. The qualities of being exhaustive and bounded transfer accordingly.
\end{example}

\begin{example}
\label{Ceho je to svedeni priznakem?}
To any $\bbk$-linear algebra $A$ one can associate  Giraudo's combinatorial operad
$\MA$, where ${\MA(n):=A^{\otimes n}}$ for all $n \geq 1$, and 
\begin{align*}
 \circ_i: \MA(m)\otimes \MA(n)&\to \MA(m+n-1)\\
 (a_1,\dots, a_m) \otimes(b_1,\dots, b_n)&\mapsto (a_1,\dots,
                                    a_{i-1},a_i\cdot b_1, 
\dots, a_i\cdot b_n,a_{i+1},\dots, a_m)
\end{align*}
for all $m,n\geq 1$ and $1\leq i\leq m$.
Any exhaustive (resp. bounded) filtration $FA = \{F_i A\}_{i\in \bbZ}$ of $A$ induces an exhaustive (resp. bounded) multifiltration $F\MA$ of $\MA$ given by 
\[
F_{\vec{p}}\MA(n)=F_{(p_1,p_2\dots p_n)}\MA(n):=
F_{p_1}A\otimes \cdots \otimes F_{p_n}A, \ \vec p \in \MZ(n),\ n \geq
1.
\]
\end{example}

\begin{example}
\label{Kdy uz bude ten vozik?}
Let $A$ be a commutative associative $\bbk$-algebra and $\End_A$ be
the endomorphism operad on the underlying vector space.
Generalizing Example ~\ref{DiffFilt}, we equip $\End_A$ with a
multi\-filtration
$F\End_A$ by 
taking $F_{(\Rada p1n)} \End_A(n)$ to be the space of $\bbk$-linear maps
\hbox{$O : A^{\otimes n} \to A$} that are differential operators of 
  order $p_i$ in the $i$-th variable, for each $1 \leq i \leq n$.
\end{example}

\begin{example}
\label{DiffDef}
Let $A\lev h\prav  := A \otimes \bbk\lev h\prav $. A `formally
deformed' version of the previous example is the multifiltration $F\End_{{A\lev h\prav }}$ of 
$\End_{{A\lev h\prav }}$ whose $(\Rada p1n)$-th component
is comprised of all
$\bbk\lev h\prav $-multi\-linear maps $O :  A\lev h\prav ^{\otimes n}
\to  A\lev h\prav $, $n \geq 1$, such that, for $\Rada a1n \in A$,
\begin{equation}
\label{Teprve tri platne discipliny.}
O(\Rada a1n)=O_0(\Rada a1n) + O_1(\Rada a1n)\cdot h +
O_2(\Rada a1n)\cdot  h^2  + \cdots,
\end{equation}
where $O_s(\Rada a1n)$ is a differential operator
of order $p_i + s$ with respect to the $i$-th argument, for each $1 \leq i \leq n$.
The multifiltrations $F\End_{{A}}$ and $F\End_{{A\lev h\prav }}$ of Examples~\ref{Kdy uz bude ten vozik?}
and~\ref{DiffDef} are bounded, but not exhaustive in general.
\end{example}
\rev{41}
{%
\begin{remark} 
The relation between the orders of differential operators
in~\eqref{Teprve tri platne discipliny.} and the powers of~$h$ is
motivated by the notion of $\IBL_\infty$-algebras, cf. Example~\ref{Podari se mi spravit Tereje?}.
The latter fits into the scheme of {\em binary QFT algebras\/} of
Park~\cite{Park}, which are, by definition, graded commutative associative algebras over $\bbk\lev h\prav$ equipped with a differential $D$ whose deviation $\Phi_D^{n+1}$ is
divisible by $h^n$, for each $n \geq 0$. 
\end{remark}
}

\rev{31}
{%
Similarly to the case of filtered algebras, multifiltrations can be passed along morphisms.
\begin{lemma}
\label{FiltByQuot}
Let\/ $\phi:\oP \to\oQ$ be a morphism of\/
$\bbk$-linear operads and
 $F\oP=\{F_{\vec p}\oP(n)\}_{\vec p,n}$ be a multifiltration of\/~$\oP$.
 Then the system  $\phi(F)\oQ =\{\phi(F)_{\vec p}\oQ(n)\}_{\vec p,n}$ of subspaces
 of\/ $\oQ = \{\oQ(n)\}_n$, where
 \[
  \phi(F)_{\vec p}\oQ(n):= \phi(F_{\vec p}\oP(n)),
 \]
 is a multifiltration of\/ $\oQ$. 
 Furthermore, if $\phi$ is surjective and $F\oP$ is exhaustive, then $\phi(F)\oQ$ is exhaustive as well.
\end{lemma}

\begin{proof}
The monotonicity and unitality of
$\phi(F)\oQ$ are automatic. 
The equivariance of $\phi(F)\oQ$ follows from $\phi$ being an operad
morphism and thus compatible with the symmetric group action. 
 To check the compositional compatibility with the $\circ_i$-products,
 assume that
 $\alpha \in \phi(F)_{\vec a}\oQ(m)$, $\beta \in \phi(F)_{\vec b}\oQ(n)$. By the 
definition of $\phi(F)\oQ$, there are
 $\alpha'\in F_{\vec a} \oP(m)$ and $\beta'\in F_{\vec b} \oP(n)$ such
 that $\phi(\alpha')=\alpha$ and $\phi(\beta') = \beta$, and thus
 \[
   \alpha\circ_i\beta = 
   \phi(\alpha')\circ_i\phi(\beta')=
   \phi(\alpha'\circ_i\beta')\in \phi(F_{\vec a\circ_i \vec b}\oP(m + n - 1)) =
   \phi(F)_{\vec a\circ_i \vec b}\oP(m + n - 1),
 \]
 yielding the desired inclusion $\phi(F)_{\vec a}\oQ(m)\circ_i \phi(F)_{\vec b}\oQ(n) \subset \phi(F)_{\vec a\circ_i \vec b}\oQ(m + n - 1)$.  
 Finally, it remains to observe that for a surjective $\phi$
 $$\bigcup\limits_{\vec p\in \MZ(n)}\phi(F)_{\vec p}\oQ(n) = \bigcup\limits_{\vec p\in \MZ(n)}\phi(F_{\vec p}\oP(n))=
 \phi\left(\bigcup\limits_{\vec p\in \MZ(n)}F_{\vec p}\oP(n)\right) = \phi(\oP(n)) = \oQ(n)$$
for any $n\geq 1$.
\end{proof}
}
\subsection{{\it D}-multifiltrations}
\noindent
\rev{26, 28}
{An additional property that the multifiltrations of Examples~\ref{Kdy uz bude ten vozik?}--\ref{DiffDef} possess
is implied by parts (iv)--(v) of Proposition~\ref{Zitra to snad alespon prestehujeme.} and concerns the behavior of multidifferential operators under commutation. To state it, for a $\bbk$-linear operad $\oP$ with the operadic compositions 
\begin{equation}
\label{Vcera jsem opet letel s Terejem!}
\circ_i: \oP(m)
\rev{21}{\otimes} \oP(n) \longrightarrow \oP(m+n-1),
\ n \geq 1,\ 1 \leq i \leq m,
\end{equation} 
we introduce a family of commutators 
\begin{equation}
\label{Jarka je na chalupe.}
[-,-]_{ij}: \oP(m)
\rev{21}{\otimes} \oP(n) \longrightarrow \oP(m+n-1),
\  1 \leq i \leq m, \ 1 \leq j \leq n,
\end{equation}
generalizing the standard commutator bracket of an associative algebra.
}
Namely, for ${a \in \oP(m)}$ and $b \in \oP(n)$, we set
\begin{equation}
\label{Citelne se ochladi.}
[a,b]_{ij} := (a \circ_i b) \cdot (\tau_{i-1,j-1} \times \id_{m+n-i-j+1})
- \sign{|a||b|} (b \circ_j a) \cdot(
\id_{i + j -1} \times \tau_{m-i,n-j}),  
\end{equation}
where $\tau_{i-1,j-1} \times \id_{m+n+i+j+1}) \in \Sigma_{m+n-1}$ is the
permutation that exchanges the first $i-1$ symbols with the next
$j-1$ symbols and leaves the remaining symbols unchanged. The action of 
${\id_{i + j -1}\times \tau_{m-j,n-i}}$ is determined in a similar way. The flow diagram on Figure~\ref{10} illustrates the~idea.

\begin{figure}
\[
\psscalebox{.7 .7}
{
\begin{pspicture}(20,-3)(2.74,2.2)
\psline[linecolor=black, linewidth=0.04](8.4,1.4682468)(7.0,0.06824677)(9.8,0.06824677)
\psline[linecolor=black, linewidth=0.04](8.4,-1.1317532)(7.0,-2.5317533)(9.8,-2.5317533)
\psline[linecolor=black, linewidth=0.04, arrowsize=0.16cm 2.0,arrowlength=1.4,arrowinset=0.0]{->}(8.4,-3.7317533)(8.4,-2.5317533)
\psline[linecolor=black, linewidth=0.04, arrowsize=0.16cm 2.0,arrowlength=1.4,arrowinset=0.0]{->}(8.4,-1.1317532)(8.4,0.06824677)
\psline[linecolor=black, linewidth=0.04](8.4,1.4682468)(9.8,0.06824677)
\psline[linecolor=black, linewidth=0.04](8.4,-1.1317532)(9.8,-2.5317533)
\pscustom[linecolor=black, linewidth=0.04]
{
\newpath
\moveto(15.2,2.0682468)
}
\pscustom[linecolor=black, linewidth=0.04]
{
\newpath
\moveto(15.0,0.46824676)
}
\pscustom[linecolor=black, linewidth=0.04]
{
\newpath
\moveto(7.086507,-1.493032)
\lineto(6.645576,-2.2181418)
\curveto(6.4251113,-2.5806968)(6.574916,-3.1702378)(7.6857257,-3.8511958)
}
\psline[linecolor=black, linewidth=0.04, arrowsize=0.16cm 2.0,arrowlength=1.4,arrowinset=0.0]{->}(7.08,-1.5317532)(8,0.06824677)
\psline[linecolor=black, linewidth=0.04, arrowsize=0.16cm 2.0,arrowlength=1.4,arrowinset=0.0]{->}(7.2,-3.1317532)(7.8,-2.5317533)
\psline[linecolor=black, linewidth=0.04, arrowsize=0.16cm 2.0,arrowlength=1.4,arrowinset=0.0]{->}(9.2,-3.7317533)(9.0,-2.5317533)
\psline[linecolor=black, linewidth=0.04](6.6,-3.7317533)(6.8,-3.5317533)
\psline[linecolor=black, linewidth=0.04, arrowsize=0.16cm 2.0,arrowlength=1.4,arrowinset=0.0]{->}(10.6,-3.7317533)(9.0,0.06824677)
\psline[linecolor=black, linewidth=0.04, arrowsize=0.16cm 2.0,arrowlength=1.4,arrowinset=0.0]{->}(8.4,1.4682468)(8.4,2.6682467)
\psline[linecolor=black, linewidth=0.04](15.604646,1.4682468)(17.004646,0.06824677)(14.204646,0.06824677)
\psline[linecolor=black, linewidth=0.04](15.604646,-1.1317532)(17.004646,-2.5317533)(14.204646,-2.5317533)
\psline[linecolor=black, linewidth=0.04, arrowsize=0.16cm 2.0,arrowlength=1.4,arrowinset=0.0]{->}(15.604646,-3.7317533)(15.604646,-2.5317533)
\psline[linecolor=black, linewidth=0.04, arrowsize=0.16cm 2.0,arrowlength=1.4,arrowinset=0.0]{->}(15.604646,-1.1317532)(15.604646,0.06824677)
\psline[linecolor=black, linewidth=0.04](15.604646,1.4682468)(14.204646,0.06824677)
\psline[linecolor=black, linewidth=0.04](15.604646,-1.1317532)(14.204646,-2.5317533)
\pscustom[linecolor=black, linewidth=0.04]
{
\newpath
\moveto(17.037224,-1.493032)
\lineto(17.418612,-2.1566439)
\curveto(17.609306,-2.48845)(17.47973,-3.027991)(16.51892,-3.6511958)
}
\psline[linecolor=black, linewidth=0.04, arrowsize=0.16cm 2.0,arrowlength=1.4,arrowinset=0.0]{->}(17.04646,-1.5317532)(16.204645,0.06824677)
\psline[linecolor=black, linewidth=0.04, arrowsize=0.16cm 2.0,arrowlength=1.4,arrowinset=0.0]{->}(16.804646,-3.1317532)(16.204645,-2.5317533)
\psline[linecolor=black, linewidth=0.04, arrowsize=0.16cm 2.0,arrowlength=1.4,arrowinset=0.0]{->}(14.804646,-3.7317533)(15.004646,-2.5317533)
\psline[linecolor=black, linewidth=0.04](17.404646,-3.7317533)(17.204645,-3.5317533)
\psline[linecolor=black, linewidth=0.04, arrowsize=0.16cm 2.0,arrowlength=1.4,arrowinset=0.0]{->}(13.404646,-3.7317533)(15.004646,0.06824677)
\psline[linecolor=black, linewidth=0.04, arrowsize=0.16cm 2.0,arrowlength=1.4,arrowinset=0.0]{->}(15.604646,1.4682468)(15.604646,2.6682467)
\rput(8.4,0.66824675){\Large $a$}
\rput[r](8.2,-.7){\Large $i$}
\rput[r](8.2,-3.2){\Large $j$}
\rput[r](15.4,-.7){\Large $j$}
\rput[r](15.44,-3.2){\Large $i$}
\rput(15.6,-1.9317533){\Large $a$}
\rput(8.4,-1.9317533){\Large $b$}
\rput(15.6,0.66824675){\Large $b$}
\rput[b](13,-1.53175324){\Large $-\ (-1)^{|a|\cdot |b|}$}
\rput[br](6.4,-3.533){\Large $\tau_{i-1,j-1} \times \id_{m+n-i-j+1}$}
\rput[bl](17.6,-3.533){\Large $\id_{i+j-1} \times \tau_{m-i,n-j}$}
\end{pspicture}
}
\]
\caption{\label{10}%
The operadic commutator $[a,b]_{ij}$.}
\end{figure}
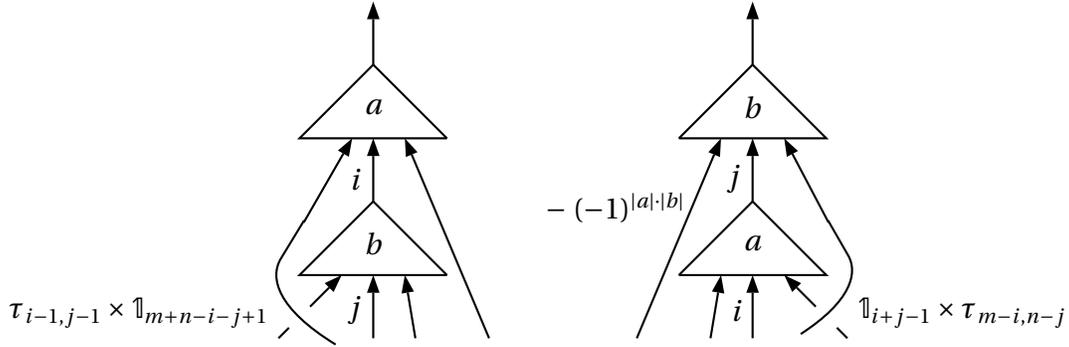
As in~(\ref{Bude vrtule?}),
operations~(\ref{Citelne se ochladi.}) induce bifunctors (denoted by
the same \hbox{symbol})
\begin{equation}
\label{bund}
[-,-]_{ij}: \Sub(\oP(m))
\times \Sub(\oP(n)) \longrightarrow \Sub(\oP(m+n-1)),
\  1 \leq i \leq m, \ 1 \leq j \leq n.
\end{equation}
The property we wish to formalize stipulates their behavior in terms of the indexing operad $\MZ$.
To this end, we define
\begin{equation}
\label{on}
[-,-]_{ij} :\MZ(m)
\times \MZ(n) \to \MZ(m+n-1),\ 1\leq i \leq m,\ 1 \leq j \leq n,
\end{equation}
by setting
\[
[\vec a , \vec b]_{ij} :=
(\vec b_L + a_i, \vec a_L + b_j,a_i + b_j -1,\vec b_R + a_i, \vec a_R+b_j),
\]
for $\vec a = (\vec a_L, a_i, \vec a_R) \in \MZ(m)$ and $\vec b = 
(\vec b_L, b_j, \vec b_R) \in \MZ(n)$.
Note that $(\MZ(1),+,[-,-]_{11})$ is equivalent to $(\bbZ,+,[-,-])$ in the sense of Proposition~\ref{Snad zavola.}.

\begin{definition}
\label{Bojim se jestli jeste budu letat.}
We say that a multifiltration $F\oP=\{F_{\vec p}\oP(n)\}_{\vec p,n}$ of a $\bbk$-linear operad $\oP$
is a \mbox{{\it D}-\emph{multifiltration} if}
$$
[F_{\vec p}\oP(m),F_{\vec q}\oP(n)]_{ij} \subseteq  F_{[\vec p, \vec q]_{ij}}\oP(m+n-1)
$$
for all $\vec p \in \MZ(m)$,  $\vec q \in \MZ(n)$, $1 \leq i \leq
m$, $1 \leq j \leq n$. 
\end{definition}
\rev{26}
{%
\begin{example}
Let $FA$ be a {\it D}-filtration of $A$. Then, trivially, it is a {\it D}-multifiltration of $A$ regarded as an operad concentrated in arity $1$.
 More generally, the multifiltration of Giraudo's operad $\MA$ inherited from $FA$ as per Example~\ref{Ceho je to svedeni priznakem?} is a {\it D}-multifiltration.
\end{example}

\begin{example}
The multifiltrations of Examples~\ref{Kdy uz bude ten vozik?} and~\ref{DiffDef} are {\it D}-multifiltrations.  
\end{example}
The property of being a {\it D}-multifiltration is preserved under taking sums and intersections of {\it D}-multifiltrations.
Furthermore, as a corollary to Lemma~\ref{FiltByQuot}, we get
\begin{lemma}
\label{DMFiltMorph}
For any \/ $\bbk$-linear operad morphism $\phi: \oP\to \oQ$, the image $\phi(F)\oQ$ of a {\it D}-multifil\-tra\-tion $F\oP$ is a {\it D}-multifil\-tra\-tion. 
\end{lemma}
}

\subsection{Saturated multifiltrations}
\noindent
\rev{29}
{%
In general, the benefit of having a non-trivial filtration on an
algebra, or a multifiltration on a linear operad, lies in gaining a mean of better understanding its structure in terms of some smaller constituents and their mutual relation. For multifiltrations, the complexity of the indexing operad $\MZ$ encoding such a decomposition is noticeably higher than that of the linearly ordered monoid $\mathbb{Z}$ governing the case of ordinary associative algebra filtrations. With that in mind, we would like to single out a certain class of multifiltrations characterized by relatively simple underlying combinatorics. 

\begin{definition}
A multifiltration $F\oP=\{F_{\vec p}\oP(n)\}_{\vec p ,n}$ is \emph{saturated} if 
\begin{equation}
\label{Otepli se jeste?}
F_{\vec p'}\oP(n) \cap F_{\vec p''}\oP(n) \subseteq
F_{\vec p' \wedge \vec p''}\oP(n), \ \hbox { for all } \ \vec p', \vec p'' \in \MZ(n).
\end{equation}
\end{definition}
Due to monotonicity~(i) of Definition~\ref{MFdef}, 
inclusion~(\ref{Otepli se jeste?}) amounts to the equality
\[
F_{\vec p'}\oP(n) \cap F_{\vec p''}\oP(n) =
F_{\vec p' \wedge \vec p''}\oP(n), \ \hbox { for all } \ \vec p', \vec p'' \in \MZ(n).
\]
For $n=1$, the defining condition automatically holds. In that regard, the notion of being saturated gets trivial in the realm of associative algebras.
\begin{example}
All multifiltrations in Examples~\ref{TrivMF}--\ref{DiffDef} are
saturated.
\end{example}

\begin{example}
As an example of a non-saturated multifiltration, let $\oP$ be the $\bbk$-linear operad with $\oP(2) = \Span(\alpha, \beta, \gamma)$
and $\oP(n) = 0$ for any $n\neq 2$. That is, it is an operad generated by three binary operations with the vanishing pairwise- and self-compositions. Then the multifiltration defined by setting $F_{(0,1)}\oP(2) := \Span(\alpha, \gamma)$, $F_{(1,0)}\oP(2) := \Span(\beta, \gamma)$, $F_{\vec p}\oP(2) := \oP(2)$ for all $\vec p\geq (1, 1)$ and 
$F_{\vec p}\oP(2) := 0$ for all $\vec p \leq (0, 0)$ is not saturated.
More natural examples of non-saturated multifiltration are to appear in Section~\ref{4 dny s Jarkou na chalupe.}.
\end{example}
}

\revnm{new}
{%
 An elementary check shows that the property of being saturated is not preserved under taking sums and images of multifiltrations, but is preserved under taking arbitrary intersections.
That, combined with the fact that for any $\bbk$-linear operad $\oP$ the largest (trivial) multifiltration of $\oP$ as per Example~\ref{TrivMF} is saturated, confirms the validity of the following
}
\begin{definition}
\label{nePojedu do Jicina?}
The {\em saturation\/} $\oF\oP$ of a multifiltration $F\oP$ is the
smallest saturated multifiltration containing $F\oP$. 
\end{definition}
\rev{29}
{%
Specifically, $\oF\oP$ can be described as the intersection of all saturated multifiltrations containing $F\oP$. 
In such a case, by \cite[Theorem 2]{Everett}, the correspondence $F\oP\mapsto \oF\oP$ makes up a \emph{closure operator} on the lattice $\MFilt(\oP)$. 
That is, it enjoys the following formal properties
\begin{subequations}
 \begin{align}
 F\oP &\preceq \oF\oP \\
 \overline{\oF}\oP &= \oF\oP 
\end{align}
and that $F'\oP\preceq F''\oP$ implies 
\begin{align}
\label{SatMono}
\overline{F'}\oP\preceq \overline{F''}\oP .
\end{align}
\end{subequations}
As is the case for any closure operator, it is a functor $\MFilt(\oP)\to \SMFilt(\oP)$ left adjoint to the inclusion $\SMFilt(\oP)
\hookrightarrow \MFilt(\oP)$, where $\SMFilt(\oP)$ is the sublattice of $\MFilt(\oP)$ consisting of all saturated multifiltrations of $\oP$, and both lattices are regarded as categories with $x \to y$ if and only if $x \preceq y$.
}

Our next task is to identify the saturation $\oF\oP$
with the  colimit of suitable intermediate steps.
To this end, given a multifiltration $F\oP=\{F_{\vec p}\oP(n)\}_{\vec p,n}$
of a $\bbk$-linear operad $\oP$, we define its {\em presaturation\/} as 
the family of subspaces 
$\{\pres{F}_{\vec p}\oP(n)\}_{\vec p,n}$ where
\begin{equation}
\label{SatDef}
\pres{F}_{\vec p}\oP(n):= 
 \sum_{k\geq 1}
 \sum\limits_{\substack{\vec p_1,\dots,\vec p_k 
\\ \vec p_1\wedge \dots\wedge \vec p_k = \vec p}}
 F_{\vec p_1}\oP(n)\cap\dots\cap F_{\vec p_k}\oP(n),\ \hbox { \ for $n \geq 1$.}
\end{equation}

\revnm{restructured}
{%
\begin{lemma}
\label{Je to z chlastu?}
 Let $F\oP$ be a multifiltration of a $\bbk$-linear operad $\oP$ and $F'\oP$ be its presaturation.
\begin{enumerate} 
 \item[(i)] 
 The family $\pres{F}\oP =\{\pres{F}_{\vec p}\oP(n)\}_{\vec p,n}$
is a multifiltration of\/ $\oP$. 
 \item[(ii)] 
 If $F\oP$ is a {\it D}-multifiltration, then so is $F'\oP$.  
\end{enumerate}
\end{lemma}

\begin{proof}
(i) Monotonicity, unitality and equivariance of $\{\pres F_{\vec p}\oP(n)\}_{\vec p,n}$
 are straightforward. It remains to prove that for 
any $\vec p\in \MZ(m), \vec q\in \MZ(n)$, $1\leq i\leq n$ and $1\leq
j\leq m$,
 \begin{align}
 \label{Sat1}
  \pres{F}_{\vec p} \oP(m)\circ_i & \pres{F}_{\vec q} \oP(n)\subset 
  \pres{F}_{\vec p \circ_i \vec q}\oP(m+n-1).
 \end{align}
By $\bbk$-linearity of the $\circ_i$-compositions it suffices to show that 
\[
(F_{\vec p_1}\oP(m)\cap\dots\cap F_{\vec p_k}\oP(m))
\circ_i
(F_{\vec q_1}\oP(n)\cap\dots\cap F_{\vec q_l}\oP(n))
\subset 
\pres{F}_{\vec p \circ_i \vec q}\oP(m+n-1)
\]
for all  $\vec p_1,\dots,\vec p_k \in \MZ(m)$,
$\vec q_1,\dots,\vec q_l \in \MZ(n)$,
 such that 
$\vec p_1\wedge \dots\wedge \vec p_k = \vec p$ and
$\vec q_1\wedge \dots\wedge \vec q_l = \vec q$. 
To this end, observe first that
 \begin{equation}
\label{Zitra prvni davka ockovani.}
\vec p \circ_i \vec q =  
(\vec p_1\wedge \dots \wedge \vec p_k)\circ_i(\vec q_1 \wedge \dots \wedge  \vec
  q_l)  
= \bigwedge_{u,v}  (\vec p_u\circ_i \vec q_v), 
\end{equation}
where in the rightmost term $u$ and $v$ run from $1$ to $k$, and from $1$ to $l$, respectively.
Then
\begin{align*}
(F_{\vec p_1}\oP(m)\cap\dots\cap F_{\vec p_k}\oP(m))
\circ_i &
(F_{\vec q_1}\oP(n)\cap\dots\cap F_{\vec q_l}\oP(n))
\subset \bigcap_{u,v} (F_{\vec p_u}\oP(m)\circ_i F_{\vec q_v}\oP(n)) 
\\
\subset \bigcap_{u,v} F_{\vec p_u\circ_i\vec q_v}\oP(m+n-1)
  & \subset \pres{F}_{\bigwedge_{u,v}  (\vec p_u\circ_i \vec q_v)}\oP(m+n-1)=
  \pres{F}_{\vec p\circ_i \vec q}\oP(m+n-1).
 \end{align*}
 
 (ii) Let $F\oP$ be a {\it D}-multifiltration. 
 Then an argument analogous to the one used in part (i) with the identity
\[
[\vec p , \vec q]_{ij} = 
[\vec p_1\wedge \dots \wedge \vec p_k, \vec q_1 \wedge \dots \wedge  \vec
  q_l]_{ij} = 
 \bigwedge_{u,v}     [\vec p_u, \vec q_v]_{ij}
\]
in place of~\eqref{Zitra prvni davka ockovani.} shows that 
 \begin{align}
 \label{Sat2}
  [\pres{F}_{\vec p} \oP(m), & \pres F_{\vec q} \oP(n)]_{ij}\subset 
  \pres{F}_{[\vec p,\vec q]_{ij}}\oP(m+n-1).
 \end{align}
  for all $n \geq 1$, $\vec p, \vec q \in  \MZ(n)$.
  \end{proof}
}

Given a multifiltration  $F\oP=\{F_{\vec p}\oP(n)\}_{\vec p,n}$,
let $F^{(s)} \oP=\{F^{(s)}_{\vec p}\oP(n)\}_{\vec p,n}$ denote the multifiltration obtained 
by iterating the presaturation $s$ times. 

\begin{proposition}
\label{porad nejake zvuky}
The components of the saturation  of  $F\oP$ 
are given by
\begin{equation}
\label{Jarka to zas vsechno popletla.}
\overline F_{\vec p}\oP(n) := \bigcup_{s \geq 1} F^{(s)}_{\vec
  p}\oP(n), \ \vec p \in \MZ(n),\ n \geq 1.
\end{equation}
In other words, $\oF\oP = \colim F^{(s)}\oP$ in the poset\/ $\MFilt(\oP)$. 
\end{proposition}

\begin{proof}
The colimit of $F^{(s)}\oP$ is a saturated multifiltration by Lemma~\ref{Je to z chlastu?}.
Each saturated multi\-filtration containing  
$F\oP$ clearly contains also $F'\oP$ and thus also
$F^{(s)}\oP$ for each $s \geq 1$. It therefore contains also the union  
of iterated presaturations, thus~\eqref{Jarka to zas vsechno
  popletla.} indeed defines the smallest saturated multifiltration containing  
$F\oP$. 
\end{proof}
\rev{30}{%
\begin{remark}
\label{MZdist}
Since the definition \eqref{SatDef} of a component $F'_{\vec p}\oP(n)$ of the presaturation $F'\oP$ does not involve the operad structure of $\oP$, a priori it might not be obvious that ${F'\oP=\{F'_{\vec p}\oP(n)\}_{\vec p,n}}$, and consequently the saturation $\oF\oP$, would comprise a well-defined non-trivial multifiltration. As follows from the proof of Lemma~\ref{Je to z chlastu?}, it does happen to be the case due to the distributive identity \eqref{Zitra prvni davka ockovani.} noteworthy in that it connects the mutually independent lattice structures of the components $\MZ(n)$ for all $n\geq 1$ with the operad structure of $\MZ$.
\end{remark}
}
\begin{corollary}
\label{SatMorph}
\rev{31}
{%
 Let\/ $\phi:\oP \to \oQ$ be a morphism of\/
$\bbk$-linear operads and
 $F\oP=\{F_{\vec p}\oP(n)\}_{\vec p,n}$ be a~multifiltration of\/~$\oP$.
 Then 
 \begin{align}
  \label{Kdy si zase skocim s Terjem?}
  \phi(\overline{F})\oQ \preceq \overline{\phi(F)}\oQ.
 \end{align}
 The equality $\phi(\overline{F})\oQ = \overline{\phi(F)}\oQ$ holds if and only if $\phi(\overline{F})\oQ$ is saturated.
}
\end{corollary}
\begin{proof}
From the definition of the presaturation~\eqref{SatDef},
for any $n\geq 1$ and $\vec p\in \MZ(n)$,
 \begin{align*}
\phi(F')_{\vec p}\oQ(n) =  \phi(F'_{\vec p}\oP(n))&=   
  \phi\left(
   \sum_{k\geq 1}
    \sum\limits_{\substack{\vec p_1,\dots,\vec p_k \\ \vec p_1\wedge \dots\wedge \vec p_k = \vec p}}
    F_{\vec p_1}\oP(n)\cap\dots\cap F_{\vec p_k}\oP(n)\right)\\
     &\subset 
    \sum_{k\geq 1}
    \sum\limits_{\substack{\vec p_1,\dots,\vec p_k \\ \vec p_1\wedge \dots\wedge \vec p_k = \vec p}}
    \phi(F_{\vec p_1}\oP(n))\cap\dots\cap \phi(F_{\vec p_k}\oP(n))
    =\phi(F)'_{\vec p}\oQ(n).
 \end{align*}
 The inclusion $\phi(\overline{F})\oQ \preceq \overline{\phi(F)}\oQ$ now follows from~\eqref{Jarka to zas vsechno popletla.}.
 
 Now, let $F\oP$ and $\phi$ be such that $\phi(\overline{F})\oQ$ is saturated. We have $\phi(F)\oQ\preceq \phi(\overline{F})\oQ$.
 Upon taking the saturation on both sides, we get
 \[
  \overline{\phi(F)}\oQ\preceq \overline{\phi(\overline{F})}\oQ =\phi(\overline{F})\oQ
 \]
 and the last claim follows.
 \end{proof} 
 In general, the inclusion~\eqref{Kdy si zase skocim s Terjem?} could be proper as illustrated by Example~\ref{Bude pocasi na Tereje?}.

\subsection{Stabilized multifiltrations}
The following notion is meant to single out multifiltrations whose components, in a given arity $n$, 
are eventually constant with respect to a natural order on $\MZ(n)$.
\begin{definition}
Let $n \geq 1$. We say that a multifiltration
$F\oP=\{F_{\vec p}\oP(n)\}_{\vec p,n}$  
{\em stabilizes in arity $n$ at $N$\/} for some integer $N$ if, for
 each $\vec p \in \MZ(n)$, 
\begin{equation}
\label{eq:4}
F_{\vec p}\oP(n) =    F_{\vec p \land (N, \dots,N)}\oP(n). 
\end{equation}
In particular, $F_{\vec p}\oP(n)=F_{(N, \dots,N)}\oP(n)$ for any $\vec p\geq (N, \dots, N)$. 
The multifiltration  $F\oP$ is said to be {\em
  stabilized\/} if it stabilizes in each arity $n \geq 1$ at some $N_n \geq 1$. 
\end{definition}
\begin{remark}
The term above is not to be confused with the notion of a \textit{stable filtration} (with respect to an ideal) of a module or an associative algebra.
\end{remark}
By monotonicity, condition~(\ref{eq:4}) implies that ${F_{\vec p}\oP(n) \subset F_{(N, \dots,N)}\oP(n)}$ for any $\vec p\in \MZ(n)$. 
If $F\oP$ is saturated, then the converse implication holds as well. 
The utility of this condition is to become more apparent in the next section, where a certain natural class of stabilized multifiltrations will be constructed. Otherwise, it is largely due to the following
\revnm{redone}
{%
\begin{proposition}
\label{divna vyrazka}
Suppose that a multifiltration
$F\oP=\{F_{\vec p}\oP(n)\}_{\vec p,n}$  stabilizes in arity $n$ at
$N$. Then, 
\begin{enumerate}
 \item[(i)]
 the presaturation $F'\oP$ stabilizes in arity $n$ at $N$ as well, and
 \item[(ii)] 
for each $\vec p = (\Rada a1n) \in \MZ(n)$, we have
\begin{align}
\label{Paleni na zadech a mezi prsty.}
\pres{F}_{\vec p}\oP(n) = F_{(a_1,N,\ldots,N)}\oP(n)& \cap F_{(N,a_2,N,\ldots,N)}\oP(n) \cap \cdots \cap
 F_{(N,\ldots,N,a_{n-1},N)}\oP(n) \cap  F_{(N,\ldots,N,a_{n})}\oP(n).
\end{align}
\end{enumerate}
\end{proposition}

\begin{proof}
  (i) Let $\vec p\in \MZ(n)$ and $\vec p_1,\dots,\vec p_k$ be such that
  ${\vec p_1\wedge\dots\wedge \vec p_k = \vec p}$.  
  By stabilization of $F\oP(n)$ at $N$, we have
  \[
   F_{\vec p_1}\oP(n)\cap \dots \cap F_{\vec p_k}\oP(n)= F_{\vec p_1 \wedge (N, \dots, N)}\oP(n)\cap \dots \cap F_{\vec p_k\wedge (N, \dots, N)}\oP(n)
  \]
  The definition of presaturation ~\eqref{SatDef} implies now that 
  ${F'_{\vec p}\oP(n)=F'_{\vec p \wedge (N,\dots, N)}\oP(n)}$.
  
  (ii) We need to show that 
  \begin{align}
    \pres{F}_{\vec p}\oP(n) \subseteq F_{(a_1,N,\ldots,N)}\oP(n)& \cap F_{(N,a_2,N,\ldots,N)}\oP(n) \cap \cdots \cap F_{(N,\ldots,N,a_{n-1},N)}\oP(n) \cap  F_{(N,\ldots,N,a_{n})}\oP(n),
  \end{align}
  as the inclusion in the opposite direction follows readily from the definition of the presaturation $F'\oP$. 
  To this end, let $\vec p_i = (u^i_1, u^i_2,\dots, u^i_n)\in \MZ(n)$ for $i=1,2,\dots, k$ 
  be such that ${\vec p_1 \wedge \dots \wedge \vec p_k=\vec p}$.   
  In particular, $a_j=\min\limits_{i=1,\dots, k} u^i_j$ for each $j=1,2,\dots, n$.
  We have
  \begin{align*}
   F_{\vec p_1}\oP(n)\cap \dots \cap F_{\vec p_k}\oP(n) &= 
   F_{\vec p_1 \wedge (N, \dots, N)}\oP(n)\cap \dots \cap F_{\vec p_k\wedge (N, \dots, N)}\oP(n)\\
   &\subseteq \bigcap\limits_{i=1}^{k} F_{(u^i_1, N, \dots, N)}\oP(n)= 
   F_{(a_1, N,\dots, N)}\oP(n),
  \end{align*}
  where the first equality is due to stabilization and the subsequent inclusion follows by monotonicity. In a similar way, 
  we obtain $F_{\vec p_1}\oP(n)\cap \dots \cap F_{\vec p_k}\oP(n)\subseteq F_{(N,a_2,N,\ldots,N)}$ and so on. The desired inclusion then follows from~\eqref{SatDef}.
 \end{proof}
}
 
\begin{corollary}
\label{nejde e-mail}
 Under the hypothesis of Proposition~\ref{divna vyrazka}, $\oF\oP(n) = F'\oP(n)$.
\end{corollary}
 Indeed, by ~\eqref{Paleni na zadech a mezi prsty.}, for any $j=1,2,\dots,n$, we have
 ${F'_{(N,\dots, a_j,\dots, N)}\oP(n)=F_{(N,\dots, a_j, \dots, N)}\oP(n)}$.
 Then by another application of the same equality, ${F''_{\vec p}\oP(n)=F'_{\vec p}\oP(n)}$ and the claim follows from Proposition~\ref{porad nejake zvuky}.

\subsection{Standard {\it D}-multifiltrations}
\label{DmultSect}
\noindent
\rev{27,\\ 32a,c, 34}
{%
We are about to introduce a construction of a multifiltration that can be associated with a $\bbk$-linear operad with a specified choice of generators. 
This can be thought of as an operadic analog of the standard filtration of an associative algebra as per Example~\ref{StdFilt}
with two enhancements. First, the construction to be discussed here
results in a~{\it D}-multifiltration in the sense of
Definition~\ref{Bojim se jestli jeste budu letat.}.  
Analogously to the case of associative algebras, this requires the
space of the chosen generators of a $\bbk$-linear operad $\oP$ 
in arity $1$ be closed under the commutator. 
If there are no generators of arity $1$, such a condition is void, yet
the combinatorics of $\MZ(n)$ for $n\geq 2$ still results in a
non-trivial {\it D}-multifiltration of $\oP$. 
Second, the combinatorics of this {\it D}-multifiltration is simplified by taking its saturation. 
This does not have its direct analog in the realm of associative algebras, since all algebra filtrations are saturated in the sense of our Definition~\ref{SatDef}. 
}

Let  $\oP$ be an $\bbk$-linear operad generated by a $\Sigma$-module $E=\{E(n)\}_{n\geq 1}$. 
For the rest of the section, we will assume that $\oP(0)=0$ and $E(1)$ is closed under the bracket $[-,-]_{11}$.
In particular, this encompasses the case when $\oP$ is simply-connected, i.e. when $\oP(1)$ is one-dimensional and is spanned by the operad unit. 
The \emph{prestandard {\it D}-multifiltration of $\oP$ with respect to $E$} is defined as the smallest {\it D}-multifiltration 
$G\oP = \{G_{\vec p}\oP(n)\}_{\vec p,n}$ of\/ $\oP$ such that $E(n)\subset G_{(1,\dots,1)}\oP(n)$ for all ${n\geq 1}$.
\rev{32}{%
Notice that such a multifiltration of $\oP$ automatically  satisfies \hbox{$E(n) \cap  G_{\vec p}\oP(n) = 0$} for \hbox{$\vec p \prec
(\rada 11)$}.} 
Explicitly, $G\oP$ can be described as follows.
\begin{itemize}
 \item[(i)] 
 For $n=1$, we define $G_{(0)}\oP(1) := \bbk\cdot e$, where $e\in \oP(1)$ is the operad unit.  
 Next, we set 
 \[
  G_{(1)}\oP(1):=\bbk\cdot e+E(1).  
 \]
  For all $p>1$, we inductively define
  \[
  G_{(p)}\oP(1):=G_{(p-1)}\oP(1)+G_{(1)}\oP(1)\circ_1 G_{(p-1)}\oP(1).
  \] 
 \item[(ii)]  
 For $n > 1$ and all $\vec p \in \MZ(n)$ such that \hbox{$\vec p \not\succeq
 \jednas$}, as well as for $n=1$ and all $\vec p\in \MZ(1)$ such that $\vec p\prec (0)$, 
 we put $G_{\vec p}\oP(n) := 0$. 
 \item[(iii)] 
 For 
 $n\geq 2$
 and $\vec p\succeq \jednas$ we proceed inductively by setting
 \begin{equation}
 \label{StdOpFilt}
 G_{\vec p} \oP(n) :=   
  E(n) + \sum_{i, \vec p', \vec p'',k,l,\sigma} 
 (G_{\vec p'}\oP(k) \circ_i G_{\vec p''}\oP(l)) \cdot \sigma
  +\sum_{i, j, \vec p', \vec p'',k,l,\sigma} [G_{\vec p'}\oP(k), G_{\vec p''}\oP(l)]_{ij} \cdot \sigma.
 \end{equation}
 \end{itemize}
 Here, both summations in~(\ref{StdOpFilt}) run over 
 $k,l \geq 1$ 
 such that
 $k+l = n+1$.  
 The first sum is taken over all
 $1 \leq i \leq k$, $\vec p' \in \MZ(k)$, $\vec p''\in \MZ(l)$ and
 $\sigma \in \Sigma_n$ such that
 $(\vec p' \circ_i \vec p'')\sigma \preceq \vec p$.
 The second one goes over all $1 \leq i \leq k$, $1 \leq j \leq l$, $\vec p'\in \MZ(k)$,
 $\vec p'' \in \MZ(l)$ and $\sigma \in \Sigma_n$ such that
 $[\vec p', \vec p'']_{ij} \sigma \preceq \vec p$. 
 We will routinely omit an explicit reference to the generating collection $E$, whenever there is no danger of confusion. 

\rev{36}
{%
\begin{proposition}
\label{StdOpFiltProp}
 Let the family $G\oP = \{G_{\vec p}\oP(n)\}_{\vec p,n}$ be as described above. 
 Then
\begin{enumerate}
 \item [(i)] 
 $G\oP$ is a {\it D}-multifiltration.
 \item [(ii)]
 If $\oP$ is simply-connected, then in arity $n$, it stabilizes at $N=n-1$.
 \item[(iii)]
 It is the smallest {\it D}-multifiltration such that $E(n)\subset G_{(1,\dots,1)}\oP(n)$ for all $n\geq 1$.
\end{enumerate}
\end{proposition}  
}
\begin{proof}
(i) Monotonicity and unitality of $G\oP$ are immediate. 
The equivariance follows from the compatibility of the operadic
compositions with the symmetric group actions. 

The property of being a {\it D}-multifiltration in arity $n=1$ follows by a simple inductive argument using \eqref{NCLeibniz}.
Now, given $n\geq 2$, $\vec a \in \MZ(x)$,  $\vec b \in \MZ(y)$ with $x+y=n+1$ and
$1 \leq s \leq k$, equation~\eqref{StdOpFilt} for
$\vec p := \vec a \circ_s \vec b$, yields
\begin{align*}
  G_{\vec a\circ_s \vec b} \oP(n) &= 
  E(n) + \sum_{i, \vec p', \vec p'',k,l,\sigma} 
(G_{\vec p'}\oP(k) \circ_i G_{\vec p''}\oP(l)) \cdot \sigma
  +\sum_{i, j, \vec p', \vec p'',k,l,\sigma} [G_{\vec p'}\oP(k), G_{\vec
         p''}\oP(l)]_{ij}\cdot \sigma
\\
&\supset G_{\vec a}\oP(x) \circ_s G_{\vec b}\oP(y),
  \end{align*}
  where the inclusion follows upon taking $k = x$, $l=y$, 
$\vec p' = \vec a$,
  $\vec p'' = \vec b$, $i = s$ and $\sigma = \id_n$, the unit 
of~$\Sigma_n$, in the first sum. 
In a similar way we get,
using~\eqref{StdOpFilt} again,
\[
G_{[\vec a, \vec b]_{st}} \oP(n) 
\supset [G_{\vec a}\oP(x), G_{\vec b}\oP(y)]_{st}
\]
for all $x,y,\vec a,\vec b,s,t$ for which the above expression makes sense.

(ii) Since, by definition, $G_{\vec
  p}\oP(n) = 0$ if $\vec p \not\succeq \jednas$, we may assume that
$\vec p \succeq \jednas$. 
Since $\oP$ is assumed to be simply-connected, the stabilization at arity $n=1$ is clear. 
For a given $n \geq 2$ and $\vec p \in
\MZ(n)$ denote ${\hat p := \vec p \land (n-1,\ldots,n-1)}$. 
As $n\geq 2$ by assumption, $\vec p \succeq \jednas$ if and only if $\hat p
\succeq \jednas$. We must therefore prove that
\begin{equation}
\label{porad napjata kuze}
G_{\vec
  p}\oP(n) \subset G_{\hat
  p}\oP(n) \ \hbox {for all $n \geq 2$ and $\vec p \in \MZ(n)$},
\end{equation}
because the opposite inclusion and thus the equality 
would follow from the monoticity.
We proceed by the induction on the arity.

The base case $n=2$ is implied by $G_{\vec p}\oP(2)=E(2) =
G_{\hat p}\oP(2)$. For the induction step consider
formula~\eqref{StdOpFilt} defining  $G_{\vec p}\oP(n)$ and compare it
with the formula
\begin{equation}
\label{StdOpFiltbis}
 G_{\hat p} \oP(n) :=   
  E(n) + \sum_{i, \vec q', \vec q'',k,l,\sigma} 
(G_{\vec q'}\oP(k) \circ_i G_{\vec q''}\oP(l)) \cdot \sigma
  +\sum_{i, j, \vec q', \vec q'',k,l,\sigma} [G_{\vec q'}\oP(k), G_{\vec q''}\oP(l)]_{ij} \cdot \sigma
 \end{equation}
in which $(\vec q'\circ_i \vec q'')\sigma \preceq \hat p$ in the first sum, and 
$[\vec q', \vec q'']_{ij}\sigma \preceq \hat p$ in the second
one. Since the term $E(n)$ occurs in both formulas, it remains to prove that
all terms of the first and the second sum of~\eqref{StdOpFilt} occur
also in~(\ref{StdOpFiltbis}). 

Consider an arbitrary term $(G_{\vec p'}\oP(k) \circ_i G_{\vec p''}\oP(l)) \cdot
\sigma$ of the first sum in~\eqref{StdOpFilt} in which, of course, 
\hbox{$(\vec p'\circ_i
\vec p'')\sigma \preceq \vec p$}. By the induction assumption, 
\begin{equation}
\label{Opet strasne pocasi na Safari.}
(G_{\vec p'}\oP(k) \circ_i G_{\vec p''}\oP(l)) \cdot
\sigma = (G_{\hat p'}\oP(k) \circ_i G_{\hat p''}\oP(l)) \cdot
\sigma. 
\end{equation}
At this point, one needs to verify that
\begin{subequations}
\begin{equation}
\label{Budu jeste pristi rok letat?}
\vec p' \circ_i \vec p'' \preceq \vec p \
\hbox { implies } \ \hat p' \circ_i \hat p'' \preceq \hat p,
\end{equation}
thus the term in~(\ref{Opet strasne pocasi na Safari.}) indeed occurs in the
first sum of~(\ref{StdOpFiltbis}), with $\vec q' := \hat p'$ and  
 $\vec q'' := \hat p''$. 
 
 The second sum in~\eqref{StdOpFiltbis} can be handled in a similar way using the property that
\begin{equation}
\label{Budu jeste rok letat?}
[\vec p', \vec p'']_{ij} \preceq \vec p \
\hbox { implies } \ [\hat p' , \hat p'']_{ij} \preceq \hat p
\end{equation}
\end{subequations}
Both~(\ref{Budu jeste pristi rok letat?}) and~(\ref{Budu
  jeste rok letat?}) can be verified directly, using the elementary
inequalities
\begin{align*}
\min(x,k) + \min(y,l)& \leq \min(x+y,k+l), \ \hbox { and}
\\
\min(x-1,k) & \leq  \min(x,k)-1
\end{align*}
that hold for arbitrary $x,y,k,l \in \bbZ$.

 (iii) Let $F\oP=\{F_{\vec p}\oP(n)\}_{\vec p,n}$ be a
{\it D}-multifiltration of $\oP$ such that 
$E(n)\subset F_{(1,\dots,1)}\oP(n)$ for $n\geq 1$. We must
prove that 
\begin{equation}
\label{Porad mam divne napnutou kuzi.}
G_{\vec
  p}\oP(n) \subset F_{\vec p}\oP(n), \
\hbox{ for each $n \geq 1$, $\vec p \in \MZ(n)$}.
\end{equation} 
Since, by definition,  $G_{\vec
  p}\oP(n) = 0$ if $\vec p \not\succeq \jednas$, we may assume that
$\vec p \succeq \jednas$.  

Inclusion~(\ref{Porad mam divne napnutou kuzi.}) is clear for $n = 1$.
For $n=2$ it follows from $G_{\vec
  p}\oP(2) = E(2) \subset   F_{\jednas}\oP(n)\subset  F_{\vec p}\oP(n)$.
For $n > 2$ we proceed by induction. Assuming that
  $G_{\vec{p}}\oP(m)\subset F_{\vec p}\oP(m)$ for all $m < n$,
  $\vec p\in \MZ(m)$, we get
  \begin{align*} 
  G_{\vec p} \oP(n) &=E(n)+ \sum_{i, \vec p', \vec p'',k,l,\sigma} (G_{\vec
                      p'}\oP(k) \circ_i 
G_{\vec p''}\oP(l))\cdot \sigma   +\sum_{i, j, \vec p', \vec
                      p'',k,l,\sigma} [G_{\vec
                      p'}\oP(k), 
G_{\vec p''}\oP(l)]_{ij} \cdot \sigma
\\
&\subset E(n)+ \sum_{i, \vec p', \vec p'',k,l,\sigma}( F_{\vec
                      p'}\oP(k) \circ_i 
F_{\vec p''}\oP(l)) \cdot \sigma 
+\sum_{i, j, \vec p', \vec p'',k,l,\sigma} [F_{\vec
                      p'}\oP(k), 
F_{\vec p''}\oP(l)]_{ij} \cdot \sigma
\\
 &\subset F_{\vec p}\oP(n)
  \end{align*}
  for all $\vec p \in \MZ(n)$, which proves the induction step.
\end{proof}

\begin{remark}
\label{Vyrazku mam na rukou.}
The stabilization bound estimate of Proposition~\ref{StdOpFiltProp} 
can be improved. Namely, if $k \geq 2$ is the smallest
integer such that \hbox{$E(k) \not= 0$}. Then
$\{G_{\vec p}\oP(n)\}_{\vec p,n}$ stabilizes in arity $n \geq
2$ at $N = \left[
  \frac{n-1}{k-1}\right]$, the integral part of the fraction.
\end{remark}
\begin{remark}
 The above construction can be simplified to one resulting in a multifiltration, rather than a $D$-multifiltration,
 by omitting the rightmost sum in \eqref{StdOpFiltProp} and dropping the requirement of $E(1)$ being closed under the commutator $[-,-]_{11}$.
\end{remark}
\begin{definition}
\label{V utery navsteva Strakonic}
Let $\oP$ be a $\bbk$-linear operad generated by a collection $E$. 
The {\em standard {\it D}-multifiltration with respect to $E$\/} is the
smallest saturated {\it D}-multifiltration  $F\oP$ of\/ $\oP$
such that ${E(n)\subset F_{(1,\dots,1)}\oP(n)}$ for each $n\geq 1$.
\end{definition}

More explicitly, we have the following

\begin{lemma}
\label{Zitra budu hovorit s Jenikem Stolbou.}
Let $\oP$ be a $\bbk$-linear operad generated by a collection $E$. 
Then the standard {\it D}-multifiltration with respect to $E$ is the saturation $\oG\oP$
of the prestandard {\it D}-multifiltration $G\oP$.
\end{lemma}
\begin{proof}
 Let $F\oP$ be the standard {\it D}-multifiltration of $\oP$ with respect to $E$.
 By part (iii) of Proposition~\ref{StdOpFiltProp}, we have $G\oP\preceq F\oP$.
 Upon taking the saturation on both sides, by~\eqref{SatMono}, we get ${\oG\oP\preceq \overline{F}\oP=F\oP}$.
 On the other hand, by the minimality of $F\oP$ among saturated {\it D}-multifiltrations, we have 
 ${F\oP\preceq \oG\oP}$.
 \end{proof}
\rev{32d}
{%
 In view of the above lemma, we will routinely use $\oG\oP$ as our notation for the standard {\it D}-multifiltration of an operad $\oP$ with respect to a certain generating collection.
}

\begin{remark}
Since prestandard {\it D}-multifiltrations of simply-connected operads are stabilized, part (ii) of Proposition~\ref{divna vyrazka} along with Corollary~\ref{nejde e-mail} provides a recipe for computing the components of standard {\it D}-multi\-filt\-ra\-tions, which is to be made use of in the next section.
\end{remark}

We conclude this section with a couple of results on the functoriality of (pre)standard {\it D}-multifiltrations that will be used later.
\begin{lemma}
\label{Dnes jsem ploval v Hradistku.}
Let $G_E\oP =\{G_{\vec p}\oP(n)\}_{\vec p,n}$ be the prestandard {\it D}-multifiltration of\/ $\oP$ with respect to a generating collection $E =
\{E(n)\}_{n\geq 1}$,\/  
$\phi:\oP \twoheadrightarrow \oQ$ be a
surjective morphism and $F := \phi(E)$ be the subcollection of\/ $\oQ$ with
components $F(n) := \phi(E(n))$, $n \geq 1$. 
Then the image $\phi(G_E)\oQ$ of the prestandard {\it D}-multifiltration $G_E\oP$ is the prestandard {\it D}-multifiltration $G_F\oQ$ of\/ $\oQ$
with respect to the generators~$F$, that is $G_F\oQ = \phi(G_E)\oQ$.
\end{lemma}

\begin{proof}
\revnm{adjusted}
{%
It is easy to verify that, given a {\it D}-multifiltration $F\oQ$ of $\oQ$, the
collection 
\[
\phi^{-1}(F)\oP  = \{\phi^{-1}(F)_{\vec p}\oP(n)\}_{\vec
  p, n}
\] 
with the components
\[
\phi^{-1}(F)_{\vec p}\oP(n) := \phi^{-1}(F_{\vec p}\oQ(n)),
 \ \vec p \in \MZ(n),\ n
\geq 1,
\]
is a {\it D}-multifiltration of $\oP$. Applying this to the prestandard {\it D}-multifiltration $G_F\oQ$, we get a {\it D}-multifiltration $\phi^{-1}(G_F)\oP$ of $\oP$ with $E(n)\subset \phi^{-1}(G_F)_{\jednas}\oP(n)$ for all $n\geq 1$. 
By the minimality property of $G_E\oP$,  we get
$G_E\oP\preceq \phi^{-1}(G_F)\oP$. Hence, 
\[
 \phi(G_E)\oP \preceq \phi(\phi^{-1}(G_F))\oQ = G_F\oQ.
\]
On the other hand, Lemma~\ref{DMFiltMorph} implies that 
$\phi(G_E)\oQ$ is a {\it D}-multifiltration of $\oQ$ satisfying that
$F(n)\subset \phi(G_E)_{\jednas}\oQ(n)$ for all $n\geq 1$. Thus, by minimality of $G_F\oQ$, 
\[
G_F\oQ \preceq \phi(G_E)\oP 
\]
and the desired equality follows.
}
\end{proof}

\revnm{adjusted}
{%
\begin{proposition}
\label{Zase mne pali hrbety rukou.}
Assume that $\alpha : \oP \to \oS$ is a not necessarily surjective morphism
of operads, $\oG\oP =\{\oG_{\vec p}\oP(n)\}_{\vec p,n}$ is the standard
{\it D}-multifiltration of\/ $\oP$ with respect to the generating collection \hbox{$E =
\{E(n)\}_{n\geq 1}$}, and $F\oS = \{F_{\vec p}\oS(n)\}_{\vec p,n}$ is
a saturated {\it D}-multifiltration of\/ $\oS$ such that 
\[
\alpha(E(n)) \subset
F_{(1,\ldots,1)} \oS(n),
\] 
for each $n \geq 2$. Then $\alpha(\oG)\oS \preceq F\oS$.
\end{proposition}

\begin{proof}
Denote by $\oQ \subset \oS$ the image of $\alpha$
so that $\alpha$ factorizes as
$\oP \stackrel{\phi}{\twoheadrightarrow} \oQ \hookrightarrow \oS$.
Then $\oQ$ is a suboperad
of $\oS$ carrying two {\it D}-multifiltrations. 
The first one is the restriction $F{\oQ}$ of $F\oS$ to $\oQ$, and the second one is the image $\phi(G)\oQ$ of the prestandard {\it D}-multifiltration of $\oP$.

For any $n\geq 2$, we have ${\alpha(E(n))=\phi(E(n))\subseteq F_{(1,\ldots,1)} \oQ(n)}$. 
By Lemma~\ref{Dnes jsem ploval v Hradistku.}, $\phi(G)\oQ$ is equal to the prestandard {\it D}-multifiltration $G \oQ$ of $\oQ$ with respect to the generators $\alpha(E)$. 
Then by the minimality property of prestandard {\it D}-multifiltrations, we have $\phi(G)\oQ\preceq F\oQ$. Passing to the saturations, by Corollary~\ref{SatMorph} and \eqref{SatMono}, we get
\[
 \phi(\oG)\oQ\preceq \overline{\phi(G)}\oQ\preceq \overline{F}\oQ=F\oQ \subseteq F\oS,
\]
where we use the fact that $F\oQ$ is saturated as a restriction of a saturated multifiltration $F\oS$ onto $\oQ$. It remains to observe that $\alpha(\oG)\oS=\phi(\oG)\oQ$.
\end{proof}
}

\section{Standard {\it D}-multifiltrations -- examples and calculations}
\label{4 dny s Jarkou na chalupe.}
\rev{37, 38}
{%
The section presents some explicit examples of standard {\it D}-multifiltrations on operads with a single generator. Much of the work done here amounts to analyzing the basic combinatorics of the standard {\it D}-multifiltration on a free operad and then passing to the saturation of its quotient. 
The corresponding results will be used later in the proof of Theorem \ref{Vymenim to kolo?}.
}

The proposition below addresses the case of the standard {\it D}-multifiltration of the free
operad $\Free$  when its generating collection $E$
is spanned by a single fully symmetric or fully antisymmetric $n$-ary operation, $n \geq 2$, and degree of the same parity as $n$.
\begin{proposition}
\label{Porad se mi chce spat.}
For the standard {\it D}-multifiltration $\oG\Free$   of the free operad\/ 
$\Free$, the following properties hold in arity $2n-1$:
 \begin{enumerate}
\item[(i)] 
$\oG_{\vec p}\Free(2n-1) = \oG_{\vec p\land
    (2,\ldots,2)}\Free(2n-1)$ for each $\vec p \in \MZ(2n-1)$,
  \item [(ii)]
$\oG_{(2,2,\dots, 2)}\Free(2n-1) = \Free(2n-1)$, thus
  $\dim \oG_{(2,2,\dots, 2)}\Free(2n-1) = {{2n-1}\choose {n}}$, and
  \item[(iii)]
\rule{0em}{1.5em}
  $\dim \oG_{(1,2,\dots, 2)}\Free(2n-1) =  {{2n-2}\choose {n}} +
  \frac{1}{2}{{2n-2}\choose{n-1}}$.
  \item[(iv)]
If\/ $E$ is spanned by a single $n$-ary antisymmetric operation
$[-,\ldots,-]$, then $\oG_{(1,1,\dots, 1)}\Free(n)$ 
contains the Jacobiator
\begin{equation}
\label{Za chvili na Brezanku.}
  \oJac_n:=\sum\limits_{\sigma\in
   {\it Sh}_{n,n-1}}\sgn(\sigma)\cdot 
\big[[\sigma(1),\dots, \sigma(n)], \sigma(n+1),\dots, \sigma(2n-1)\big]
\in \Free(2n-1).  
\end{equation}
Here, the summation runs over all $(n,n-1)$-shuffles, i.e.\ permutations
$\sigma \in \Sigma_{2n-1}$ such that
\[
\sigma(1) < \cdots < \sigma(n) \ \hbox { and } \ 
\sigma(n+1) < \cdots < \sigma(2n-1).
\]
\end{enumerate} 
\end{proposition}

In the above formula, $\big[[\sigma(1),\dots, \sigma(n)],
\sigma(n+1),\dots, \sigma(2n-1)\big]$ denotes the operation 
\[
\big[[-,\dots, -]\circ_1 [-,\dots, -] \in \Free(2n-1)
\]
acted on by $\sigma \in \Sigma_{2n-1}$. We will use the same type of
notation for the action of the symmetric group also in the rest of the
paper.

In can be shown by a tedious, but straightforward, argument that $\oG_{(1,1,\dots, 1)}\Free(2n-1)$ in part (iv) is in fact one-dimensional and spanned by the Jacobiator~(\ref{Za chvili na Brezanku.}). 
A particular case corresponding to $n=2$ is addressed in Example~\ref{Jarca je na chalupe.}. 

\begin{proof}[Proof of Proposition~\ref{Porad se mi chce spat.}]
Item (i) says that the the standard {\it D}-multifiltration of\/ $\Free$
stabilizes in arity $2n-1$ at $2$. This follows from 
Proposition~\ref{StdOpFiltProp}, resp.~from its enhancement spelled out in
Remark~\ref{Vyrazku mam na rukou.}. 

Since all the results of the proposition concern pieces of arity $2n-1$, we
will not specify, in the rest of this proof, that arity explicitly where it is
clear from the context.
Let $\omega$ be a fully symmetric or fully antisymmetric operation of
arity $n$ spanning  $E$. Then the expressions
\begin{equation}
\label{Porad mne svedi ruce.}
\omega\big(\omega(\sigma(1),\dots, \sigma(n)), \sigma(n+1),\dots,
\sigma(2n-1)\big), \ \sigma \in \Sh{n,n-1},
\end{equation} 
form a basis of $\Free(2n-1)$. Since all terms above are obtained from
$\omega \circ_1 \omega$ by the action of an element $\sigma$ of
$\Sigma_{2n-1}$  such that $\big((1,1,\ldots,1) \circ_1
(1,1,\ldots,1)\big)\cdot
\sigma \preceq (2,2,\ldots,2)$, 
they all belong to  $G_{(2,2,\dots, 2)}\Free\arity{}$
by~\eqref{StdOpFilt}, and thus
\[
G_{(2,2,\dots, 2)}\Free(2n-1) = \Free(2n-1).
\]
The equality in (ii) then follows from the inclusions
\[
G_{(2,2,\dots, 2)}\Free(2n-1) \subset 
\oG_{(2,2,\dots, 2)}\Free(2n-1) \subset \Free(2n-1).
\]
The second part of  item~(ii) expresses
that there are exactly  ${{2n-1}\choose {n}}$ shuffles in $\Sh{n,n-1}$.

Let us attend to~(iii). It is straightforward to verify, using the
(anti)symmetry of the generating operation, that in arity $2n-1$, 
equation~\eqref{StdOpFilt}  for
$G_{(1,2,\ldots,2)} \Free$ reduces to
\begin{equation}
\label{zablesk nadeje?}
G_{(1,2,\ldots,2)} \Free :=   
\sum_{\sigma} 
(E(n) \circ_2 E(n)) \cdot \sigma
  +\sum_{\sigma} [E(n), E(n)]_{11} \cdot \sigma.
\end{equation}
The first sum in~(\ref{zablesk nadeje?}) generates the vectors
\begin{equation}
\label{Vcera jsem popijel.}
\omega\big(1,\omega(\sigma(2),\dots,\sigma(n+1)),
\sigma(n+2),\dots,\sigma(2n-1)\big),
\end{equation}
where $\sigma$ is a permutation
of the set $\{2,\ldots,2n-1\}$ such
that 
\begin{equation}
\label{Bude foukat vitr.}
\sigma(2) < \cdots < \sigma(n+1) \ \hbox { and } \ 
\sigma(n+2) < \cdots < \sigma(2n-1).
\end{equation}
The second sum in~(\ref{zablesk nadeje?}) generates
the vectors
\begin{align}
\label{eq:1}
\omega\big(\omega(1,\sigma(2),\dots,\sigma(n)),\ &\sigma(n+1),\dots,\sigma(2n-1)\big)
\\
\nonumber 
&- (-1)^n
             \omega\big(\omega(1,\sigma(n+1),\dots,\sigma(2n-1)),\sigma(2),\dots,\sigma(n)\big), 
\end{align}
where $\sigma$ is a permutation 
of the set $\{2,\ldots,2n-1\}$ such
that 
\begin{equation}
\label{Bude foukat vitr?}
\sigma(2) < \cdots < \sigma(n), \ 
\sigma(n+1) < \cdots < \sigma(2n-1) \  \hbox { and }\  \sigma(2) < \sigma(2n-1).
\end{equation}
It is simple to show that the
vectors in~(\ref{Vcera jsem popijel.}) and~(\ref{eq:1}) are linearly
independent, thus they form a~basis of
$G_{(1,2,\ldots,2)}\Free(2n-1)$. 
Moreover, by~\eqref{Paleni na zadech a mezi prsty.} with $N=2$,  
\[
\oG_{(1,2,\ldots,2)}
\Free(2n-1) =   G_{(1,2,\ldots,2)} \Free(2n-1).
\] 
The formula in (iii) then simply expresses the total number of the vectors
in~(\ref{Vcera jsem popijel.}) and~(\ref{eq:1}).

Let us finally turn our attention to (iv). By
formula~\eqref{Paleni na zadech a mezi prsty.} with $N=2$, 
\[
\oG_{(1,1,\dots,1)}\Free(2n-1) = \bigcap\  G_{(2,\ldots,1,\ldots,2)}\Free(2n-1) 
\]
where the intersection runs over all positions of $1$ in the multiindex. We
thus need to show that $\oJac_n \in G_{(2,\ldots,1,\ldots,2)}\Free(2n-1)$ for
each position of $1$.  Since $\oJac_n$ is stable
under cyclic permutations, it suffices to
establish that  $\oJac_n \in G_{(1,2,\ldots,2)}\Free(2n-1)$.
To this end we decompose
$\oJac_n = A_n + B_n$, where
\begin{align*}
A_n:= 
\sum_\sigma \sgn(\sigma)  \Big\{
\big[[1,\sigma(2),\dots,\sigma(n)],\
  &\sigma(n+1),\dots,\sigma(2n-1)\big]
\\
\nonumber 
& 
- (-1)^n
\big[[(1,\sigma(n+1),\dots,\sigma(2n-1)],\sigma(2),\dots,\sigma(n)\big]
\Big\}, 
\end{align*}
where $\sigma$ runs over all permutations as in~(\ref{Bude foukat
  vitr?}), and
\begin{align*}
B_n & :=(-1)^n \sum_\sigma \sgn(\sigma)
\big[[(\sigma(2),\dots,\sigma(n+1)],1,
\sigma(n+2),\dots,\sigma(2n-1)\big]
\\
&=(-1)^{n+1} \sum_\sigma \sgn(\sigma)\big[1,[(\sigma(2),\dots,\sigma(n+1)],
\sigma(n+2),\dots,\sigma(2n-1)\big],
\end{align*}
with
$\sigma$ running over permutations in~(\ref{Bude foukat vitr.}). Now
it suffices to notice that $A_n$ is a linear combination of the 
vectors~(\ref{eq:1}), while $B_n$ is combination of the
vectors~(\ref{Vcera jsem popijel.}), thus both $A_n$ and $B_n$ belong
to $G_{(1,2,\ldots,2)}\Free(2n-1)$. The decomposition $\oJac_n =
A_n + B_n$ is an abstract version of a similar trick used in the proof
of Proposition~\ref{Zase selhavam.}.
\end{proof}

Let $\Free$ be the free operad on a collection of, possibly several, binary operations~$E$.
Then for the standard {\it D}-multifiltration  $\{\oG_{\vec p}\Free(n)\}_{\vec p,n}$, we have 
\[
\oG_{(p_1,p_2)}\Free(2) =
\begin{cases}
 E(2)  & \hbox {if $p_1,p_2 \geq 1$, and}
\\
0 & \hbox{otherwise.}
\end{cases}
\]  
By Proposition~\ref{StdOpFiltProp}, in arity $3$, $\oG$ stabilizes at $N=2$ and its components form the lattice
\begin{equation}
\label{V nedeli ma byt hrozna zima.}
\xymatrix{& \oG_{(2,2,2)}\Free(3) = \Free(3)&
\\
\oG_{(1,2,2)}\Free(3)\ar@{^{(}->}[ru]  &  \oG_{(2,1,2)}\Free(3)\ar@{^{(}->}[u] 
& \oG_{(2,2,1)}\Free(3)\ar@{_{(}->}[lu]
\\
\oG_{(2,1,1)}\Free(3)\ar@{^{(}->}[ru]\ar@{^{(}->}[rru]
& \ar@{_{(}->}[lu]\ar@{^{(}->}[ru] \oG_{(1,2,1)}\Free(3) &
\oG_{(1,1,2)}\Free(3)\ar@{_{(}->}[lu]\ar@{_{(}->}[llu] 
\\
& \oG_{(1,1,1)}\Free(3)\ar@{_{(}->}[lu]\ar@{^{(}->}[u]\ar@{^{(}->}[ru]     &
}
\end{equation} 
\rev{38}
{%
By the saturation property, the above lattice is determined by the components 
$\oG_{(1,2,2)}\Free(3)$, $\oG_{(2,1,2)}\Free(3)$ and $\oG_{(2,2,1)}\Free(3)$ as we illustrate in the following examples.
}
\begin{example}
\label{Jarca je na chalupe.}
Let $E=E(2)$ be spanned by a single antisymmetric operation $[-,-]$; 
its antisymmetry is expressed as $[1,2] = - [2,1]$. 
The pieces in the the top two tiers of~(\ref{V nedeli ma byt hrozna zima.}) are equal to
\begin{align*}
\oG_{(2,2,2)}\Free(3) &= \Span\big(\A,\B,\C\big) = \Free(3),
\\
\oG_{(1,2,2)}\Free(3) &= \Span\big(\A,\B+\C\big),
\\
\oG_{(2,1,2)}\Free(3) &= \Span\big(\B,\C+\A\big), \hbox { and}
\\
\oG_{(2,2,1)}\Free(3) &= \Span\big(\C,\A+\B\big).
\end{align*}
To identify the remaining components, 
suppose that 
\[
\mu\in\oG(1,1,2)\Free(3)=\mu\in\oG(1,2,2)\Free(3)\cap\mu\in\oG(2,1,2)\Free(3).
\]
Then there exist $c_1,c_2,d_1,d_2 \in \bbk$ such that
\[
\mu =
c_1\A + c_2\big(\B+\C\big) = d_1\B + d_2\big(\C+\A\big).
\]
Since $\A$, $\B$ and $\C$ form a basis of $\oG_{(2,2,2)}\Free(3)$, it follows that $c_1=c_2=d_1=d_2$, and setting the common value of these coefficients 
to $1$ produces $\mu = \oJac_3$, where
\begin{equation}
\label{Vcera nove predni gumy.}
\oJac_3 := \A+\B+\C
\end{equation}
is the abstract Jacobiator. Thus
$\oG_{(1,1,2)}\Free(3) =  \Span\big(\oJac_3\big)$.
Since we already know from part (iv) of Proposition~\ref{Porad se mi chce
  spat.} 
that $\oJac_3 \in
\oG_{(1,1,1)}\Free(3)$,  we conclude
that
\begin{equation}
\label{Je to atopicky exem?}
\oG_{(2,1,1)}\Free(3) = \oG_{(1,2,1)}\Free(3)=\oG_{(1,1,2)}\Free(3)=
\oG_{(1,1,1)}\Free(3) = \Span\big(\oJac_3\big).
\end{equation}
\end{example}

\begin{example}
\label{F1sym}
Let the generating collection 
of $\Free$ be spanned by one symmetric binary operation
$(-,-)$. That is, $(1,2)=(2,1)$ in terms of the shorthand notation of 
Example~\ref{Jarca je na chalupe.}. Then $\Free (3)$ has a basis consisting of
$\X,\Y$ and $\Z$. We easily verify that in~\eqref{V nedeli ma byt hrozna zima.}
\begin{align*}
\oG_{(2,2,2)}\Free(3) &= \Span\big(\X,\Y,\Z\big) = \Free(3),
\\
\oG_{(1,2,2)}\Free(3) &= \Span\big(\X,\Y-\Z\big),
\\
\oG_{(2,1,2)}\Free(3) &= \Span\big(\Y,\Z-\X\big), \hbox { and}
\\
\oG_{(2,2,1)}\Free(3) &= \Span\big(\Z,\X-\Y\big).
\end{align*}
In contrast with the situation of Example~\ref{Jarca je
  na chalupe.},  
the remaining pieces of the poset~\eqref{V nedeli ma byt hrozna zima.}
are trivial,
\begin{equation}
\label{Zacinaji strasne zimy.}
\oG_{(2,1,1)}\Free(3) = \oG_{(1,2,1)}\Free(3)=\oG_{(1,1,2)}\Free(3)=
\oG_{(1,1,1)}\Free(3) = 0.
\end{equation}
Indeed, by the saturation property,
\begin{align*}
\oG_{(2,1,1)}\Free(3) = \oG_{(2,1,2)}\Free(3) \cap \oG_{(2,2,1)}\Free(3)\\
\oG_{(1,2,1)}\Free(3) = \oG_{(1,2,2)}\Free(3) \cap \oG_{(2,2,1)}\Free(3)\\
\oG_{(1,1,2)}\Free(3) = \oG_{(1,2,2)}\Free(3) \cap \oG_{(2,1,2)}\Free(3)
\end{align*}
whereas the intersections on the right hand sides are all trivial by a
simple linear algebra.
\end{example}

\begin{example}
\label{Ibisek jeste kvete.}
Consider finally the case when $E$ is spanned by a single binary operation
$(-,-)$ with no symmetry. Then $\Free(3)$ has a basis consisting of
$12$ vectors
\[
\big\{\ \L{\sigma(1)}{\sigma(2)}{\sigma(3)}, \
\R{\sigma(1)}{\sigma(2)}{\sigma(3)}\ \big\}_{\sigma \in \Sigma_3}.
\]
Just as before, we observe that
$\oG_{(1,2,2)}\Free(3)$, $\oG_{(2,1,2)}\Free(3)$ and $\oG_{(2,2,1)}\Free(3)$
have their respective bases
\begin{align*}
&\left\{
\begin{array}{cc}
\L123,\ \R231,\  \L132,\ \R321,
\\ 
\L231-\L321,\ \R123-\R132,\  \L312 -\R312,\
\L213 - \R213
\end{array}
\right\},
\\
&\left\{
\begin{array}{cc}
\L231,\ \R312,\  \L213,\ \R132,
\\ 
\L312-\L132,\ \R231-\R213,\  \L123 - \R123,\  \L321 -\R321
\end{array}
\right\} \ \hbox { and }
\\
&\left\{
\begin{array}{cc}
\L312,\ \R123,\  \L321,\ \R213,
\\ 
\L123-\L213,\ \R321-\R312,\  \L132 -\R132,\
\L231 - \R231
\end{array}
\right\}.
\end{align*}
Then $\oG_{(1,1,2)}\Free(3)=\oG_{(1,2,2)}\Free(3) \cap
\oG_{(2,1,2)}\Free(3)$ is 
$4$-dimensional, spanned~by
\begin{align*}
\L123+\R132-\R123,\ \ \L213 + \R231 - \R213,
\\ 
\L132+\R312-\L312,\ \ \L231+\R321 - \L321.
\end{align*}
Finally, $\oG_{(1,1,1)}\Free(3)=\oG_{(1,1,2)}\Free(3) \cap \oG_{(2,2,1)}\Free(3)$ turns
out to be one-dimensional, with a basis vector
\begin{align*}
\L123 +  \R132 - \R123 - \L213 -\R231 + \R213 \hskip 1cm
\\
\hskip 1cm
 - \L132 - \R312  + \L312  
 + \L231   +  \R321 - \L321.
\end{align*}
The above vector can be conveniently rewritten using the associator
\[
{\rmAss}(1,2,3) := \big((1,2),3 \big) -  \big(1,(2,3) \big) \subset \Free(3)
\]
of the operation $(-,-)$.
Namely, it is
\[
\rmLieAdm(1,2,3) := \sum_{\sigma \in \Sigma_3}
\sgn(\sigma)
\rmAss\big({\sigma(1)},{\sigma(2)},{\sigma(3)}\big)
\]
making up a relation 
characterizing Lie admissible algebras~\cite[Example~6]{markl-remm:JA06}.
\end{example}

\begin{example}
\label{Za chvili mam skypovat s Dotsenkem a nejde mi e-mail.}
Let $\Lie$\/ be the operad governing Lie algebras, presented as the
quotient of the free operad $\Free$ in Example \ref{Jarca je na
  chalupe.} modulo the
Jacobiator~\eqref{Vcera nove predni gumy.}, and $\phi : \Free
\twoheadrightarrow \Lie$ be the natural projection. We are going to
describe the standard {\it D}-multifiltration of $\Lie$ with respect to the
generator $\phi([-,-]) \in \Lie(2)$. Let us single out the following
elements of $\Lie(3)$,
\[
e := \phi\big(\A\big),\ f := \phi\big(\B\big),\ g := \phi\big(\C\big) 
\]
related by the Jacobi identity $e+f+g=0$. We choose $\{e,f\}$ as a
basis of $\Lie(3)$. According to Lemma~\ref{Dnes jsem ploval v
  Hradistku.}, we may describe the relevant pieces of the 
prestandard {\it D}-multifiltration of
$\Lie(3)$ as the image of the same pieces of the
prestandard {\it D}-multifiltration of $\Free$.
The result is
\begin{align*}
G_{(2,2,2)}\Lie(3) &= \Span(e,f) = \Lie(3),
\\
G_{(1,2,2)}\Lie(3) = \Span(e),\
G_{(2,1,2)}\Lie(3) &= \Span(f) \hbox { and }
G_{(2,2,1)}\Lie(3) = \Span(e+f).
\end{align*}
By part (i) of Proposition \ref{divna vyrazka}, the prestandard
{\it D}-multifiltration $G\Lie$ stabilizes, so we may use
formula~\eqref{Paleni na zadech a mezi prsty.} combined with
Corollary~\ref{nejde e-mail} to describe the standard {\it D}-multifiltration
of $\Lie(3)$. The result is
\begin{align*}
\oG_{(2,2,2)}\Lie(3) &= \Span(e,f) = \Lie(3),
\\
\oG_{(1,2,2)}\Lie(3) = \Span(e),\
\oG_{(2,1,2)}\Lie(3) &= \Span(f) ,\
\oG_{(2,2,1)}\Lie(3) = \Span(e\!+\!f),
\\
\oG_{(1,1,2)}\Lie(3) = \oG_{(2,1,1)}\Lie(3)
=& \oG_{(1,2,1)}\Lie(3)
= \oG_{(1,1,1)}\Lie(3) = 0.
\end{align*}
The lattice analogous to~\eqref{V nedeli ma byt hrozna zima.} for the standard
{\it D}-multifiltration of $\Lie(3)$ thus looks as
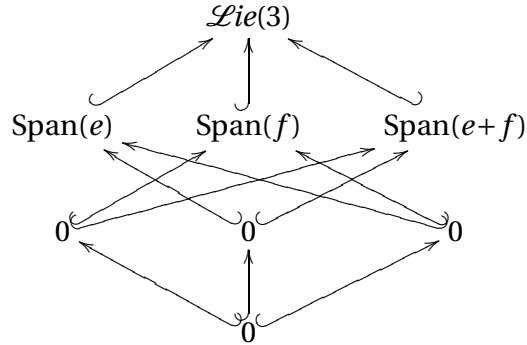
\[
\xymatrix{& \Lie(3)&
\\
\Span(e)\ar@{^{(}->}[ru]  &  \Span(f)\ar@{^{(}->}[u] 
& \Span(e\!+\!f)\ar@{_{(}->}[lu]
\\
0\ar@{^{(}->}[ru]\ar@{^{(}->}[rru]
& \ar@{_{(}->}[lu]\ar@{^{(}->}[ru] 0 &
0\ar@{_{(}->}[lu]\ar@{_{(}->}[llu] 
\\
& \ar@{_{(}->}[lu]\ar@{^{(}->}[u]\ar@{^{(}->}[ru]    0 &
}
\]
The configuration of  $\oG_{(1,2,2)}\Lie(3),
\oG_{(2,1,2)}\Lie(3)$ and $\oG_{(2,2,1)}\Lie(3)$ in $\Lie(3)$ is
portrayed in:
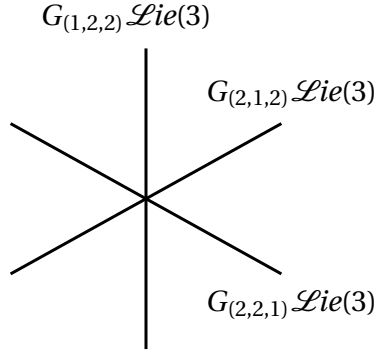
\[
\psscalebox{1.0 1.0} 
{
\begin{pspicture}(0,-2.2367675)(4.329713,2.2367675)
\psline[linecolor=black, linewidth=0.04](1.8097129,1.7632325)(1.8097129,-2.2367675)
\psline[linecolor=black, linewidth=0.04](1.8097129,-0.23676758)(0.009712829,0.7632324)
\psline[linecolor=black, linewidth=0.04](1.8097129,-0.23676758)(3.6097128,-1.2367675)
\psline[linecolor=black, linewidth=0.04](1.8097129,-0.23676758)(0.009712829,-1.2367675)
\psline[linecolor=black, linewidth=0.04](1.8097129,-0.23676758)(3.6097128,0.7632324)
\rput[bl](0.40971282,1.9632324){$G_{(1,2,2)}\Lie(3)$}
\rput[bl](2.6097128,0.9632324){$G_{(2,1,2)}\Lie(3)$}
\rput[bl](2.6097128,-1.8367676){$G_{(2,2,1)}\Lie(3)$}
\end{pspicture}
}
\]
Notice that in this case the image $\phi(\oG_{\vec p}\Free(3))$ equals
$\oG_{\vec p}\Lie(3)$ for each $\vec p \in \MZ(3)$. One may in fact
prove that, more generally, if $\oP$ is a binary quadratic operad
with the quadratic presentation $\Free/(R)$ such that $R \subset
\oG_{(1,1,1)}\Free(3)$, then $\oG_{(1,1,1)}\oP(3)$ coincides with 
the image of $\oG_{(1,1,1)}\Free(3)$ under the canonical projection $\Free
\twoheadrightarrow 
\oP$.
\end{example}

\begin{example}
\label{Bude pocasi na Tereje?}
Let $\Com$ be the operad for commutative associative algebras, presented as the
quotient of the free operad $\Free$ of Example \ref{F1sym} modulo the
associativity $\big(1,(2,3)\big) = \big((1,2),3\big)$. Denoting by $\phi: \Free
\twoheadrightarrow \Com$ the canonical projection and 
\[
a := \phi\big(1,(2,3)\big), \ b := \phi\big(2,(3,1)\big), \ c :=
\phi\big(3,(1,2)\big), 
\] 
the vector space $\Com(3)$ is isomorphic to $\Span(a)$, 
and $a=b=c$ in $\Com(3)$. Using the
pattern of Example~\ref{Za chvili mam skypovat s Dotsenkem a nejde mi
  e-mail.}, we calculate
\begin{align*}
G_{(2,2,2)}\Com(3) = G_{(1,2,2)}\Com(3) =
G_{(2,1,2)}\Com(3) =
G_{(2,2,1)}\Com(3) = \Span(a) = \Com(3).
\end{align*}
From this we conclude that $\oG_{\vec p}\Com(3) = \Com(3)$ for each
$\vec p \in \MZ(3)$. 
All entries of the lattice analogous to~\eqref{V nedeli ma byt hrozna
  zima.} equal $\Com(3)$. 
The standard {\it D}-multifiltration of $\Com$ is strictly bigger than the image of the standard
{\it D}-multifiltration of $\Free$ under the projection $\phi:\Free \twoheadrightarrow \Com$.
\end{example}

\begin{example}
A 3-Lie algebra, aka Filipov algebra~\cite{Fil},  
is a vector space $V$ together with a~trilinear antisymmetric bracket $[-,-,-]$ 
satisfying 
\[
[1, 2, [3, 4, 5]] = [[ 1, 2, 3], 4, 5] + [3, [ 1, 2, 4], 5] + [3, 4, [ 1, 2, 5]].
\]
One can verify that if $\Free$ is the free $\bbk$-linear operad generated by a single antisymmetric ternary operation, then the above identity determines an element in $\oG_{(2,2,1,1,1)}\Free$.
\end{example}

\rev{44}
{%
Given a $\bbk$-linear operad $\oP$ with a generating collection $E$,
its relations $R=\{R(n)\}_{n\geq 1}$ are, in general, expected to be contained 
in the components $\oG_{\vec p}\oP(n)$ of the corresponding standard {\it D}-multifiltration with $\vec p\succ (1,\dots,1)$. 
To single out the special case of $R$ being confined to the lower-level components of $\oG\oP$, we introduce the following
\begin{definition}
\label{Indexovani fotografii}
A $\bbk$-linear operad $\oP$ with a generating collection $E$
  is  {\em tight\/} if it admits a~{\em tight
  presentation\/}, that is, a presentation
$\oP = \Free/(R)$ such that the collection $R = \coll {R(n)}$
generating the operadic ideal $(R)$ satisfies
\begin{equation}
\label{Otepli se na vikend?}
R(n) \subset \oG_{\jednas} \Free(n),\  \hbox { for each $n \geq 1$.}
\end{equation}   
\end{definition}
}
\revnm{relocated\\adjusted}
{%
\begin{lemma}
\label{TightRel}
Let $\oP$ be a tight $\bbk$-linear operad with a tight presentation
$\oP = \Free / (R)$ and $\oS$ be a $\bbk$-linear operad with a saturated 
{\it D}-multifiltration $F\oS$. If $\tilde{\alpha} : \Free \to \oS$ is an operad
morphism such that
\[
\tilde{\alpha}(E(n)) \subset F_{(\rada{1}{1})}\oS(n)\
\hbox { for all $n\geq 1$},
\] 
then
\[
\tilde{\alpha}(R(n)) \subset F_{(\rada{1}{1})}\oS(n)\
\hbox { for all $n\geq 1$}. 
\]
\end{lemma}

\begin{proof}
We have
$
\tilde{\alpha}(R(n)) \subseteq \tilde{\alpha}(\oG_{(\rada{1}{1})}\Free(n)) \subseteq F_{(\rada{1}{1})}\oS,
$
where the second inclusion follows from  
Proposition~\ref{Zase mne pali hrbety rukou.}.
\end{proof}
}

\begin{theorem}
\label{Vymenim to kolo?}
The only tight quadratic operads generated by a single binary operation
are the free operad $\Free$, the operad $\Lie$ for Lie algebras, and the operad
$\LieAdm$ for Lie admissible algebras.
\end{theorem}

\begin{proof}
The result follows from the analysis carried out in Examples 
\ref{Jarca je na chalupe.}-\ref{Ibisek jeste kvete.}. 
If the generating collection of the quadratic operad $\oP = \Free/(R)$ 
is spanned by one antisymmetric operation, then  $\oG_{(1,1,1)} \Free(3)$ is the
one-dimensional span of the Jacobiator by~\eqref{Je to atopicky exem?}. 
Thus either $R = 0$, in which case $\oP$ is free, or $R = \Span(\oJac_3)$, 
in which case $\oP$ is the operad for Lie algebras.

If the generating operation is commutative, then $ \oG_{(1,1,1)}
\Free(3) = 0$ by~\eqref{Zacinaji strasne zimy.}, thus $\oP$ must be free.
If the generating operation has no symmetry, then either $\oP$ is free
 or $\oP = \LieAdm$ by the result of Example~\ref{Ibisek jeste kvete.}.
\end{proof}

\begin{remark}
It is evident that the coproduct $\oP' \sqcup \oP''$ of tight
operads is tight again. As argued in
\cite[Example~6]{markl-remm:JA06}, the operad $\LieAdm$ is the
coproduct
\[
\LieAdm \cong \Lie \sqcup {\mathbb F}(\varpi)
\]
of the operad for Lie algebras with the free operad on one commutative binary
operation $\varpi$. Thus the tightness of $\LieAdm$ is corroborated by the
tightness of the operads at the right hand side of the above display.
\end{remark}

Theorem \ref{Vymenim to kolo?} indicates that in case of operads generated by a single binary operation tightness  is a fairly restrictive property. 
Yet some meaningful examples of tight operads generated by multiple operations or operations of arities other than two are available. A rather simple example is the operad $\DG$ (cf. Example~\ref{Uz aby bylo jaro.}) whose algebras are differential graded  vector spaces. 
A~less trivial one is the operad $\oLLL$ governing $\LLL$ (strongly homotopy Lie) algebras in the category of graded vector spaces.
Recall the following
\begin{definition}[{\cite[Definition~2.1]{LM}}]
\label{shlie}
An {\em $\LLL$-algebra\/} \ consists of a $\bbk$-linear graded vector space $L$ equipped with 
$\bbk$-linear maps
$l_k:\bigotimes^k L \to L$, $k \geq 1$, of degree $k\!-\!2$ which are
antisymmetric, i.e.\ 
\begin{equation}
\label{antisymmetry}
l_k(\Rada \lambda{\sigma(1)}{\sigma(k)})=
\chi(\sigma)l_k(\Rada \lambda1k)
\end{equation}
for all permutations $\sigma \in \Sigma_k$ and homogeneous $\Rada \lambda1k\in
L$. Moreover, the
following generalized form of the Jacobi identity is required to hold 
for any $k \geq 1$:
\begin{equation}
\label{Zacinaji vedra}
\Jac_k(\Rada \lambda1k):= \sum
\chi(\sigma)(-1)^{i(j-1)}l_j(l_i(\Rada 
\lambda{\sigma(1)}{\sigma(i)}),\Rada
\lambda{\sigma(i+1)}{\sigma (k)})=0,
\end{equation}
with the summation running over all $i,j \geq 1$ with $i+j =
k+1$, and all $(i,k\!-\!i)$-shuffles $\sigma \in \Sigma_k$.
\end{definition}

\rev{43}{%
$L_\infty$-algebras 
are governed by the operad $\oLLL$ with a quadratic presentation 
$\Free/(R)$, where the generating collection $E=\{E(k)\}_{k\geq 1}$ is
such that for each $k\geq 1$, $E(k)$ is the one-dimensional sign
representation of the symmetric group $\Sigma_k$ spanned by $\ol_k$,
and $R=\{R(k)\}_{k\geq 1}$ has its $k$th component~$R(k)$ spanned by the
following abstraction of~(\ref{Zacinaji vedra}):
\begin{equation}
\label{Varim si houbovou omacku.}
\oJac_k(\rada 1k) := \sum_{i+j=k+1} \sum_\sigma 
\sgn(\sigma)(-1)^{i(j-1)}\ol_j(\ol_i(\rada{\sigma(1)}{\sigma(i)}),\rada
{\sigma(i\!+\!1)}{\sigma (k)}),
\end{equation}
where the second summation runs over all 
$(i,k\!-\!i)$-shuffles $\sigma \in \Sigma_k$. 

\begin{proposition}
\label{Pojedu za Taxem.}
The operad $\oLLL$ is tight quadratic.
\end{proposition}

\begin{proof}
Consider the above presentation of $\oLLL$ and let $\oG\Free$ be the standard {\it D}-multifiltration of $\Free$. Since $E(k) \subset \oG_{(\rada 11)} \Free(k)$, then
$\oJac_k \in  \oG_{(\rada 22)} \Free(k)$.
Since $\oJac_k$ is cyclically
symmetric, it suffices to establish that
\begin{equation}
\label{Bude to v nedeli foukat na Podhorany?}
\oJac_k\in \oG_{(1,\rada 22)} \Free(k).
\end{equation}
Then, indeed, by the cyclic
symmetry and equivariance
\[
\oJac_k\in \oG_{(2,\ldots,2,1,2,\ldots,2)}\Free(k)
\]
for any position of $1$ thus, 
by the saturation property,\  $\oJac_k \in \oG_{(\rada 11)}
\Free(k)$ as required.
To prove~(\ref{Bude to v nedeli foukat na Podhorany?}), we start by decomposing
\[
\oJac_{k}(\rada 1k) = \oA_{k}(\rada 1k) + \oB_{k}(\rada 1k).
\]
where $\oA_{k}(\rada 1k)$
is the restriction of the second sum in the right hand side 
of~\eqref{Varim si houbovou omacku.} to  
$(i, k\! -\! i)$-shuffles with $\sigma(1) = 1$, and $\oB_k(\rada 1k)$ 
accounts for the remaining summands, i.e.\ those with $\sigma(i\!+\!1) = 1$. 

Notice first that all the summands of
$\oB_{k}$ belong to  $\oG_{(1,\rada 22)} \Free(k)$. Indeed, each summand
\[
\sgn(\sigma)(-1)^{i(j-1)}\ol_j(\ol_i(\rada{\sigma(1)}{\sigma(i)}),\rada
{\sigma(i\!+\!1)}{\sigma (k)})
\]
of $\oB_k(\rada 1k)$ can be written as 
\[
\sgn(\sigma)(-1)^{i(j-1)}(\ol_j \circ_1
\ol_i)\sigma,
\]
using the operadic $\circ_1$-composition and the right action of the
symmetric group $\Sigma_k$. By definition of the 
standard {\it D}-multifiltration, $\ol_j \in \oG_{(\rada 11)}\Free(j)$,
$\ol_i \in \oG_{(\rada 11)}\Free(i)$ thus, by the compositional compatibility,  
\[
(\ol_j \circ_1 \ol_i) \in \oG_{(\rada 22,\rada 11)}\Free(k)
\]
where $2$ occupies the first $i$ positions of the multiindex and $1$
the remaining ones, starting with position $i\!+\!1$. Since $\sigma(1) =
i\!+\!1$, the first input of   $(\ol_j \circ_1 \ol_i)\sigma$ is mapped
to the $(i\!+\!1)$th input of  $(\ol_j \circ_1 \ol_i)$, so
\[
\sgn(\sigma)(-1)^{i(j-1)}(\ol_j \circ_1 \ol_i)\sigma \in  \oG_{(1,\rada
  22)} \Free(k)
\] 
by the equivariance and monotonicity. Thus all the
summands of $\oB_k$ and thus also $\oB_k$ itself belong to  
$\oG_{(1,\rada 22)} \Free(k)$.

To prove that $\oA_k\in \oG_{(1,\rada 22)} \Free(k)$, 
we decompose it further as
\[
\oA_{k}(\rada 1k) = \oA'_{k}(\rada 1k) +  \oA''_{k}(\rada 1k)
+  \oA'''_{k}(\rada 1k),
\]
where  $\oA'_{k}$ is the sum of terms of $\oA_{k}$ 
with $i = j = 1$, and $\oA''_{k}$ is the sum of terms
where either $i$ or $j$, but not both, 
equals $1$, and $\oA'''_{k}$ is the
sum of the remaining terms. Clearly, $\oA'_{k}$ is
nontrivial only for $k=1$ in which case
\[
\oA'_{1} = \ol_1 \circ \ol_1 =\frac 12 [\ol_1,\ol_1]_{11},
\]
so $\oA'_{1} \in \oG_{(1)}\Free(1)$ because $\oG\Free$ is a
{\it D}-multifiltration and $\ol_1 \in  \oG_{(1)}\Free(1)$. 
Furthermore, 
\begin{align*}
\oA''_{k}   = \ol_1 \circ_1 \ol_k - (-1)^{k} \ol_k \circ_1 \ol_1 = 
[\ol_1,\ol_k]_{11}, 
\end{align*}
thus $\oA''_{k} \in \oG_{(1,2,\dots,2)}\Free(k)$
by the same argument. 
It remains to analyze $\oA'''_{k}$. Notice that it decomposes as
\[
 \oA'''_{k}(\rada 1k) =  \oA^<_{k}(\rada 1k) +  \oA^>_{k}(\rada 1k)
\] 
where
\[
\oA^<_{k}(\rada 1k) :=
\sum_{i+j=k+1} \sum_\sigma
\sgn(\sigma)(-1)^{i(j-1)}\ol_{j}(\ol_{i}(1,\rada {\sigma(2)}{\sigma(i)}),\rada
{\sigma(i\!+\!1)}{\sigma (k)})
\] 
with the second sum is restricted to the shuffles $\sigma$ 
with $\sigma(2) < \sigma(i\!+\!1)$ and thus
also $\sigma(1) = 1$, while
\[
\oA^>_{k}(\rada 1k) :=
\sum_{i+j=k+1} \sum_\sigma
\sgn(\sigma)\sgn(\tau)(-1)^{j(i-1)}\ol_{i}(\ol_{j}(1,\rada {\sigma(i\!+\!1)}{\sigma(k)}),\rada
{\sigma(2)}{\sigma (i)})
\] 
with the same range for $\sigma$ and $\tau$ being the permutation
\[
 \rada {\sigma(2)}{\sigma (i)},\rada {\sigma(i\!+\!1)}{\sigma(k)}
\longmapsto
\rada {\sigma(i\!+\!1)}{\sigma(k)}, \rada {\sigma(2)}{\sigma (i)}.
\]
Now, to any $(i,j\!-\!1)$-shuffle such that $\sigma(2) < \sigma(i\!+\! 1)$ we
associate a $(j,i\!-\!1)$-shuffle $\sigma^\dagger$ by
\[
\sigma^\dagger :=
(\sigma(1),\sigma(i+1),\ldots,\sigma(k),\sigma(2),\ldots,\sigma(i)).
\] 
Then, with the same range for $\sigma$ as before, we have
\[
\oA'''_{k} =\oA^<_{k}  + \oA^>_{k}=
\sum_{i+j=k+1} \sum_\sigma
\sgn(\sigma)(-1)^{i(j-1)}(\ol_j \circ_1 \ol_i)\sigma +
\sgn(\sigma)\sgn(\tau)(-1)^{j(i-1)}(\ol_i \circ_1 \ol_j)\sigma^\dagger
\]
from which we conclude that
\[
\oA'''_{k} = \sum_{i+j=k+1} \sum_\sigma \sgn(\sigma)(-1)^{i(j-1)}
[\ol_{j},\ol_i]_{11} \sigma.
\]
Since
$\ol_i \in \oG_{(\rada 11)}\Free(i)$ and  $\ol_j \in \oG_{(\rada
  11)}\Free(j)$, and since $\oG \Free$ is a {\it D}-multifiltration, each
$[\ol_{i},\ol_j]_{11} $ belongs to 
 $\oG_{(1,2,\dots,2)}\Free(k)$, and therefore $\oA'''_{k} \in \oG_{(1,2,\dots,2)}\Free(k)$.
\end{proof}
}

\section{Operator algebras}
\label{OpAlgsSec}

Throughout this section, $A$ will be a graded commutative associative
$\bbk$-algebra.  Let $\oP$ be a $\bbk$-linear operad and
$F\End_A = \{F_{\vec p}\End_A(n)\}_{\vec p, n}$  the multifiltration of
Example \ref{Kdy uz bude ten vozik?}. Then the collection
$\Diff:=\{\Diff(n)\}_{n\geq 1}$, where
\[
\Diff(n):=\bigcup\limits_{\vec p,n}F_{\vec p}\End_A(n),\ n
\geq 1,
\] 
is, by Proposition~\ref{Zitra to snad alespon prestehujeme.}, a
multifiltered suboperad of the endomorphism operad $\End_A$. It has a
multifiltered suboperad $\oDer = \{\oDer(n)\}_n$ such that $F_{(p_1,\ldots,p_n)}\oDer(n)$
consists  of $\bbk$-linear maps
\hbox{$O : A^{\otimes n} \to A$} that are derivations of 
  order $p_i$ in the $i$-th variable, for each $1 \leq i \leq n$.

\begin{definition}
\label{OpAlgDef}
 An \emph{operator $\oP$-algebra} $A$ is an operad morphism $\alpha:\oP\to
 \Diff$. We say that
it is {\em of order $k \geq 0$},
if $\oP$ is generated by a  collection $E=\coll{E(n)}$
such that  
\[
 \alpha(E(n)) \subset F_{(\rada{k}{k})}\Diff(n), \
\hbox {for all $n\geq 1$.}
\]
\end{definition}

\begin{remark}
We ought to warn the reader that the term `operator algebra' is an 
\rev{11}{abbreviation} for a `multilinear differential
operator algebra' and is not directly related to the homonymous, but
more elaborate, functional analytic concept.
\end{remark}

An operator $\oP$-algebra
on a graded commutative \hbox{$\bbk$-algebra $A$} is a $\oP$-algebra in the
ordinary operadic sense, via the composite $\oP\to \Diff
\hookrightarrow \End_A$. \rev{39}{A~{\em morphism\/} of operator $\oP$-algebras
with underlying graded commutative associative algebras $A'$
resp.~$A''$ is an algebra morphism $f : A' \to A''$ which is
simultaneously a morphism of $\oP$-algebras in the usual 
sense~\cite[Definition~II.1.21]{markl-shnider-stasheff:book}.}

\rev{40}{Operator $\oP$-algebras  can be regarded as algebras over the operad defined as the quotient of the coproduct $\oP \sqcup \Com$, where $\Com$ is the operad of graded commutative associative $\bbk$-algebras, by the ideal ${\mathcal I}$ expressing that $\oP$ acts via (higher order) differential operators with respect to the multiplicative structure encoded by $\Com$. 
Aside from some exceptional cases, such as Poisson algebras recalled in
Example~\ref{Pujdu si pro listky}, the ideal  ${\mathcal I}$ is not
generated by a distributive law in the sense
of~\cite{markl:dl}. Indeed, as proven in~\cite{bm}, the only operad
tied to $\Com$ via a nontrivial distributive law is the operad $\Lie$.
The relations expressing the higher derivation property do not even have the form resembling a distributive law. 
} 

\begin{example}
Commutative associative
algebras
are operator algebras of order $0$.
\end{example}

\begin{example}
  A \textit{Jacobi structure}
  on a manifold $M$ is an operator Lie algebra on
  $C^{\infty}(M)$.
A well-known \cite{Kir, Grab} structure
  theorem gives a precise characterization of the only non-trivial
  bracket $[-,-]_1$. Namely,
 \[
  [f,g]_1=\Pi(df, dg)+\xi\lrcorner (f\,dg-g\,df)
 \]
 for a $2$-vector field $\Pi$ and a $1$-vector field $\xi$ on $M$ subject to the compatibility conditions
\[
 [\xi,\Pi]=0,\quad [\Pi,\Pi]=2\xi\wedge \Pi.
\]
\end{example}

\begin{example}
\label{Pujdu si pro listky}
A Poisson algebra $A$ is, by definition, a Lie algebra whose
underlying space is a commutative associative algebra such that the Lie bracket is a derivation of order 1 in each variable. It is thus 
the same as an operad map $\Lie \to \oDer$ that sends the
generator of $\Lie(2)$ into an antisymmetric operation
$A^{\ot 2} \to A$ which is an order $1$ derivation in each
variable. That is, $A$ is operator $\Lie$-algebra of order $1$. The situation is summarized by the commutative diagram 
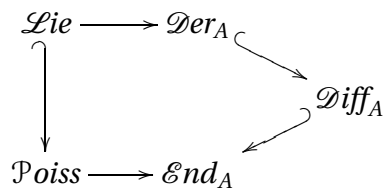
\[
\xymatrix@R=1em{\raisebox{-.2em}{}\Lie\ar@{_(->}[dd]\ar[r] &\ar@{^(->}[dr] \oDer
\\
& &\Diff \ar@{_(->}[dl]
\\
\Poiss \ar[r] & \End_A
}
\]
of operad morphisms in which $\Poiss$ is the operad governing Poisson algebras.
\end{example}

\begin{example}
Let $\DG$ be the operad treated in Examples~\ref{Uz aby bylo jaro.}.
A Batalin-Vilkovisky algebra with underlying commutative
associative algebra $A$ is given by an operad morphism  \hbox{$\DG \to \oDer$} 
that sends the generator of $\DG(1)$ to a derivation of order $2$ and
degree  $-1$ with respect to the
grading of $A$. Batalin-Vilkovisky algebras are therefore operator
algebras of order $2$. The situation is expressed by the diagram
\[
\xymatrix@R=1em{\DG\ar@{_(->}[dd]\ar[r] &\ar@{^(->}[dr] \oDer
\\
& &\Diff \ar@{_(->}[dl]\ ,
\\
\BaVi \ar[r] & \End_A
}
\]
where $\BaVi$ is the operad governing Batalin-Vilkovisky algebras. 
\end{example}
\rev{44}
{%
On the practical side, given a $\bbk$-linear operad $\oP=\Free/(R)$,
defining an operator $\oP$-algebra on a commutative associative $\bbk$-algebra $A$
amounts to assigning a multilinear differential operator to each
generator of $\oP$ and verifying the defining relations within
$\Diff$. The latter task simplifies if each of the relations gets
mapped to a multilinear differential operator of order $1$ in each
variable, which is, in particular, the case for tight operads. More
specifically, we have the following

\begin{corollary}[to Lemma~\ref{TightRel}]
 \label{CheckRel}
 Let $\oP$ be a $\bbk$-linear operad with a tight presentation $\oP=\Free/(R)$,
 $A$ be a commutative associative unital $\bbk$-algebra with a set of generators $S$
 and $\tilde\alpha:\Free\to \Diff$ be an operad morphism such that $\tilde{\alpha}(E(n))\subseteq F_{\jednas}\Diff$ for all $n\geq 1$. 
 Then $\tilde{\alpha}$ factorizes through $\alpha: \Free/(R) \to \Diff$ if and only if \/
 $O(x_1,\dots,x_n)=0$ for all $O\in \tilde{\alpha}(R(n))$, $x_i\in S\sqcup \{1\}$, $i=1,2,\dots, n$, and $n\geq 1$.
\end{corollary}
\begin{proof}
 By Lemma~\ref{TightRel}, $\tilde{\alpha}(R(n))\subseteq F_{\jednas}\Diff(n)$ for all $n\geq 1$.
 The claim follows now upon noting that by \eqref{Kolik najedu?} and by setting $x=1$ in \eqref{Psi2}, a differential operator $\nabla:A\to A$ of order 
 $\leq 1$ is identically zero if and only if $\nabla(x)=0$ for each $x\in S\sqcup \{1\}$.
\end{proof}
If $A$ is not necessarily unital, an analogous result holds upon replacing $\Diff$ in the statement of the above corollary by $\oDer$ and restricting $x_i$'s to the generators $S$.
}

We are going to introduce a deformed version of the multifiltered operads
$\Diff$ and $\oDer$
\rev{41}
{%
to account for the case of algebraic structures with operations representable as formal series of differential operators.
}
To this end, let $h$ be a formal parameter of an even degree and  
$\End_{A\lev h\prav }$ be the endomorphism operad of the $\bbk\lev h\prav $-module
$A\lev h\prav $ with the multifiltration 
$F\End_{A\lev h\prav }= \{F_{\vec p}\End_{A\lev h\prav }(n)\}_{\vec p,n}$ of Example \ref{DiffDef}. The collection
$\Diff{\lev h\prav }:=\{\Diff{\lev h\prav }(n)\}_{n\geq 1}$, 
where
\[
 \Diff{\lev h\prav }(n):=\bigcup\limits_{\vec p \in \MZ(n)}F_{\vec
   p}\End_{A\lev h\prav }(n)\ , n \geq 1,
\]
form a multifiltered suboperad of the endomorphism operad
$\End_{A\lev h\prav }$.  Analogously we define a multifiltered
suboperad $\oDer\lev h \prav$ of $\Diff\lev h \prav$, requiring that
the multilinear maps $O_0,O_1,O_2,\ldots$ in~(\ref{Teprve tri platne
  discipliny.}) are derivations in each of its variables of the
indicated degrees.

\begin{definition}
\label{FormOpAlgDef}
A {\em formal operator $\oP$-algebra} is an operad
morphism $\alpha : \oP \to \Diff\lev h \prav$.
We say that such an algebra is {\em of order $k \geq 0$},
if $\oP$ admits a generating collection $E=\coll{E(n)}$
such that
\[
 \alpha(E(n)) \subset F_{(\rada{k}{k})}\Diff\lev h \prav(n),\
\hbox {for all $n\geq 1$.}
\]
\end{definition}

Expanding this definition, we note that a formal operator
$\oP$-algebra $\alpha: \oP \to \Diff\lev h \prav$ of order $k$ is a
$\oP$-algebra supported on $A\lev h \prav$ and such that its $n$-ary
generating operations are of the~form
\[
 O(\Rada a1n) = O_0(\Rada a1n) + O_{1}(\Rada a1n)\cdot h + O_{2}(\Rada a1n)\cdot h^2+\dots 
\]
where each $O_{s}$ is a multilinear differential operator of order
$\leq k+s$ with respect to each of its arguments. In particular, any
formal operator algebra of order $k$ is automatically a formal
operator algebra of order $l$ for any $l>k$.
The canonical inclusion $A\hookrightarrow A\lev h\prav$ and the
reduction $A\lev h\prav \to A$ mod $h$ determines an operad
morphism $\End_{A\lev h\prav }\to \End_A$, which restricts to 
a~morphism of 
\rev{42}
{multifiltered suboperads $\pi: \Diff\lev h \prav \to \Diff$ as per
  Definition~\ref{MFiltMorph}}.

\rev{39}{A~{\em morphism\/} of formal operator $\oP$-algebras
with underlying algebras $A'$
resp.~$A''$ is a $\bfk\lev h \prav$-linear morphism $f : A'\lev h \prav \to
A''\lev h \prav$  of graded commutative associative  algebras which is
also a standard morphism of \/
$\oP$-algebras~\cite[Definition~II.1.21]{markl-shnider-stasheff:book}
with the underlying graded vector spaces  $A'\lev h \prav$ resp.\
$A''\lev h \prav$.}

\begin{definition}
\label{V podvecer musim na chalupu.}
 The {\em semiclassical limit} of a formal operator algebra $\oP \to \Diff\lev h \prav$ is
 an operator $\oP$-algebra whose structure map is the composite
 \[
  \oP \to \Diff\lev h \prav \overset{\pi}{\longrightarrow} \Diff.
 \]
\end{definition}

The semiclassical limit of a formal operator algebra of order $k$ is
an operator algebra of the same order.

\begin{example}
An example  of a formal operator $\oP$-algebra of order $0$ for
$\oP$ the associative operad $\Ass$ is provided by Terilla's
quantization~\cite{Terilla} as recalled in Subsection~\ref{Kam jsem dal
  tu masticku?}.
Its semiclassical limit is 
the commutative associative algebra $\tS(V \oplus V^*)$.
Another example of the same type is the celebrated Kontsevich deformation
quantization of a Poisson manifold $M$~\cite{kontsevich:LMPh03}. Its semiclassical
limit is the algebra $C^\infty(M)$ of smooth functions on $M$.
\end{example}

\part{Applications}
\label{Exam}

\section{Formal operator $\LLL$-algebras}
\label{Po sesti dnech.}
This and the following section are devoted to a discussion of operator $\LLL$-algebras and and some of their 
instances -- $\IBL_\infty$-alge\-bras, commutative $\BV_\infty$-algebras,  operator Lie algebras
and Poisson algebras. 

\subsection{Formal operator $L_\infty$-algebras} 
Let $A$ be a graded commutative associative algebra.
\begin{definition}
\label{Opet jsem selhal.}
A \hbox{\em formal operator $\LLL$-algebra\/} is an $\LLL$-algebra
whose underlying vector space $L$ is
$A\lev h\prav  := A \ot \bbk\lev h\prav $, 
the structure operations are $\bbk\lev h\prav $-linear and such that, for each 
$\Rada a1k \in A \subset A\lev h\prav $,
\begin{equation}
\label{Ztratil jsem vlecne lano.}
l_k(\Rada a1k) = \sum_{n \geq 1} l_{k,n}(\Rada a1k) \cdot h^{n-1},
\end{equation}
where $l_{k,n} : A^{\bigotimes k} \to A$ is a differential operator of
order $n$ in each variable.
\rev{49}
{In terms of Definition ~\ref{FormOpAlgDef}, such an algebra is a formal operator $\oLLL$-algebra of order one.
}
\end{definition}

\begin{example}
If the underlying algebra $A$ is a graded vector space with the
trivial multiplication, then the differential operators of order one on $A$ are all $\bbk$-linear
maps. Thus the usual $\LLL$-algebras can be thought of as a particular case of the formal
operator $\LLL$-algebras with the structure operations $l_{k,n}$ vanishing
for $n > 1$. In Example~\ref{zatraceny prepinac} we describe a less trivial embedding of
the category of $\LLL$-algebras with weak morphisms into the category of formal operator
$\LLL$-algebras. 
\end{example}

\begin{example}
\label{Vit}
In a similar vein, L.~Vitagliano's \emph{multiderivation $\LLL$-algebras} \cite{Vit1} 
are a particular case of formal operator $\LLL$-algebras, where $l_{k,n}=0$ for $n>1$ and each $l_{k,1}$ is a derivation of order one in each of its variables for all $k\geq 1$.
\end{example}

\begin{example}
\label{Podari se mi spravit Tereje?}
Assume that $A$ is unital. A formal operator $\LLL$-algebra all of
whose structure operations except $l_1
: A\lev h\prav  \to A\lev h\prav $ vanish and $l_1(1)=0$, 
is the same as a {\em commutative ${\rm
    BV}_\infty$-algebra\/}~\cite[Definition~7]{kravchenko1}.
In particular, if $A = \S(V)$ for a graded vector space $V$, we recover the definition of
an $\IBL_\infty$-algebra~\cite[Subsection~4.2]{munster-sachs},
cf.~also~\cite[Example~9]{MV}. 
\end{example}

Due to the $\bbk\lev h\prav $-linearity of the structure operations $l_k$ and the
decomposition~\eqref{Ztratil jsem vlecne lano.}, the
condition $\Jac_k(\Rada \lambda1k)=0$ for $\Rada \lambda1k \in A\lev h\prav $
is equivalent to the system of equalities
\begin{equation}
\label{Ve stredu ma byt pekne.}
\Jac_{k,n}(\Rada a1k) :=  \sum
\chi(\sigma)(-1)^{i(j-1)}l_{j,s}(l_{i,t}(\Rada a{\sigma(1)}{\sigma(i)}),\Rada
a{\sigma(i+1)}{\sigma (k)})=0,
\end{equation}
where $i,j$ and $\sigma$ run over the same ranges as in~(\ref{Zacinaji
  vedra}) and $s,t \geq 1$ run over all values such that $s+t = n+1$,
  for each $n \geq 1$ and $\Rada a1k \in A$.

\rev{43}{%
\begin{proposition}
\label{Zase selhavam.}
The multilinear map $\Jac_{k,n}: A^{\ot k} \to A$ 
defined in~(\ref{Ve stredu ma byt pekne.}) is a
differential operator of order $n$ in each of its variables. 
\end{proposition}

\begin{proof}
We will refer to the notation introduced in the paragraph before
Proposition~\ref{Pojedu za Taxem.}. Let~$l_k$, $k \geq 1$, be as
in~\eqref{Ztratil jsem vlecne lano.} and notice that $l_k \in \oG_{(\rada 11)}
\Diff\lev h \prav(k)$. 
Let us define an operad map \hbox{$\tilde\alpha
: \Free \to \Diff\lev h \prav$} by $\tilde\alpha(\ol_k) : = l_k$, $k
\geq 1$. Then clearly 
\[
\tilde \alpha (\oJac_k) = \sum_{n \geq 1} \Jac_{k,n} \cdot h^{n-1}.
\]
Since $\oLLL$ is tight by Proposition~\ref{Pojedu za
  Taxem.}, then Lemma~\ref{TightRel} implies that $\tilde \alpha (\oJac_k)$ and thus also
the right-hand side of the above display belongs to $\oG_{(\rada 11)}
\Diff\lev h \prav(k)$, which is equivalent to the statement of the proposition.
\end{proof}
}

\begin{remark}
\label{Zas pojedu vlakem.}
It is known that the individual structure 
operations $l_k:\bigotimes^k L \to L$, $k \geq 1$, of an
$\LLL$-algebra can be assembled into a degree $-1$ coderivation 
$\delta$ on the cofree conilpotent cocommutative coassociative coalgebra
\hbox{$\S^c(\uparrow\! L)$} cogenerated by the suspension of
the underlying vector space $L$. Then the infinite system of
relations~\eqref{Zacinaji vedra} is equivalent to a single
equation~$\delta^2=0$, cf.~\cite[Theorem~2.3]{LM}.

Likewise, the structure operations of a formal operator $\LLL$-algebra in
Definition~\ref{Opet jsem selhal.} assemble into a~co\-derivation $\delta_h$ of \hbox{$\S^c(\uparrow\! L)\lev h\prav $} that
squares to $0$. 
We however do not know how to express conveniently the required
decomposition~\eqref{Ztratil jsem vlecne lano.} in terms of
$\delta_h$.
\end{remark}

\begin{example}
\label{zatraceny prepinac}
Consider an $\LLL$-algebra whose structure operations
$l_k:\bigotimes^k L \to L$, $k \geq 1$, are assembled into a
coderivation $\delta$ of the coalgebra $\S^c(\uparrow\! L)$ as in
Remark~\ref{Zas pojedu vlakem.}. Clearly 
\[
\delta = \delta_1 + \delta_2 + \delta_3 + \cdots,
\] 
where $\delta_k$, corresponding to $l_k$, takes \rev{54}{$\S^k(\uparrow\! L)$ to $\uparrow\! L$.
By inspecting explicit formulas~\cite[Eq.~(3)]{lada-stasheff} for
the components $\delta_k$'s of the coderivation $\delta$ 
(denoted $\hat l_n$'s in loc.~cit.), it 
is simple to check that, under the canonical isomorphism of graded
vector spaces
$\S^c(\uparrow\! L) \cong \S(\uparrow\! L)$, each $\delta_k$ is a
differential operator of order~$\leq k$ on the free graded commutative
associative algebra $\S(\uparrow\! L)$.} Taking, in
Definition~\ref{Opet jsem selhal.}, $A := \S(\uparrow\! L)$, $l_k :=
0$ for $k \geq 2$, and 
\[
l_1 : = \delta_1 + \delta_2 h + \delta_3 h^2 + \cdots,
\] 
one represents the initial classical
$\LLL$-algebra as a formal operator $\LLL$-algebra
with underlying commutative associative algebra $\S(\uparrow\! L)$.
\end{example}

\rev{53}{%
A (weak) {\em morphism\/} of $\LLL$-algebras that are 
represented in the language of Remark~\ref{Zas pojedu vlakem.}  by the~dg
coalgebras 
by \hbox{$\big(\S^c(\uparrow\! L'),\delta'\big)$} resp.\
\hbox{$\big(\S^c(\uparrow\! L''),\delta''\big)$} is, by definition, a~dg
coalgebra morphism 
\begin{equation}
\label{Musim behat.}
F : \hbox{$\big(\S^c(\uparrow\! L'),\delta'\big)$} \longrightarrow  \hbox{$\big(\S^c(\uparrow\! L''),\delta''\big)$}.
\end{equation} 
It turns out that~(\ref{Musim behat.}) is the same
as a system 
\begin{equation}
\label{16 vysokych vleku za den}
\textstyle
f_k:\bigotimes^k L' \to L'',\ k \geq 1,
\end{equation} 
of degree $k\!-\!1$ graded antisymmetric linear maps that satisfy an
infinite 
system of equations listed 
e.g.~in~\cite[Section~7.3]{fregier-markl-yau}, whose
explicit form is not needed in the present~article.

Analogously, a (weak) {\em morphism\/} of 
formal operator $\LLL$-algebras with underlying graded commutative
associative algebras $A'$ resp.~$A''$ is given by 
a system~(\ref{16 vysokych vleku za den})
with $L' := A'\lev h \prav$, 
$L'':= A''\lev h \prav$ 
satisfying equations in~\cite[Section~7.3]{fregier-markl-yau}, 
such that each $f_k$ is $\bfk\lev h
\prav$-linear in each variable and, moreover, 
$f_1: A'\lev h\prav \to  A''\lev h\prav$ is 
a morphism of commutative
associative algebras. With this definition, Poisson algebras and their
morphisms form a subcategory of formal operator $\LLL$-algebras with
$l_{2,0}$ a derivation in both variables, and all other $l_{k,n}$'s trivial. 
}

Let us denote by $\Assoc$, $\Linfty$ and $\oLinfty$ the categories
of commutative associative algebras, $\LLL$-algebras and formal 
operator $\LLL$-algebras,
respectively. We clearly have a diagram of functors\rev{54}{%
\[
\xymatrix{\Assoc \ar@/_2em/[r]^I   & \ar@/_2em/[l]_{\Box'} 
        \oLinfty  & \Linfty \ar@/_2em/[l]_{\Box''}
  \ar@/^2em/[l]_{I_0}
}
\]}
in which $\Box'$ forgets the $\LLL$-structure and remembers
underlying commutative  associative algebra only, while the functor $I$ equips a
commutative associative algebra with the trivial \hbox{$\LLL$-structure}. \rev{54}{The
functor $I_0$ interprets an $\LLL$-algebra as a formal  operator
$\LLL$-algebra supported on a trivial graded commutative associative
algebra.}

Finally, $\Box'' :
\oLinfty \to \Linfty$ forgets underlying commutative associative algebra
structure and remembers only the operations $l_{n,k}$ with $n=1$,
cf.~\eqref{Ztratil jsem vlecne lano.} for the notation. 
One is tempted to define another functor $\oLinfty \to \Linfty$ by
putting $h=1$ to~\eqref{Ztratil jsem vlecne lano.}, but the infinite
sum thus obtained might not be well defined. Examples when this
happens are easy to construct.

\section{Formal operator Lie algebras}
\label{V Jicine jsem nebyl kvuli spatne predpovedi pocasi.}
In this section we consider the case of formal operator
algebras over the operad $\Lie$. 
\revnm{redone}
{%
\begin{definition}
\label{Koupim Tereje?}
\rev{55}{A {\em formal operator Lie algebra\/} is a formal operator $\Lie$-algebra of order one in the sense of Definition ~\ref{FormOpAlgDef}.}
\end{definition}

Explicitly, it is given by a graded commutative associative algebra $A$ and a~$\bbk\lev h \prav$-linear Lie bracket
$[-,-]$ on
$A\lev h \prav$ whose restriction to $A \subset A\lev h \prav$ decomposes~as 
\begin{equation}
\label{buseni i ve ctvrtek}
[a',a''] = \sum_{n \geq 1}\ [a',a'']_n \cdot  h^{n-1}, \ a',a'' \in A,
\end{equation}
where $[-,-]_n : A \ot A \to A$ is a 
differential operator of order $n$ in each variable. 
Such an algebra is said to be a {\em formal derivation Lie algebra\/} 
if $[-,-]_n$ is a derivation of order $n$ in each variable.

Trivially, a formal operator Lie algebra $A$ with $[-,-]_n$ vanishing for $n \geq 2$ is an operator $\Lie$ algebra of order one, which includes Poisson algebras as a special case, cf. Example~\ref{Pujdu si pro listky}. 
An example of a formal operator Lie algebra with non-trivial higher-order brackets
is to be presented in Section~\ref{Bude alespon v patek trocha termiky?}.
In that regard, the following result providing a way of constructing bilinear differential operators of arbitrary orders will be quite useful for us:
}
\begin{lemma}
\label{Zitra poletim do Jicina.}
Let $X$ be a graded vector space and
\begin{equation}
\label{Mel bych jet pro vozejk.}
\big\{\Upsilon(-,-)_n^{ij} : \S^i(X) \ot \S^j(X) \longrightarrow
\S(X) \ | \  1 \leq i,j \leq n\big\}
\end{equation}
be a family of \/ $\bbk$-linear maps such that 
\begin{equation}
\label{Kdy to buseni prestane?}
\Upsilon(a',a'')_n^{ij} = (-1)^{|a'||a''|} \cdot  \Upsilon(a'',a')_n^{ji},\ \hbox {
  for }
a'\! \ot\! a'' \in \S^i(X) \ot \S^j(X).
\end{equation}
Then the formula
\begin{align}
\label{Pocasi na Safari bylo strasne .}
[x'_1 x'_2 \cdots x'_s,& \ x''_1 x''_2 \cdots x''_t]_n :=
\\
 \nonumber 
\sum_{1 \leq i,j \leq n} \ &
\sum_{\sigma,\mu}   \epsilon(\sigma)\epsilon(\mu)
x'_{\sigma(1)} \cdots x'_{\sigma(s-i)}
\Upsilon(x'_{\sigma(s-i+1)} \cdots x'_{\sigma(s)},x''_{\mu(1)} \cdots
x''_{\mu(j)})_n^{ij} x''_{\mu(j+1)} \cdots x''_{\mu(t)}
\end{align}
where $\sigma$ runs over $(s-i,i)$-shuffles,
$\mu$ runs over $(j,t\!-\!j)$-shuffles, $x'_1 \cdots x'_s
\in \S^s(X)$, $x''_1 \cdots x''_t \in \S^t(X)$, 
\rev{61,62}{defines a unique
graded antisymmetric 
operation $[-,-]_n : \S(X) \ot \S(X) \to \S(X)$ which is a
derivation of order $n$ in each variable and such that $\Upsilon(-,-)_n^{ij}$ is equal to the restriction of 
$[-,-]_n$ onto $\S^i(X) \ot \S^j(X)$.}
\end{lemma}

\begin{proof}
Existence follows by a direct verification, cf.~the formula at the top of page~374 of
\cite{markl:la}, whereas uniqueness is due to Lemma~\ref{Koupil jsem si tenisky.}.
\end{proof}

Now, given a formal operator Lie algebra $A$, denote for $a,b,c \in A$ and $n\geq 1$
\begin{equation}
\label{Dostal jsem technicak.}
\Jac_n(a,b,c) := \sum_{s+t = n+1} \left\{(-1)^{|a||c|}[a,[b,c]_s]_t +(-1)^{|b||a|}
[b,[c,a]_s]_t + (-1)^{|c||b|}[c,[a,b]_s]_t \right\}.
\end{equation}
Note that the Jacobi identity for $[-,-]$ is equivalent to 
\begin{align}
 \label{JackieChan}
 \Jac_n(a,b,c)=0
\end{align}
holding for all $n\geq 1$ and $a,b,c\in A$. Since $\Jac_n$ can be identified, up to a sign, with 
$\Jac_{3,n}$ of equation ~\eqref{Ve stredu ma byt pekne.}, then by 
\rev{56}{Proposition ~\ref{Zase selhavam.}}, it is a differential operator of order $n$ in each of its three variables.
This observation leads to the following
\rev{60}
{%
\begin{corollary}
\label{Zacinaji vedra.}
Let $A = \S(X)$ for a graded vector space $X$, and 
$[-,-]:A\lev h\prav \otimes A\lev h\prav \to A\lev h\prav$ be an antisymmetric $\bbk\lev h\prav$-linear mapping admitting a decomposition as in \eqref{buseni i ve ctvrtek}.
Then for each $n\geq 1$, \eqref{JackieChan} is satisfied for all $a,b,c \in \S(X)$ if and only if it holds for all $a,b,c \in \S^{\leq n}(X)$. If each of the brackets in \eqref{buseni i ve ctvrtek} is a derivation of the corresponding order with respect to each of the arguments, then it is enough to check this condition for $a,b,c \in \S_+^{\leq n}(X)$.
\end{corollary}
}
\begin{corollary}
\label{Zitra bude cista termika.}
A graded Poisson algebra with underlying graded commutative associative algebra 
$\S(X)$ is uniquely determined by a graded antisymmetric bilinear
map $\langle-,-\rangle : X \ot X \to \S(X)$
whose extension $[-,-] : \S(X) \ot \S(X) \to \S(X)$ 
into a derivation in each variable satisfies
the Jacobi identity on $X \ot X \ot X$.   
\end{corollary}

\begin{proof}
The extension $[-,-]$ of $\langle-,-\rangle$ mentioned in
the corollary is given
by formula~\eqref{Pocasi na Safari bylo
  strasne .}
with 
\[
\Upsilon(-,-)^{ij}_n := 
\begin{cases}
\langle-,-\rangle& \hbox { for $i=j=n=1$, and}
\\
0& \hbox { otherwise.} 
\end{cases}
\]
The only nontrivial Jacobiator~\eqref{Dostal jsem technicak.} 
is $\Jac_1(a,b,c)$, which  
is a derivation of order $1$ in
each variable. Hence, by Corollary
\ref{Zacinaji vedra.},  it is enough to verify that it vanishes for
any $a,b,c  \in \S_+^{\leq 1}(X) = X$.   
\end{proof}

\begin{example}
The basic example illustrating Corollary~\ref{Zitra
    bude cista termika.} is obtained by taking $X$ to be a Lie algebra $L$
 with the Lie bracket $\langle-,-\rangle$.  Note that assignment $L \mapsto
  \S(L)$ yields a left adjoint to the forgetful functor
  from the category of Poisson algebras to the category of Lie
  algebras. By imposing an additional degree shift on $L$ and treating the bracket as a bilinear mapping 
  ${\uparrow\! L\otimes \uparrow\! L\to \uparrow\! L}$ of degree $(-1)$, the same construction returns the Schouten bracket on
  $\ext(L)\simeq \S(\uparrow\! L)$.
\end{example}

\rev{64}{%
\begin{remark}
\label{Musim si vyjednat pojisteni.}
\textcolor{black}{Let us consider the complete topological algebra 
$\tS(X): =\lim_{k} \S(X)/\S^{\geq k}(X)$
where,
\[
\tS^{\geq k}(X) := \bigoplus_{n \geq k}\S^n(X).
\]
It is not difficult to prove that each differential operator $\nabla$ of order 
$\leq r$ defined on $\S(X)$ uniquely extends into a continuous linear
map $\tS(X) \to \tS(X)$. Moreover, any algebraic equation that
$\nabla$ satisfies is, by continuity, satisfied also for its
extension. In particular, any Poisson algebra with underlying
commutative associative algebra $\S(X)$ determines a Poisson algebra with
underlying commutative associative algebra $\tS(X)$.}
\end{remark}
}

\section{${L}_\infty$-bialgebras, $\IBL_\infty$-algebras, the big and the 
  superbig bracket}
\label{Pujdu si zabehat ale moc se mi nechce.}
\revnm{}
{The section is devoted to a discussion of a certain formal operator Lie algebra 
with non-vanishing higher-order brackets arising in the context of the deformation theory of involutive Lie bialgebras. We precede the actual discussion of the subject 
with a brief recollection of some auxiliary facts concerning linear algebra in the graded setting,
followed by a reminder on the big bracket algebra.
}
%

Let $W$ be a graded vector space, $\susp W$ its
suspension and $\uparrow 
: W \to \susp W$ the obvious degree $+1$ isomorphism. 
The exterior (aka Grassmann) algebra  generated by $W$ is the quotient 
\[
\ext(W) := T(W) / {\, \EuScript I}
\] 
of the tensor algebra $T(W)$ modulo the ideal $\, {\EuScript I}$  generated by the
expressions  
\[
w' \ot w'' + \sign{|w'||w''|} w'' \ot w'
\]
with homogeneous $w',w'' \in W$. One has a sequence $\{f_n\}_{n \geq 0}$
of \rev{65}{$\bbk$-linear  degree $n$ d\'ecalage isomorphisms}
\[
f_n : \ext^n(W) \longrightarrow  \S^n(\susp W),
\]
between the components spanned by the  products of generators of length $n$,
inductively defined by $f_0:= \id_\bbk$,  
\hbox{$f_1(w) := \susp w$} for $w \in W$ while
\begin{equation}
\label{Dnes s Jarkou na vyhlidku.}
f_{a+b} (u \land v) := \sign{b |u|} f_a(u)\cdot f_b(v), \ \hbox {for }
u \in \ext^a(W), \ v \in \ext^b(W), \ a,b \geq 1.
\end{equation}
The family  $\{f_n\}_{n \geq 0}$ assembles into an
isomorphism of non-graded algebras
\begin{equation}
\label{Videl jsem Tereje.}
f: \ext(W) \cong \S(\susp W)
\end{equation}
whose components satisfy~\eqref{Dnes s Jarkou na vyhlidku.}.
Likewise one defines a sequence $\{g_n\}_{n \geq 0}$
of linear  degree $2n$ isomorphisms
\[
g_n : \S^n(\desusp W) \longrightarrow  \S^n(\susp W),
\]
inductively by $g_0 := \id_\bbk$, $g_1(\desusp w):= \susp w$ for $w
\in W$, while  
\begin{equation*}
g_{a+b} (u \cdot v) :=  g_a(u)\cdot g_b(v), \ \hbox {for }
u \in \S^a(\desusp W), \ v \in \S^b(\desusp W), \ a,b \geq 1.
\end{equation*}
The family  $\{g_n\}_{n \geq 0}$ again gives rise  to an
isomorphism 
\begin{equation*}
g: \S(\desusp W) \cong \S(\susp W).
\end{equation*}

\subsection{The big bracket}\label{Jak se bude Terej ridit?}
Let $V$ be a graded $\bbk$-vector space, $\susp V$ its suspension, and
$\desusp V^*$ the desuspension of its linear dual.
In what follows, we will assume that $V$ is finite-dimensional. This assumption can be relaxed using the compact-linear topology on dual spaces and completed tensor products~\cite[Chapter~1]{markl12:_defor},
but since this generalization brings nothing conceptually new, we will
stick to the finite-dimensional case.
Take, in Corollary~\ref{Zitra bude cista termika.}, $X :=  \susp V \oplus\!
\desusp V^*$ 
and  $\langle-,-\rangle : X \ot X \to \S(X)$
given by
\begin{equation}
\label{Zitra s Jarkou na Petrin}
\langle \alpha, u\rangle = - \sign{|u||\alpha|} \langle u, \alpha
\rangle := \alpha(u) \in \bbk = \S^0(X) \subset \S(X)
\end{equation}
for $\alpha \in  \desusp V^*$, $u \in \susp V$, while
\rev{66}{$\langle \susp V,\susp V \rangle = \langle \desusp V^* ,\desusp V^* \rangle := 0$.}
Notice that $|u|$ must equal $|\alpha|$ for $\langle \alpha, u\rangle$ 
above to be nonzero.
The assumptions of  Corollary~\ref{Zitra bude cista termika.} are easy to
verify. Denote by $\big(\S( \susp V \oplus\!
\desusp V^*),[-,-]\big)$ the resulting Poisson algebra.

Take $B(V) :=\ \desusp^2\S(\susp V \oplus\! \susp V^*)$ and equip $B(V)$ with the
Lie bracket $\{-,-\}$ induced from the one on \hbox{$\S( \susp V \oplus\!
\desusp V^*)$} by the vector space isomorphism
\begin{equation}
\label{Koupil jsem olej.}
\S(\susp V \oplus\! \desusp V^*) \cong 
\S( \susp V) \otimes \S(\desusp V^*)
\stackrel{\id \ot g}\longrightarrow \S( \susp V) \otimes \S(\susp V^*)
\cong \S(\susp V \oplus\! \susp V^*)
\stackrel{\desusp^{\ 2}}\longrightarrow  \desusp^2   \S(\susp V
\oplus\! \susp V^*)
=B(V).
\end{equation}
Since the above isomorphism
involves even degree shifts only, the related signs issues are
trivial. The bracket  $\{-,-\}$  thus constructed is the {\em big
  bracket\/} of~\cite{lecomte1990modules}. Applying~\eqref{Videl
  jsem Tereje.} gives its standard presentation
\begin{equation}
\label{Koupil jsem Tereje.}
\big(B(V)^*,\{-,-\}\big) \cong \big(\ext^{*+2}(V \oplus V^*), \{-,-\}\big),
\end{equation}
cf.~\cite[Equation~(2)]{Krav}.

Let $B^p_q(V) \subset B(V)$ be the subspace of $B(V)$ spanned, in the
presentation~(\ref{Koupil jsem Tereje.}), by the exterior products of $p$
elements of $V$ and $q$ elements of $V^*$. Maurer-Cartan
elements in $B^1_2(V) \oplus B^2_1(V)$ describe  Lie bialgebras, and
those in $B^1_2(V) \oplus B^2_1(V) \oplus B^0_3(V)$ Lie
quasi-bialgebras~\cite[Section~3]{kravchenko1}. 

To make room for $L_\infty$-bialgebras, we define
$\gbiLie$ to be  the Poisson subalgebra of $B(V)$ with the underlying space
\begin{equation}
\label{Kdy se zase svezu v Tereji?}
\gbiLie := \bigoplus_{  p,q
\geq 1, \ p+q \geq 3} B(V)^p_q .
\end{equation}
Its closure $\tgbiLie$\, in\, $\desusp^2\tS(\susp V \oplus\! \susp
V^*)$  has an induced Poisson structure by
Remark~\ref{Musim si vyjednat pojisteni.}. 
Maurer-Cartan elements in $\tgbiLie$ are known to describe
$L_\infty$-bialgebras~\cite[Subsection~4.4]{kravchenko1}.

\rev{67}{%
\begin{remark}\label{Na chalupu pojedeme asi az zitra.}
Recall that works~\cite{markl:JHRS10,merkulov-valletteI} provide an
explicit method of constructing $\LLL$-algebras controlling
deformations of algebraic structures starting with a cofibrant or, if
it exists, minimal model of the governing operad- or PROP-like object. For structures with
quadratic relations, the resulting $\LLL$-algebra is
actually a dg-Lie algebra~\cite[Proposition~2]{markl:JHRS10}.  

It turns out that $\tgbiLie$ is precisely the Lie algebra capturing
deformations of Lie bialgebras constructed using the explicit minimal
model of the properad $\shadew$ for Lie bialgebras.
Indeed, by definition, the degree $n$ part of the
completion $\tgbiLie$
consists of infinite families 
\begin{equation}
\label{Nevim jestli to stihnu.}
\big\{f^p_q \in B^p_q(V)\ | \ p,q \geq 1, \
p+q \geq 3\big\},
\end{equation}
where $f^p_q$, interpreted as a linear map $\ext^q V \to
\ext^p V$,  raises the degree by $(p+q)-n-2$.  

On the other hand, according to~\cite[Example 20]{mv},
cf. also~\cite[Eq.~(2)]{MW}, the minimal model of the properad for Lie
bialgebras is generated by the collection $\Upsilon$ spanned by the
generators $\xi^p_q$, $p,q \geq 1$,
$p+q \geq 3$, with $q$ fully antisymmetric inputs and $p$ fully
antisymmetric outputs, placed in degree
$p+q-3$. By~\cite[Eq.~(9)]{markl:JHRS10}, the degree $n$ piece
of the  Lie algebra capturing
deformations of Lie bialgebras is given by
\[
C^n_{{\it biLie}}(V;V) 
:= {\it Lin}^{1-n}_{\Sigma-\Sigma}(\Upsilon,{\tt End}_V).
\]
Here ${\tt End}_V$ denotes the endomorphism properad of $V$ and ${\it
  Lin}^{1-n}_{\Sigma-\Sigma}(\Upsilon,{\tt End}_V)$ the vector space
of $\Sigma-\Sigma$-equivariant degree $1-n$  maps $\Upsilon \to
{\tt End}_V$ of bicollections. 
The elements of $C^n_{{\it biLie}}(V;V)$
are precisely the families~(\ref{Nevim jestli to stihnu.}), the map
$f^p_q$ being the image of the generator $\xi^p_q$ of~$\Upsilon$. This
verifies that
$\tgbiLie \cong C_{{\it biLie}}(V;V)$ as graded vector spaces.

The Lie bracket on $ C_{{\it
    biLie}}(V;V)$ is determined by the differential of the
minimal model of $\shadew$ by formula~(11) of~\cite{markl:JHRS10}. Applying that formula
to the differential  described pictorially in~\cite[Eq.~(3)]{MW}, we
conclude that the bracket $\{{f'}^p_q,{f''}^r_s\}$ of generators in~\eqref{Nevim
  jestli to stihnu.} is the sum of all possible
contractions of one output of ${f''}^r_s$ with one input of ${f'}^p_q$, minus  
 the sum of all 
contractions of one output of ${f'}^p_q$ with one input of ${f''}^r_s$, with
appropriate signs. This is precisely what the bracket in $\tgbiLie$ does.
\end{remark}
}

We are going to give, following~\cite[Subsection~3.1]{CMW}, 
an interesting alternative
description of $\tgbiLie$. Choose a~basis $(e_1,e_2,\ldots)$ of $V$
and its dual basis $(\alpha^1,\alpha^2,\ldots)$ of $V^*$. Let
\[
(\psi_1,\psi_2,\ldots) 
:= (\susp e_1,  \susp e_2,\ldots) \hbox { resp. }
(\eta^1,\eta^2,\ldots) 
:= (\susp \alpha^1, \susp \alpha^2,\ldots)
\]  
be the corresponding bases of $\susp V$ resp.~$\susp V^*$.
Elements of $\tS(\susp V \oplus\! \susp
V^*)$ then appear as power series $f \in
\bbk\lev \psi,\eta\prav  :=
\bbk\lev \psi_1,\psi_2,\ldots,\eta^1,\eta^2,\ldots\prav $. In this language, 
$\tgbiLie$ consists of power series $f \in
\bbk\lev \psi,\eta\prav $ satisfying the
boundary conditions
\[
f(\psi,\eta)  \in {\mathfrak m}^3,\ f(\psi,\eta)|_{\psi_i = 0} = 0,\
 f(\psi,\eta)|_{\eta^j=0} = 0,\
i,j = 1,2,\ldots,
\]
where $ {\mathfrak m}$ is the maximal ideal of the complete local ring
$\bbk\lev \psi,\eta\prav $, assigned the degree two less than the degree of
$f$ in $\bbk\lev \psi,\eta\prav $.
With this convention, the big bracket is expressed as
\[
\{f,g\} := \sum_{i = 1,2,\ldots}
\left(
 \frac{\partial f}{\partial \eta^i} 
 \frac{\partial g}{\partial \psi_i}
- \sign{|f|\cdot |g|} 
  \frac{\partial g}{\partial \eta^i} 
 \frac{\partial f}{\partial \psi_i}
\right).
\]
The big bracket turns out to be the semiclassical limit, in the sense
of Definition~\ref{V podvecer musim na chalupu.}, 
Part~\ref{general}, of the superbig
bracket  constructed in the following subsection.

\subsection{The superbig bracket}
\label{Bude alespon v patek trocha termiky?}
We will construct a formal
operator Lie algebra $(B(V)\lev \hbar \prav, \sleft -,-\sright)$  deforming the big bracket algebra recalled earlier. 
\rev{70}{%
Similarly to the big bracket, the only input required for constructing $\sleft -,-\sright$ is the natural pairing between the graded vector space $V$ and its linear dual $V^*$, extended to the corresponding exterior algebras. 
The Lie algebra $\gIBL$, whose completion $\tgIBL$  constructed in~\cite{CMW} controls deformations of involutive Lie bialgebras, is going to be contained in $B(V)\lev \hbar \prav$ 
as a subalgebra. 
Consequently, both $\gIBL$ and $\tgIBL$ are intrinsic in the sense of~\cite{St}. The material below may also be regarded as an elementary version of the construction of $\tgIBL$
given in~\cite{CMW}.
}

Let $V$ be a graded $\bbk$-vector space and \rev{68}{$X=\susp V \oplus\! \desusp V^*$.}
Using the notation of Lemma~\ref{Zitra poletim do Jicina.}, we define
\[
\Upsilon(-,-)^{ij}_n : \S^i(X) \ot \S^j(X) \longmapsto \S(X)
\]
for all $n\geq 1$ and $1\leq i,j\leq n$ as follows. 
First, for $n\geq 1$ and $i\neq n$ or $j\neq n$, $\Upsilon(-,-)^{ij}_n:=0$.
Next, to define the terms for $n\geq 1$ and $i=j=n$, we invoke the canonical isomorphism 
\[
\S^n (X) \cong \bigoplus _{p+q = n} \S^p(\susp V) \ot \S^q(\desusp
V^*) 
\cong \bigoplus _{p+q = n} \S^p(\susp V) \ot \S^q(\desusp V)^* 
\]
and set
\[
\Upsilon\big(\S^p(\susp V) \ot \S^q(\desusp V^*) ,\  \S^s(\susp V) 
\ot \S^t(\desusp V^*)
\big)^{nn}_n := 0
\]
if $(p,q;s,t) \not\in \big\{(n,0;0,n), (0,n;n,0)\big\}$, while
\[
\Upsilon(1 \ot \alpha, u \ot 1)^{nn}_n = -(-1)^{|\alpha|\cdot |u|} 
\Upsilon(u\ot 1, 1 \ot \alpha)^{nn}_n 
:= 
\alpha(u) \in \bbk = \S^0(X) \subset \S(X),
\]
for
$\alpha \in \S^n(\desusp V)^*,\ u \in \S^n(\susp V)$. 
Notice that $\Upsilon(-,-)^{11}_1$ is the
map $\langle -,- \rangle : X \ot X \to \S(X)$ in~\eqref{Zitra s Jarkou na Petrin}. Formula~\eqref{Pocasi na Safari bylo strasne .} then
defines a map 
\rev{70}{%
\begin{equation}
\label{Mely bychom jet na chalupu.}
[-,-]_n : \S(X) \ot \S(X) \to \S(X),\ X  =\susp V \oplus\! \desusp V^*,
\end{equation}}%
which is a differential operator of order $n$ in each variable. 
The main result of this subsection~is

\begin{theorem}
\label{Tento vikend bude pocasi velmi hrozne.}
Under the above notation, the formula
\begin{equation}
\label{Za tyden letim na Vivat tour.}
[a',a''] := [a',a'']_1 + [a',a'']_2 \cdot h + [a',a'']_3 \cdot h^2 +
\cdots,\
a',a'' \in \S(X),
\end{equation}
equips the graded commutative associative algebra $S(X)\lev h\prav  =
 \S(\susp V \oplus \desusp V^*)\lev h\prav $ 
with the structure of a formal operator Lie algebra.
\end{theorem}

\begin{proof}
The only property which is not obvious is the Jacobi identity for the
bracket. Rather that verifying it directly, we identify
$\big(\S(X)\lev h\prav ,[-,-]\big)$ with \rev{69}{a slight} generalization of 
a~construction in~\cite{merkulov-valletteI}.

For a pair of elements $a'$ and $a''$ of a properad $\oP$, Merkulov
and Vallette denoted,
in~\cite[Subsection~2.2]{merkulov-valletteI}, 
by $a' \circ a''$ the sum of all possible composites
of $a'$ by $a''$  in
$\oP$ along any $2$-leveled graph with two vertices. They proved
that the commutator 
\[
[a',a''] := a'   \circ a'' - (-1)^{|a'|\cdot |a''|}  a''   \circ
a'
\]
is a Lie bracket on the total space $\bigoplus \oP :=
\bigoplus_{m,n \geq 0} \oP(m,n)$, and that it induces a Lie bracket on the
total space $\bigoplus \oP^{\Sigma}$ of invariants. Let us modify their
definition of the $\circ$-operation into
\begin{equation}
\label{Letal jsem 4 hodiny na Tereji.}
a' \circ_h a'' = a' \circ_1 a'' + (a' \circ_2 a'') h  + (a' \circ_3 a'')
h^2+\cdots
\end{equation}    
where $a' \circ_k a''$, $k \geq 1$, is the sum of all possible composites
of $a'$ by $a''$ along $2$-leveled graphs with two vertices connected
by $k$ edges. The proof of~\cite[Theorem~8]{merkulov-valletteI} 
can be easily modified to show that also the commutator $[-,-]_h$ of the
$\circ_h$-operation is a Lie bracket on the $\bbk\lev h\prav $-linear 
extension  $\bigoplus
\oP \lev h\prav $ of the total space of $\oP$, which in turn induces a Lie
algebra structure on the $\bbk\lev h\prav $-linear
extension $\bigoplus \oP^\Sigma \lev h\prav $ 
of the space of invariants.

Let us apply the above constructions to the endomorphism properad
$\oP := \End_{\uparrow V}$ of the suspension of $V$. Recall that
\[
\End_{\uparrow V}\textstyle
(m,n) = \Lin\,\big(\bigotimes^m \susp V, \bigotimes^n \susp V\big),\
m,n \geq 0,
\]
the space of $\bbk$-linear maps $\bigotimes^m \susp V \to \bigotimes^n
\susp V$. One has the canonical isomorphism 
\[
\textstyle
\Lin\,\big(\bigotimes^m \susp V, \bigotimes^n \susp V\big)
\cong
\bigotimes^m (\susp V) \ot \bigotimes^n( \desusp V^*)
\]
which translates the properadic composition of $\End_{\uparrow V}$ into
the contraction via the canonical pairing between $\susp V$ and
$\desusp V^*$.
For the space of invariants one gets 
\[
\textstyle
\Lin\,\big(\bigotimes^m \susp V, \bigotimes^n \susp V\big)^{\Sigma_m
  \times \Sigma_n}
\cong \Lin\,\big(\S^m (\susp V), \S^n( \susp V)\big) \cong
\S^m (\susp V) \ot \S^n( \desusp V^*)
\]
therefore
\[
\bigoplus \big(\End_{\uparrow V}\big)^\Sigma \cong \S(X).
\]
It is easy to check that, under the canonical isomorphism above, the
commutator of the $\circ_h$-product in~\eqref{Letal jsem 4 hodiny na
  Tereji.} becomes the bracket in Theorem~\ref{Tento vikend bude
  pocasi velmi hrozne.}.  
\end{proof}

\rev{70}
{%
Let  $n \geq 1$ and, as before,  $B(V)  =\ \desusp^2\S(\susp V \oplus\! \susp
V^*)$. Denote by $\sleft -,-\sright_n : B(V) \ot B(V) \to B(V)$ the 
linear map induced from the operation $[-,-]_n$ in~(\ref{Mely bychom
  jet na chalupu.}) via isomorphism~\eqref{Koupil jsem olej.}. Since
$[-,-]_n$ is a differential operator of order $n$ in each
variable, so is  $\sleft -,-\sright_n$. A simple degree
count shows that $\sleft -,-\sright_n$ is of degree $2-2n$.
Furthermore, $\sleft -,-\sright_1$ matches the big bracket
$\{-,-\}$ recalled in Subsection~\ref{Jak se bude Terej ridit?}.

We want to assemble all $\sleft -,-\sright_n$'s into  
a degree $0$ operation. To this end we introduce 
a formal variable $\hbar$ of degree $+2$. Denoting $B(V)\lev \hbar
\prav := B(V)\otimes \bbk\lev \hbar \prav$,  the formula 
\[
\sleft -,-\sright := \{-,-\} + \sleft -,-\sright_2 \hbar +
\sleft -,-\sright_3 \hbar^2 + \cdots 
\] 
indeed defines a degree $0$ $\bbk$-linear map $\sleft -,-\sright : B(V) \ot B(V) \to B(V)\lev \hbar
\prav$. We will denote its $\bbk \lev\hbar\prav$-bilinear extension by
the same symbol  and call it  the {\em
  superbig bracket\/}.

Notice that $\sleft -,-\sright$ defined this way corresponds to the
Lie bracket~\eqref{Za tyden letim na Vivat tour.} under the obvious
isomorphism  $\hbox{$\S(\susp V \oplus\! \desusp V^*)\lev h\prav$}  \cong B(V)\lev \hbar
\prav$ of vector spaces induced by~\eqref{Koupil jsem olej.}. Since
this isomorphism involves even degree shifts only, it does not
influence the signs involved, thus $(B(V)\lev \hbar
\prav, \sleft -,-\sright)$ is a Lie algebra as well.
By definition, for $a',a'' \in B(V)$,
$\sleft a',a''\sright_n=0$ if  $n$ is big enough. 
The superbig bracket is therefore
a {\em global deformation\/} of the big bracket.}

Let us denote
by $\gIBL$  the Poisson subalgebra of $\big(B(V)\lev \hbar\prav , \sleft -,-\sright\big)$ given by
\[
\gIBL := \bigoplus_{  p,q
\geq 1, \ p+q \geq 3} B(V)^p_q\lev \hbar\prav ,
\]
where $B(V)^p_q$ has the same meaning as in~(\ref{Kdy se zase svezu v
  Tereji?}). 
Its closure $\tgIBL$\, in the completed tensor product $\desusp^2\tS(\susp V \oplus\! \susp
V^*) \widehat \ot \bbk\lev \hbar\prav $ bears an induced Poisson structure by
Remark~\ref{Musim si vyjednat pojisteni.}. 

The following description of $\tgIBL$ is taken almost verbatim
from~\cite[Subsection~3.1]{CMW}; the notation is the same as in the
second half of Subsection~\ref{Jak se bude Terej ridit?}.  The authors
of~\cite{CMW} interpreted 
the elements of $\tgIBL$
as power series
$f \in \bbk\lev \psi,\eta,\hbar\prav  :=
\bbk\lev \psi_1,\psi_2,\ldots,\eta^1,\eta^2,\ldots,\hbar\prav $ subject to the
conditions
\[
f(\psi,\eta,\hbar)|_{\hbar=0}  \in {\mathfrak m}^3,\ f(\psi,\eta,\hbar)|_{\psi_i =
  0} = 0\
\hbox { and } \
 f(\psi,\eta,\hbar)|_{\eta^j=0} = 0,\
i,j = 1,2,\ldots,
\]
where $ {\mathfrak m}$ is the maximal ideal in
$\bbk\lev \psi,\eta,\hbar\prav $. As in Subsection~\ref{Jak se bude Terej
  ridit?}, such an $f$ is taken with degree two less than its actual
degree in $\bbk\lev \psi,\eta,\hbar\prav $. The induced superbig bracket can then be written~as
\[
\sleft f,g\sright = f *_\hbar g  - \sign{|f|\cdot |g|}  g *_\hbar f 
\]
where 
\begin{equation}
\label{Pojedu zitra na kole?}
f *_\hbar g  :=  \sum_{n=1}^\infty \frac{\hbar^{n-1}}{n!}  \sum_{\Rada i1n}
\epsilon \cdot
 \frac{\partial^n f}{\partial \eta^{i_1} \cdots \partial \eta^{i_n}} 
 \frac{\partial^n g}{\partial \psi_{i_1} \cdots \partial \psi_{i_n}},
\end{equation}
with $\epsilon$ the Koszul sign of the permutation
\[
\rada {\eta^{i_1}}{\eta^{i_n}},\rada{\psi_{i_1}}{\psi_{i_n}}
\ \longmapsto \
\rada{\eta^{i_1},\psi_{i_1}}{\eta^{i_n},\psi_{i_n}}.
\]
Notice that, unlike the authors of \cite{CMW}, we did not allow $n=0$
in the above sum. The operation $*_\hbar$ in~\eqref{Pojedu zitra na kole?}
decomposes as
\[
f *_\hbar g = f *_1 g  + f *_2 g \cdot \hbar  + f *_3 g \cdot \hbar^2  + \cdots, 
\]
where
\[
*_n := \sum_{\Rada i1n}
\epsilon \cdot
 \frac{\partial^n f}{\partial \eta^{i_1} \cdots \partial \eta^{i_n}} 
 \frac{\partial^n g}{\partial \psi_{i_1} \cdots \partial \psi_{i_n}},
 \ n \geq 1,
\]
is a differential operator of order $n$ in each
variable, $k \geq 1$. The $*_\hbar$-product can thus serve as an
example of a {\em formal operator Lie-admissible algebra\/}.  
\rev{70}{%
\begin{remark}
\label{MinModRmk}
The Lie algebra $\tgIBL$ is isomorphic to the Lie algebra controlling
deformations of $\IBL$-algebras (i.e.\ involutive Lie bialgebras),
constructed using the recipe
of~\cite{markl:JHRS10,merkulov-valletteI}, from the
explicit minimal model for the properad  $\shade$ for $\IBL$-algebras,
which is described 
in~\cite[Subsection~2.3]{CMW}. The scheme of verification is the same as in
Remark~\ref{Na chalupu pojedeme asi az zitra.}; a useful pictorial
presentation of the differential of the minimal model of $\shade$ is  given
in~\cite[Eq.~(7)]{CMW}, cf.~also~\cite[Eq.~(5)]{MW}.
If $V$ has a differential, then
$\tgIBL$ receives an induced differential~$\delta$, and the Maurer-Cartan
elements in the dg-Lie algebra $(\tgIBL,\delta)$ are $\IBL_\infty$-algebras.    
\end{remark}
}

\section{Miscellany}
\label{Jaruska je na chalupe az do soboty.}
We take a brief look at some known constructions from the point of view of (formal) operator algebras.
\subsection{Terilla's quantization}
\label{Kam jsem dal tu masticku?}
\noindent
\revnm{reworded}
{%
Let $V$ be a graded vector space and $P:= \Lin\big(\S(V),\tS(V)\big)$.
In~\cite{Terilla}  Terilla studied a deformation quantization of a graded commutative associative 
algebra structure on $P$, where the algebra product is given by the symmetrization and completion of the simple associative product
\[
 \circ_0:\Lin(V^{\otimes i}, V^{\otimes j})\otimes \Lin(V^{\otimes m}, V^{\otimes n})\to \Lin(V^{\otimes (i+m)}, V^{\otimes (j+n)})
\]
defined for all $i,j,m,n\geq 0$ by setting
\begin{align}
\label{SideBySide}
 (f\circ_0 g)(v_1,\dots, v_{i+m}):=f(v_1,\dots, v_i)\otimes g(v_{i+1},\dots, v_{j+m}).
\end{align}
Let $\{e_1,e_2\dots\}$ be a basis of $V$ and $\{\alpha_1,\alpha_2\dots\}$ be the corresponding dual basis of $V^*$.
Upon identifying the completed symmetric algebra $\tS(V \oplus V^*)$ with a subalgebra of $P$, 
the product $f\circ_0 g$ for any $f,g\in \tS(V \oplus V^*)$ becomes the standard product of the corresponding power series.
Terilla's deformation of $\tS(V \oplus V^*)$ is then defined by setting
\[
f \star g  :=  \sum_{k=0}^\infty \frac{h^{k}}{k!}  \sum_{\Rada i1k}
\epsilon \cdot
 \frac{\partial^k f}{\partial \alpha^{i_1} \cdots \partial \alpha^{i_k}} 
 \frac{\partial^k g}{\partial e_{i_1} \cdots \partial e_{i_k}},
\]
for any $f,g\in \tS(V \oplus V^*)=\bbk\lev e_1, e_2\dots\alpha_1,\alpha_2\dots\prav$,
and extending it further to $\tS(V \oplus V^*)\lev h\prav$ by $\bbk\lev h\prav$-bilinearity.
In the above formula, $\epsilon$ is the Koszul sign of the permutation
\[
\rada {\alpha^{i_1}}{\alpha^{i_k}},\rada{e_{i_1}}{e_{i_k}}
\ \longmapsto \
\rada{\alpha^{i_1},e_{i_1}}{\alpha^{i_k},e_{i_k}}.
\]
Notice that $f \star g$ is of the form
\begin{equation}
\label{Dnes jdu na Brezanku.}
f \star g = f \circ_0 g +   f \circ_1 g \cdot h +  f \circ_2 g \cdot
h^2 + \cdots, 
\end{equation}
where each $\circ_i$ is a $\bbk$-bilinear differential operator of order $i$ in each variable.
}

The above construction provides an example of a formal
operator $\Ass$-algebra of order $0$, in the sense of Definition~\ref{FormOpAlgDef},
Part~\ref{general}, with  $\Ass$ being the operad governing
associative algebras. As illustrated above, such a structure 
 consists of a graded commutative associative
algebra~$A$ carrying a \hbox{$\bbk\lev h\prav$-bilinear} associative
multiplication 
$\star$ on
$A\lev h\prav $ whose restriction to $A \subset A\lev h\prav $ decomposes~as 
\begin{equation}
\label{Uz je tech nemoci moc.}
a'\star a'' = \sum_{n \geq 0} a' \star_n a'' \cdot  h^{n}, \ a',a'' \in A,
\end{equation}
where $\star_n : A \ot A \to A$ is a 
differential operator of order $n$ in both variables.

It is worth pointing out that the relation between the orders of $\star_n$ and the powers of $h$ differs from the one
in~\eqref{buseni i ve ctvrtek} where the power series starts with the
differential operator of degree $\leq 1$. The associator
$
\Asss(a,b,c) : = a \star (b \star c) -  (a \star b) \star c
$
decomposes as
\[
\Asss(a,b,c) = \sum_{n \geq 0} \Asss_n(a,b,c)\cdot  h^n,
\]
where each
\begin{equation}
\label{Zase mi vynechava srdce.}
\Asss_n(a,b,c) : = \sum_{s+t=n} a \star_s (b \star_t c) -  
(a \star_s b) \star_t c
\end{equation}
is a differential operator of order $n$
by~\cite[Proposition~1]{markl:la}.

\subsection{Operator Leibniz algebras}
Recall that a (left)
\textit{Leibniz algebra} structure on a graded $\bbk$-vector space
$V$, also known as a \textit{Loday algebra}, is a 
$\bbk$-bilinear mapping $[-,-]: V\otimes V
\to V$ subject to the (left) Leibniz rule
\[
 [[a,b],c]=[a,[b,c]]+  \sign{|a|\cdot |b|}  [b,[a,c]] 
\]
for all $a,b,c\in V$. In particular, any Lie algebra is trivially a
Leibniz algebra, where the above bracket happens to be graded skew-symmetric. 
The concept was originally conceived as a 
noncommutative analogue of a Lie algebra in \cite{Bloh}
and was 
further elaborated in \cite{Loday}.

\begin{example} 
\label{KanEx}
The following construction, originally due to
  Kanatchikov \cite{Kanatchikov}, arises in the context
  of the De Donder\textendash Weyl covariant Hamiltonian formulation of classical
  field theory.
It might serve as an example of an operator Leibniz algebra.
 Let $E$ be a smooth manifold, regarded as a \textit{polymomentum phase
  space}, 
equipped with a multisymplectic form $\Omega$, that is, 
 a closed $(k+1)$-form subject to the non-degeneracy condition
 \[
  X\lrcorner \Omega = 0\,\Rightarrow X=0
 \]
for any vector field $X$ on $E$. The special case of $k=1$ corresponds to the ordinary symplectic setting.
 
 Here, $E$ is to be thought of as the total space of a fiber bundle $E\to M$ over an $n$-dimensional space-time manifold $M$.
 Locally, $\Omega$ provides a splitting of a sufficiently small coordinate chart on $E$ into the \textit{horizontal} (the space-time coordinates)
 and the \textit{vertical} (the field variables and the corresponding polymomenta) directions.
 In particular, that enables one to single out the vertical component $d^V$ of the de Rham differential $d$ on $E$.
 The crucial feature of this set-up is that for any $p$-form $F$ on $E$, there is a vertical multivector-valued horizontal $1$-form $X_F$
 satisfying $X_F\lrcorner \Omega=d^V F$. Now, taking $A$ to be the de Rham algebra $A_{\it dR}^*(E)$ with the exterior product and setting
 \[
[F,G]=(-1)^{n-|F|}X_F\lrcorner d^V G, \ \hbox { for } G \in A_{\it dR}^*(E),
 \]
 gives rise to an operator Leibniz algebra. A peculiar asymmetry of
 this bracket is manifest. Namely, the bracket satisfies the Leibniz rule with respect to the
 first argument, and it is a differential operator $[F,-]$ of
 order $n-p$ with respect to the second argument whenever $F$ is a $p$-form.
 In that regard, the above construction presents an example of an operator algebra over the corresponding operad of Leibniz algebras.
\end{example}

\subsection{F-manifolds} 
A generalization of Poisson manifolds is
provided by the following algebraic abstraction of {\em F-manifolds\/} of
Hertling and Manin~\cite{HM}. What we mean is a commutative
associative algebra $A$ equipped with a Lie algebra product, but
the standard derivation property
\begin{equation*}
0=[a'_1a'_2,a''] - a'_1[a'_2,a''] -  \sign{|a'_1| \cdot |a'_2|} a'_2[a'_1,a'']
\end{equation*}
of Poisson algebras replaced by\footnote{Here and below, the possibly decorated symbol
  $a$ will denote an element of $A$.}
\begin{align}
0=&\ \nonumber 
[a'_1a'_2,a_1''a''_2] - 
a'_1[a'_2,a_1''a''_2] -  \sign{|a'_1| \cdot
    |a'_2|}a'_2[a'_1,a_1''a''_2] 
-[a'_1a'_2,a_1'']a''_2 
\\
\label{Kdy si zase poletam?}
&\ - \sign{|a''_1| \cdot
  |a''_2|}[a'_1a'_2,a_2'']a''_1 +a'_1[a'_2,a''_1]a''_2 + \sign{|a'_1| \cdot
  |a'_2|}a'_2[a'_1,a''_1]a''_2
\\ \nonumber
&\ +
 \sign{|a''_1| \cdot |a''_2|}a'_1[a'_2,a''_2]a''_1+
 \sign{|a'_1| \cdot |a'_2| + |a''_1| \cdot |a''_2|} a'_2[a'_1,a''_2]a''_1.
\end{align}
The above condition says that the linear map $[-,-]: A \ot A \to A$
is a bidifferential operator in the sense we introduce below.

Let $A$ be a commutative associative algebra and $\nabla: A \ot A \to
A$ a linear map. To shorten the formulas, we will write
$\nabla(a',a'')$ for $\nabla(a' \ot a'')$. Inspired
by the scheme~(\ref{Nejak mne boli v krku}), we define the {\em
  bideviations \/} $\FD n: A^{\ot n} \ot A^{\ot n} \to A$, $n \geq 1$
inductively as
\[
\FD1(a',a'') := \nabla(a',a''),
\]
while, for  $n \geq 2$,
\begin{align*}
\nonumber 
\FD{n+1}(&\rada {a'_1a'_2}{a'_{n+1}},\rada{a''_1}{a''_na''_{n+1}}):=
\\+ &a'_1\FD n(\rada {a'_2}{a'_n},\rada{a''_1}{a''_na''_{n+1}}) 
+  \sign{|a'_1| \cdot    |a'_2|}
a'_2\FD n(a'_1,\rada {a'_3}{a'_n},\rada{a''_1}{a''_na''_{n+1}})
\\
+ &\FD{n}(\rada {a'_1a'_2}{a'_{n+1}},\rada{a''_1}{a''_{n}})a''_{n+1}
+ \sign{|a''_n| \cdot
  |a''_{n+1}|}\FD{n}(\rada {a'_1a'_2}{a'_{n+1}},\rada{a''_1}{a''_{n-1},a''_{n+1}})a''_{n}
\\
- & 
a'_1\FD{n}(\rada {a'_2}{a'_{n+1}},\rada{a''_1}{a''_{n}})a''_{n+1}
- \sign{|a'_1| \cdot |a'_2|}a'_2\FD{n}(a'_1,\rada {a'_3}{a'_{n+1}},\rada{a''_1}{a''_{n}})a''_{n+1} 
\\
- & \sign{|a''_n| \cdot
  |a''_{n+1}|}a'_1\FD{n}(\rada
  {a'_2}{a'_{n+1}},\rada{a''_1}{a''_{n-1}},a''_{n+1})a''_{n}
\\
-& \sign{|a'_1| \cdot |a'_2| + |a''_n| \cdot
  |a''_{n+1}|}
a'_2\FD{n}(a'_1,\rada {a'_3}{a'_{n+1}},\rada{a''_1}{a''_{n-1}}a''_{n+1})a''_{n}.
\end{align*}

Notice that the right hand side of~(\ref{Kdy si zase poletam?})
equals $\FD2(a'_1a'_2,a_1''a''_2)$ with $\nabla = [-,-]$.  We say that
$\nabla$ is a {\em bidifferential operator of order $r$\/} if
$\FD {r+1}$ is identically zero. Thus~(\ref{Kdy si zase poletam?})
says that the bracket of an $F$-manifold is a bidifferential
operator of order $1$.

\begin{proposition}
\label{Pujdu koupit olej.}
If a linear map \ $\nabla : A \ot A \to A$ is a  
differential operator of order $r$ in both variables, then it is a bidifferential operator of order $r$.
\end{proposition}

\begin{proof}
For permutations $\sigma,\mu \in \Sigma_n$, elements
$\rada{a'_1}{a'_n},\rada{a''_1}{a''_n}\in A$ 
and $0 \leq i,j \leq n$
we define the linear maps $L_i(\sigma),R_j(\sigma) : A \to A$ by the formulas
\begin{align*}
L_i(\sigma)(a) &:=  \epsilon(\sigma) \cdot a'_{\sigma(1)} \cdots  a'_{\sigma(i)}
\nabla( a'_{\sigma(i+1)} \cdots  a'_{\sigma(n)},a), \ \hbox { and }
\\
R_j(\mu)(a) &:=  \epsilon(\mu) \cdot 
\nabla(a, a''_{\mu(1)} \cdots  a''_{\mu(j)})a''_{\mu(j+1)} 
\cdots  a''_{\mu(n)}.
\end{align*}
Notice that when $\nabla$ is a differential operator of order $r$
in both variables, then both $L_i(\sigma)$ and $R_j(\sigma)$ are
differential operators of order $r$. 

The proposition follows from
the fact that the bideviation $\FD n$ is, for each $n \geq 1$, a 
linear combinations of the
deviations $\Phi^r_{L_i(\sigma)}$ and
$\Phi^r_{R_j(\sigma)}$. Namely, one can verify that
\begin{align*}
2 \FD{n}(\rada {a'_1}{a'_{n}}&,\rada{a''_1}{a''_n})
=
\\
&\sum_{0 \leq i < n}\sign i
\sum_\sigma   \Phi^n_{L_i(\sigma)}(\rada{a''_1}{a''_n}) +
\sum_{0 < j \le n} \sign {j+n}
\sum_\mu   \Phi^n_{R_j(\sigma)}(\rada{a'_1}{a'_n}),
\end{align*}
where $\sigma$ runs over all $(i,r-i)$-shuffles and $\mu$ over all
$(j,r-j)$-shuffles. 
\end{proof}

An immediate consequence of Proposition~\ref{Pujdu koupit olej.} is
the fact 
that the bracket of a~Poisson algebra satisfies~(\ref{Kdy si zase
  poletam?}). One obviously can likewise introduce polydifferential
operators $\nabla : A^{\ot k} \to A$ for arbitrary $k$ and modify
e.g.\ Definition~\ref{Opet jsem selhal.} by requiring that the
structure maps $l_k$'s are polydifferential operators. On the other
hand, polydifferential operators on the free commutative associative
algebra $\S(X)$ are not necessarily determined by their values
on $\S^{\leq n}(X)$ for a finite~$n$, so one cannot expect results as
e.g.\ Corollary~\ref{Zacinaji vedra.}.


\def\cprime{$'$}\def\cprime{$'$}

\end{document}